\documentclass{amsart}

\usepackage{graphicx}
\usepackage{float}
\usepackage{subcaption}
\usepackage{caption}
\usepackage{amsmath,amssymb,amstext}
\usepackage{mathrsfs}
\usepackage{multirow}
\usepackage{algorithm}
\usepackage{algorithmic}
\usepackage[multiple]{footmisc} 
\usepackage{xcolor}
\usepackage[colorlinks=true, linkcolor=blue, citecolor=red, urlcolor=blue]{hyperref}
\usepackage[top=2.5cm, bottom=2.5cm, right=3.5cm, left=3.5cm]{geometry}
\usepackage{booktabs}

\newtheorem{theorem}{Theorem}[section]
\newtheorem{lemma}[theorem]{Lemma}
\theoremstyle{definition}

\newtheorem{assumption}[theorem]{Assumption}
\theoremstyle{remark}

\numberwithin{equation}{section}

\begin{document}
	
	\title[Zeroth-Order Algorithms for Composite Problems]{A Unified Zeroth-Order Proximal Newton-Type Framework for Composite Optimization}
	\thanks{*Corresponding author. This work is supported by the National Natural Science Foundation of China grant 12371307.}

	\author{Zekun Liu}
	\address{School of Mathematical Sciences, Shanghai Jiao Tong University, Shanghai 200240, China.}
	\email{sjtu\_lzk@sjtu.edu.cn}

	\author{Jinyan Fan*}
	\address{School of Mathematical Sciences,  and MOE-LSC, Shanghai Jiao Tong University, Shanghai 200240, China.}
	\email{jyfan@sjtu.edu.cn}
	
	\subjclass[2020]{Primary 65K10, 90C30, 90C56, 90C53}
	
	\keywords{Derivative-free optimization, Composite minimization, Proximal Newton-type method, Local superlinear convergence}
	
	\begin{abstract}
		We propose a unified derivative-free proximal Newton-type algorithm framework for solving composite optimization problems formulated as the sum of a black-box function and a known regularization term. We establish the iteration and oracle complexity bounds for the algorithm to attain an $\epsilon$-optimal solution under both nonconvex and strongly convex settings.
		We also establish its local R-superlinear convergence based on the Dennis--Mor\'{e} condition, and theoretically address an open problem by showing that the BFGS scheme is more compatible with finite-difference gradient estimators than with smoothing-based ones.
		Numerical experiments are further presented to demonstrate the efficiency of the proposed method.
	\end{abstract}

	\maketitle
	
\section{Introduction}\label{sec:Intro}

Composite optimization problems arise frequently in a wide range of applications such as  asset risk management~\cite{ZORO2022}, image restoration~\cite{Globalized2021}, binary classification~\cite{IPZOPM2024}, hyperparameter tuning~\cite{ZOPG2023}, and black-box adversarial attacks~\cite{ZORO2022}. In this paper, we consider the composite optimization problem
\begin{equation}\label{problem}
	\min_{x\in \mathbb{R}^n}  F(x) = f(x) + h(x),
\end{equation}
where $f:\mathbb{R}^n\to \mathbb{R}$ is continuously differentiable, black-box and possibly nonconvex, with only function value access, and $h: \mathbb{R}^n \to \mathbb{R} \cup \{ +\infty \} $ is white-box, convex but not necessarily smooth. In applications, $h$ is often adopted to encode the prior structure of the problem. For instance, the $\ell_1$-norm is commonly used to induce sparsity. Since the gradient $\nabla f$ is either unavailable or computationally expensive, conventional proximal gradient methods are not applicable in the zeroth-order setting.

Derivative-free optimization methods have been studied for decades. Interested readers are referred to the review articles~\cite{DFOReview2019,ZOReview2020} for a brief overview and to seminal monographs~\cite{DFObook2009,DFObook2017} for comprehensive details.
Over the past two decades, numerous methods have been proposed to solve~\eqref{problem} by only using the zeroth-order oracle of $f$ and the proximal oracle of $h$. Among these methods, the most natural and widely adopted ones are the vanilla proximal gradient methods built upon various gradient estimations that rely solely on function evaluations.
For example, Huang et al.~\cite{ZOPSVRG2019} leveraged Gaussian smoothing~\cite{GS2017} and central differences to approximate $\nabla f$, and incorporated variance reduction techniques of SVRG~\cite{SVRG2014,VR2016} and SAGA~\cite{SAGA2014,VR2016} to solve nonconvex nonsmooth finite-sum composite optimization problems.
The spherical smoothing technique~\cite{SS2005} is also leveraged for gradient estimation in proximal gradient methods (cf.~\cite{ZOPG2023,UniZProxSG2020}).
Assuming that $h$ reduces to an indicator function, Balasubramanian and Ghadimi~\cite{GSH2022} investigated nonconvex Lipschitz smooth stochastic problems via Gaussian smoothing, with a particular focus on high-dimensional settings and saddle-point avoidance.  Kungurtsev and Rinaldi~\cite{DSZProxSG2021} proposed a proximal method based on double Gaussian smoothing to tackle the nonconvex and nonsmooth problem~\eqref{problem}.
It was later enhanced by Pougkakiotis and Kalogerias~\cite{ZOPG2023} into a single two-point Gaussian smoothing scheme for a much broader class of composite problems with $f$ weakly convex.
Cai et al.~\cite{ZORO2022} proposed a zeroth-order regularized optimization method that employs an adaptive randomized gradient estimator combined with an inexact proximal-gradient scheme. The method requires only a small number of objective function evaluations when the underlying gradient is approximately sparse.

Unfortunately, most existing methods, including but not limited to those mentioned above, only estimate the gradient and thus suffer from slow convergence. They fail to attain favorable local superlinear convergence (or even faster rates) near the minimizer under strongly convexity. To address this limitation, it is natural to exploit second-order information to attain fast convergence.
In~\cite{GSH2022}, Balasubramanian and Ghadimi proposed a zeroth-order stochastic cubic regularized Newton method, in which the Hessian is approximated via the second-order Gaussian Stein identity.
Though the second-order Gaussian smoothing technique appears to require a smaller number of function evaluations for Hessian estimation, it actually requires at least $\mathcal{O}(n/\epsilon^2)$ function evaluations per iteration to obtain an $\epsilon$ second-order stationary point (cf.~\cite[Theorem 9]{GSH2022}).
In~\cite{ZOCNM2023}, Doikov and Grapiglia developed first- and zeroth-order implementations
of the regularized Newton method with lazy Hessian approximations.
The lazy Hessian updating scheme reuses the previously computed Hessian approximation over multiple iterations and requires $\mathcal{O}(n^2)$ function evaluations via finite differences.
It is shown that their method achieves local R-superlinear convergence under standard assumptions. However, the introduction of the cubic regularization term makes its subproblems computationally expensive to solve.
To exploit a light-weight approximation of the Hessian, Liu et al.~\cite{IPZOPM2024}
approximated the gradient and diagonal Hessian simultaneously via central differences, and proposed a preconditioned zeroth-order proximal gradient method leveraging diagonal second-order information.
Note that these methods either only incorporate partial second-order information, or incur prohibitive function evaluation costs when the dimension $n$ is large.

\paragraph{\textbf{Contributions}}
This paper develops an efficient zeroth-order algorithm for solving the composite optimization problems, and establishes its global convergence, local superlinear convergence, and oracle complexity guarantees. The main contributions are summarized as follows.

First, we propose a unified zeroth-order proximal Newton-type framework for solving~\eqref{problem},
which accommodates a broad class of gradient and Hessian estimation methods.
Rigorous convergence results can be established whenever these estimators satisfy mild and standard conditions.
To the best of our knowledge, this is the first work dedicated to proximal Newton-type algorithms under the full zeroth-order setting.
Most existing studies on proximal Newton methods concentrate on the first-order regime.
Although several existing zeroth-order works incorporate second-order approximate information for composite optimization, they are not formulated within the standard proximal Newton-type framework. Instead, they are predominantly built upon cubic Newton or regularized cubic Newton schemes, and the available literature in this direction remains rather limited.
Further elaborations are provided in Section~\ref{sec:ZOPN}.

Second, we derive an iteration complexity of $\mathcal{O}(1/\epsilon^2)$ and an oracle complexity of $\mathcal{O}(n/\epsilon^2)$ for the proposed algorithm to attain an $\epsilon$-stationary solution of the nonconvex objective $F$.
When $F$ is strongly convex, we prove that the algorithm converges R-linearly to the minimizer in terms of both objective values and iterates. Based on this property, we further establish the  iteration and oracle complexity bound $\mathcal{O}(n\log(1/\epsilon))$ for obtaining an $\epsilon$-optimal solution.
All these theoretical results are presented in Section~\ref{sec:Conv}.

Third, under standard assumptions, we establish the local R-superlinear convergence of the proposed framework. Such favorable superlinear convergence results are rare in derivative-free optimization. Although R-superlinear convergence has been established in several prior works, those results hold for distinct algorithmic structures. For instance, the ZO-CNM algorithm in~\cite{ZOCNM2023}
achieves local R-superlinear convergence via lazy Hessian updates under a regularized cubic Newton framework, where its analysis relies on finite-difference approximation of the Hessian. In contrast, we adopt a different theoretical route based on the Dennis--Mor\'e condition. Furthermore, our convergence result is unified and applicable to a broad family of Hessian estimation strategies.
All these results are shown in Section~\ref{sec:Local}.

Fourth, we partially address an open problem raised by Bollapragada and Wild~\cite{ZOQN2023} and Bollapragada et al. \cite{Bollapragada2024}, namely, which gradient estimation strategy is preferable for derivative-free quasi-Newton methods.
By establishing upper bounds on the number of samples required for gradient approximation,
we theoretically demonstrate for the first time that the BFGS scheme is better suited to deterministic finite-difference gradient estimators than to Gaussian and spherical smoothing counterparts. Existing studies have frequently employed BFGS combined with finite or central difference gradient estimation, yet have rarely provided theoretical justification for preferring  finite-difference schemes over their smoothing-based counterparts.
Our work provides a theoretical interpretation of this empirical phenomenon in Section~\ref{subsec:FDvsGS}.

All numerical results are reported in Section~\ref{sec:experiment}.

\paragraph{\textbf{Notations and assumptions}}
The set of all symmetric positive definite matrices in $\mathbb{R}^{n\times n}$ is denoted by $\mathbb{S}_{++}^n$. We use $\| \cdot \|$ to denote the standard Euclidean norm for vectors or the spectral norm for matrices, and define $\| x\|_A:=\sqrt{x^{\top}Ax}$ as the induced norm associated with $A\in \mathbb{S}_{++}^n$.
We denote the standard inner product $\langle x,y \rangle:=x^{\top}y $ for $x, y \in \mathbb{R}^n$. For $H, G\in \mathbb{R}^{n\times n}$, $H\succ G$ (resp., $H\succeq G$) means $H-G$ is positive definite (resp., positive semi-definite). The distance from a point $x$ to a nonempty set $A$ is defined as $\mathrm{dist}(x,A):=\inf_{y\in A} \| x-y \| $. The subgradient of a convex function $h$ at $x$ is denoted as $\partial h(x)$.   We use $[n]$ to represent the set $\{1,2,\cdots,n\}$ for $n\in\mathbb{N}_+$,  and $\lceil \cdot \rceil$, $\lfloor \cdot \rfloor$ to denote the ceiling and floor operators, respectively.
We write $[ \cdot ]_i $ for the $i$-th entry of a vector, and $[ \cdot ]_{ij} $ for the $(i,j)$-th entry of a matrix.
The big-$\mathcal{O}$ notation $\mathcal{O}(\cdot)$ stands for the upper bound,
the big-$\Theta$ notation $\Theta(\cdot)$ characterizes an order-tight bound,
and the big-$\varOmega$ notation $\varOmega(\cdot)$ corresponds to a lower bound.
$\mathcal{N}( 0,I)$ denotes the standard multivariate normal distribution, and $\mathcal{U}(\mathcal{S}(0,1))$ the multivariate uniform distribution on the unit sphere centered at the origin with radius one. The standard canonical basis is denoted by $\{e_i\}_{i=1}^n$.  Finally, we define the indicator function of a nonempty set $A$ by $\chi_A(x)=0$ if $x\in A$ and $\chi_A(x)=+\infty$ otherwise.

For functions $f, h$ and $F$, we make the following standard assumptions.

\begin{assumption}\label{as:problem1}
	The gradient $\nabla f$ is $L_f$-Lipschitz continuous.
	The regularization term	$h$ is proper, closed, and convex, but not necessarily differentiable.  Furthermore, the objective $F$ is bounded below by $F_{\mathrm{low}}$.
\end{assumption}

\begin{assumption}\label{as:problem2}
	The Hessian $\nabla^2 f$ satisfies $\nabla^2 f \succeq \mu I$ for some $\mu>0$, and is Lipschitz continuous with constant $L_H$.
\end{assumption}

\section{A Unified Zeroth-Order Proximal Newton-Type Framework}\label{sec:ZOPN}
In this section, we propose a unified zeroth-order proximal Newton-type framework for solving problem \eqref{problem}.
We leave the gradient and Hessian approximations in an abstract form to accommodate a variety of zeroth-order approaches.

Denote the sampling radius at the $k$-th iteration by $\Delta_k>0$, and let the corresponding approximate gradient be $g_k:=g_{\Delta_k}(x_k)$.
We impose the following assumptions on the approximate gradient and Hessian.

\begin{assumption}\label{as:ZOPN-g}
	The approximate gradient  $g_k$ satisfies the \textit{fully linear}-like property (cf.~\cite[Section~6.1]{DFObook2009}):
	\begin{equation}\label{eq:gradienterror}
		\|g_k - \nabla f(x_k)\| \le \kappa_{\mathrm{eg}}\Delta_k,
	\end{equation}
	where $\kappa_{\mathrm{eg}}>0$ is a constant independent of $k$.
\end{assumption}

\begin{assumption}\label{as:ZOPN-H}
	The approximate  Hessian $H_k$ satisfies
	\begin{equation}\label{eq:bound}
		\underline{\kappa}I \preceq H_k \preceq \bar{\kappa}I, \quad \forall k\in \mathbb{N},
	\end{equation}
	where $\bar{\kappa}\ge \underline{\kappa}>0$ are constants.
\end{assumption}

\begin{assumption}\label{as:ZOPN-DM}
	The sequence  $\{H_k\}$ satisfies the Dennis--Mor\'e condition (cf.~\cite{Dennis1974DM}):
	\begin{equation}\label{eq:DM}
		\varrho_k :=	\frac{\|(H_k - \nabla^2 f(x^{\star}))(x_{k+1}-x_k)\|}{\|x_{k+1} - x_k\|} \to 0 \quad \text{as } k \to \infty,
	\end{equation}
	where $x^{\star}$ denotes the minimizer of problem~\eqref{problem}.
\end{assumption}

We emphasize that typical deterministic gradient estimators satisfy~\eqref{eq:gradienterror}, including the standard forward difference, simplex gradients, and linear interpolation; see \cite{DFObook2009,GE2022} for more comprehensive discussions.
For example, consider the central difference
\begin{equation*}
	g_k^{\mathrm{CD}}:=\sum_{i=1}^n \frac{f(x_k+\Delta_k e_i)-f(x_k-\Delta_k e_i)}{2\Delta_k}e_i,
\end{equation*}
which admits a second-order error bound (see, e.g.,~\cite[Lemma~4]{IPZOPM2024})
\begin{equation*}
	\|g_k^{\mathrm{CD}} - \nabla f(x_k)\| \le \frac{\sqrt{n}L_H}{6}\Delta_k^2.
\end{equation*}
The bound is naturally consistent with~\eqref{eq:gradienterror} provided that the sampling radius satisfies $\Delta_k\le 1$.
In addition, Assumptions~\ref{as:ZOPN-H} and \ref{as:ZOPN-DM} can be fulfilled by Hessian approximations via zeroth-order approaches, and we will discuss two popular Hessian approximations and their consistency with Assumption~\ref{as:ZOPN-DM} in detail in Section~\ref{sec:Local}.

At the $k$-th iteration, we construct the quadratic model of $f$ at $x_k$:
\begin{equation*}
	\tilde{f}(x) = f(x_k) +  g_k^{\top}   (x-x_k) + \frac{1}{2}(x-x_k)^{\top}H_k (x-x_k).
\end{equation*}
Note that exactly minimizing $\tilde f(x)+h(x)$ can be computationally expensive in many cases.
Instead, we inexactly solve the subproblem
\begin{equation}\label{equ:subproblem}
	\min_{d\in \mathbb{R}^n}  g_k^{\top} d + \frac{1}{2}d^{\top}H_k d + h(x_k + d)
\end{equation}
to obtain the step $d_k$.
Thus there exists a subgradient residual $r_k$ satisfying
\begin{equation}\label{equ:residual}
	r_k \in g_k + H_k d_k+\partial h(x_k+d_k).
\end{equation}
Similar to the inexact condition given in~\cite{Inexact2021},
we accept $d_k$ if it satisfies
\begin{equation}\label{equ:inexactcriterion}
	\|r_k \|_{H_k^{-1}}\le ( 1-\gamma ) \|d_k\|_{H_k},
\end{equation}
where $\gamma \in (0,1]$ is a constant.
Then a line search along $d_k$ is performed such that a sufficient decrease in the objective function value could be achieved.

It follows from Assumption~\ref{as:ZOPN-g} and Young's inequality that
\begin{equation*}
	(\nabla f(x_k) - g_k)^{\top}d_k
	\leq \frac{1}{2\underline{\kappa}\gamma} \| \nabla f(x_k)- g_k \|^2 + \frac{\underline{\kappa}\gamma}2 \| d_k  \|^2 \leq \frac{\underline{\kappa}\gamma}2 \| d_k  \|^2 + \frac{\kappa_{\mathrm{eg}}^2}{2\underline{\kappa}\gamma}\Delta_k^2.
\end{equation*}
So by Assumption~\ref{as:problem1}, for $t\in (0,1]$,
	\begin{align}\label{equ:lemdd1}
		F(x_k+td_k)-F(x_k)
		&=f(x_k+td_k)-f(x_k)+h(t(x_k+d_k)+(1-t)x_k)-h(x_k)  \nonumber \\
		&\leq t (\nabla f(x_k)^{\top}d_k+h(x_k+d_k)-h(x_k) ) +\frac{1}{2}L_f t^2\left\| d_k\right\|^2 \nonumber \\
		&=t (g_k^{\top}d_k+h(x_k+d_k)-h(x_k)  )  +t (\nabla f(x_k)-g_k )^{\top}d_k+\frac{1}{2}L_f t^2 \| d_k \|^2 \nonumber \\
		&\leq t (g_k^{\top}d_k+h(x_k+d_k)-h(x_k)  )
		+\frac{t(\underline{\kappa}\gamma+tL_f)}{2}  \| d_k \|^2 + \frac{t\kappa_{\mathrm{eg}}^2}{2\underline{\kappa}\gamma}\Delta_k^2.
	\end{align}
Define
\begin{equation*}
	\Phi_k:=g_k^{\top}d_k+h(x_k+d_k)-h(x_k).
\end{equation*}
Based on~\eqref{equ:lemdd1}, we define the sufficient decrease condition
\begin{equation}\label{equ:lambda}
	F(x_k+t_kd_k)-F(x_k)\le c_1 t_k \Phi_k + n c_2 \Delta_k^2,
\end{equation}
where $c_1 \in (0,1/2)$ and $c_2>0$ are constants.

The lemma below indicates that~\eqref{equ:lambda} can be satisfied if the step size is smaller than some threshold value.

\begin{lemma}\label{lem:stepsizeuppbd}
	Suppose that Assumptions~\ref{as:problem1}, \ref{as:ZOPN-g}, and \ref{as:ZOPN-H} hold, and that $d_k$ satisfies~\eqref{equ:inexactcriterion}. Then~\eqref{equ:lambda} holds if
	\begin{equation*}
		t_k \le \min \left\lbrace 1, \frac{\underline{\kappa}\gamma (1-2c_1)}{L_f}, \frac{2\underline{\kappa}\gamma nc_2}{\kappa_{\mathrm{eg}}^2} \right\rbrace.
	\end{equation*}
	That is, if the backtracking line search starts with an initial step size $\bar{t}>0$ and is reduced by a factor $\beta \in (0,1)$, then we have
	\begin{equation}\label{equ:tlowerbd}
		\underline{t}:= \beta \min \left\lbrace 1, \frac{\underline{\kappa}\gamma (1-2c_1)}{L_f}, \frac{2\underline{\kappa}\gamma nc_2}{\kappa_{\mathrm{eg}}^2} \right\rbrace \le t_k \le \bar{t}.
	\end{equation}
\end{lemma}

\begin{proof}
	Since $h$ is convex, by~\eqref{equ:residual}, we have
	\begin{equation*}
		h(x_k+d_k)- (r_k - g_k - H_kd_k  )^{\top}d_k \le h(x_k).
	\end{equation*}
	This, together with~\eqref{equ:inexactcriterion}, implies that
		\begin{align}\label{equ:lemdd2}
			\Phi_k &\le r_k^{\top}d_k - d_k^{\top}H_kd_k
			\le \|r_k\|_{H_k^{-1}}\| d_k\|_{H_k} - d_k^{\top}H_kd_k \nonumber \\
			&\le (1-\gamma ) \| d_k \|_{H_k}^2- d_k^{\top}H_kd_k = -\gamma d_k^{\top}H_kd_k.
		\end{align}
	Combining~\eqref{equ:lemdd1}, \eqref{equ:lemdd2}  and Assumption~\ref{as:ZOPN-H}, we obtain
	\begin{equation}\label{equ:lemstepsize1}
		F(x_k+t_k d_k)-F(x_k)
		\le \frac{t_k}2 \left(1 - \frac{t_kL_f}{\underline{\kappa}\gamma}\right) \Phi_k + \frac{t_k \kappa_{\mathrm{eg}}^2}{2\underline{\kappa}\gamma}\Delta_k^2.
	\end{equation}
	
	If $t_k\le \min\{ \underline{\kappa}\gamma (1-2c_1) / L_f,  2\underline{\kappa}\gamma nc_2 / \kappa_{\mathrm{eg}}^2\}$, then we have
	\begin{equation*}
		\frac{t_k}2 \left(1 - \frac{t_kL_f}{\underline{\kappa}\gamma}\right) \ge c_1 t_k, \quad  \frac{t_k \kappa_{\mathrm{eg}}^2}{2\underline{\kappa}\gamma} \le nc_2.
	\end{equation*}
	Substituting this inequality into~\eqref{equ:lemstepsize1}, we obtain
	\begin{equation*}
		F(x_k+t_k d_k)-F(x_k)\le  c_1 t_k \Phi_k + nc_2 \Delta_k^2.
	\end{equation*}
	That is, the step size $t_k$ satisfies~\eqref{equ:lambda}.
	The proof is completed.
\end{proof}

For typical deterministic gradient estimation schemes (e.g., forward difference, central difference, linear interpolation), we have $\kappa_{\mathrm{eg}}=\Theta(\sqrt{n})$ (cf.~\cite{GE2022}). Therefore, $\underline{t}$ in~\eqref{equ:tlowerbd} is a constant independent of the dimension $n$.

In the following, we illustrate the relationship between the step length and the distance to the set of stationary points.

\begin{lemma}\label{lem:station}
	Suppose that Assumptions~\ref{as:problem1}, \ref{as:ZOPN-g}, and \ref{as:ZOPN-H} hold, and that $d_k$ satisfies \eqref{equ:inexactcriterion}.
	Then, we have
	\begin{equation*}
		\mathrm{dist}(0,\partial F(x_k+d_k))\le (L_f + (2-\gamma)\bar{\kappa})\|d_k\| + \kappa_{\mathrm{eg}} \Delta_k.
	\end{equation*}
\end{lemma}

\begin{proof}
	It follows from~\eqref{equ:inexactcriterion} and Assumption~\ref{as:ZOPN-H} that
	\begin{equation*}
		\| r_k \| \le \sqrt{\bar{\kappa}} \|r_k  \|_{H_k^{-1}}\le (1-\gamma )\sqrt{\bar{\kappa}} \| d_k \|_{H_k} \le (1-\gamma )\bar{\kappa}  \| d_k \|.
	\end{equation*}
	Combining \eqref{equ:residual}, Assumptions~\ref{as:problem1}, \ref{as:ZOPN-g} and \ref{as:ZOPN-H}, we obtain
	\begin{equation*}
		\begin{aligned}
			\mathrm{dist}(0,\partial F(x_k+d_k))
			&= \mathrm{dist}(0,\nabla f(x_k+d_k)+\partial h(x_k+d_k))  \\
			& \le  \| \nabla f(x_k+d_k)+r_k-g_k -H_kd_k \|   \\
			&\le  \| \nabla f(x_k+d_k)-\nabla f(x_k) \|  + \| \nabla f(x_k)-g_k \| + \| H_kd_k \| + \| r_k \|   \\
			&\le  (L_f + (2-\gamma)\bar{\kappa}) \|d_k  \| + \kappa_{\mathrm{eg}}\Delta_k.
		\end{aligned}
	\end{equation*}
	The proof is completed.
\end{proof}

Based on the above analysis, we present a unified zeroth-order proximal Newton-type framework for solving~\eqref{problem}. The subproblem~\eqref{equ:subproblem} can be solved inexactly by a variety of first-order methods. We adopt the fast iterative shrinkage-thresholding algorithm (FISTA)  \cite{FISTA2009,Nesterov2013} to solve it, which efficiently yields an inexact solution $d_k$
satisfying~\eqref{equ:inexactcriterion} within a finite number of iterations; see Algorithm~\ref{alg:FISTA} and Theorem~\ref{thm:FISTA} for details.

\begin{algorithm}[H]
	\renewcommand{\algorithmicrequire}{\textbf{Input:}}
	\renewcommand{\algorithmicensure}{\textbf{Output:}}
	\caption{A unified zeroth-order proximal Newton-type framework (ZOPN) for solving~\eqref{problem}}\label{ZOPN}
	\vskip6pt
		\begin{algorithmic}[1]
		\REQUIRE  $x_0 \in \mathrm{dom}F$, $\epsilon\geq 0$, $c_1\in (0 ,1/2)$, $c_2>0$, $\beta\in (0,1)$, $\gamma \in ( 0,1] $, $\bar{t}>0$.  	
		\FOR{$k=0,1,\dots$}
		\STATE  Set the sampling radius $\Delta_k>0$.
		\STATE Compute $g_k$ and $H_k$ that satisfy Assumptions~\ref{as:ZOPN-g}--\ref{as:ZOPN-DM} via zeroth-order oracle. 	
		\STATE Solve \eqref{equ:subproblem} inexactly to obtain $d_k$ that satisfies \eqref{equ:inexactcriterion}.
		\STATE If $\|d_k\|\leq \epsilon$, stop.	
		\STATE Perform the backtracking line search to obtain $t_k$ such that \eqref{equ:lambda} is satisfied.
		\STATE Update $x_{k+1} = x_k + t_k d_k$.
		\ENDFOR
		\ENSURE $x_k+d_k$.
	\end{algorithmic}
\end{algorithm}

\begin{algorithm}[H]
	\renewcommand{\algorithmicrequire}{\textbf{Input:}}
	\renewcommand{\algorithmicensure}{\textbf{Output:}}
	\caption{FISTA for solving subproblem~\eqref{equ:subproblem} at the $k$-th iteration}
	\label{alg:FISTA}
	\vskip6pt
	\begin{algorithmic}[1]
		\REQUIRE $x_k$, $g_k$, $H_k$, $\epsilon$, $\gamma$.
		\STATE Initialize $y_{k,0}=z_{k,0}=x_k$, $\vartheta_0=1$, $\alpha=1/\|H_k\|$.
		\FOR{$l = 1,2,\dots $}
		\STATE Set $\vartheta_l = \frac{1+\sqrt{1+4\vartheta_{l-1}^2}}{2}$.
		\STATE Compute
		\begin{align}
			y_{k,l}=\mathop{\mathrm{argmin}}\limits_{y\in\mathbb{R}^n} \Big\lbrace   \langle g_k + H_k(z_{k,l-1} - x_k), y - z_{k,l-1}  \rangle + \frac{1}{2\alpha}\|y-z_{k,l-1}\|^2 + h(y) \Big\rbrace \label{equ:ykl}
		\end{align}
		and
		\begin{align}\label{equ:zkl}
			z_{k,l} = y_{k,l} + \frac{\vartheta_{l-1}-1}{\vartheta_l}(y_{k,l}-y_{k,l-1}).
		\end{align}
		\STATE If $d_{k,l}=y_{k,l}-x_k$ satisfies \eqref{equ:inexactcriterion} or $\|d_{k,l}\| \le \epsilon$, stop.
		\ENDFOR
		\ENSURE $d_k = d_{k,l}$.
	\end{algorithmic}
\end{algorithm}

\section{Global Convergence Properties of Algorithm~\ref{ZOPN}}\label{sec:Conv}

In this section, we establish the iteration and oracle complexity bounds for Algorithm~\ref{ZOPN} under both nonconvex and strongly convex objectives.

\subsection{Finite termination of Algorithm~\ref{alg:FISTA}}\label{subsec:FISTA}

We first show that Algorithm~\ref{alg:FISTA} is well-defined, by proving that it produces an inexact solution $d_k$ satisfying~\eqref{equ:inexactcriterion} within finitely many inner iterations.

\begin{lemma}\label{thm:FISTA}
	Suppose that Assumptions~\ref{as:problem1} and \ref{as:ZOPN-H} hold.
	Then Algorithm~\ref{alg:FISTA} requires at most
	\begin{equation}\label{equ:inneriter}
		\left\lceil \frac{16\bar{\kappa}^{\frac32}\|x_k-y_k^{\ast}\|}{(1-\gamma)\underline{\kappa}^{\frac32}\epsilon} \right\rceil + 1
	\end{equation}
	iterations to find a solution $d_k$ satisfying~\eqref{equ:inexactcriterion}, where $y_k^{\ast}$ denotes the unique minimizer of~\eqref{equ:FISTAobj}.
\end{lemma}

\begin{proof}
	Denote
	\begin{equation}\label{equ:FISTAobj}
		\psi(y)=g_k^{\top}(y-x_k)+\frac12 \|y-x_k\|_{H_k}^2+h(y).
	\end{equation}
	This is the objective function minimized by Algorithm~\ref{alg:FISTA}.
	Since $H_k\succeq \underline{\kappa}I$, $\psi(y)$ is a strongly convex function. It follows from~\cite[Theorem~4.4]{FISTA2009} that the sequence $\{y_{k,l}\}$ generated by Algorithm~\ref{alg:FISTA} satisfies
	\begin{equation}\label{equ:FISTAlinearconv}
		\psi(y_{k,l}) - \psi(y_k^{\ast}) \le
		\frac{2\bar{\kappa}\|x_k - y_k^{\ast}\|^2}{(l+1)^2}, \quad \forall l\ge1,
	\end{equation}
	where $y_k^{\ast}$ denotes the unique minimizer of $\psi(y)$.
	
	By the optimal condition of~\eqref{equ:ykl}, we have
	\begin{equation*}
		0\in g_k +  H_k(z_{k,l-1} - x_k) + \frac{1}{\alpha}(y_{k,l}-z_{k,l-1})+\partial h(y_{k,l}),
	\end{equation*}
	which implies that there exists some $s_{k,l}\in \partial h(y_{k,l})$ satisfying
	\begin{equation*}
		s_{k,l} = \frac{1}{\alpha}(z_{k,l-1}-y_{k,l})- g_k -  H_k(z_{k,l-1} - x_k).
	\end{equation*}
	Therefore, there exists some  $r_{k,l}\in \partial \psi(y_{k,l})$, such that
	\begin{equation}\label{equ:rkl}
		r_{k,l} = g_k +   H_k(y_{k,l} - x_k) + s_{k,l} =  H_k(y_{k,l} - z_{k,l-1}) + \frac{1}{\alpha}(z_{k,l-1}-y_{k,l}).
	\end{equation}
	Since $\nabla^2 \psi(y)\succeq H_k \succeq  \underline{\kappa} I$, we have
	\begin{equation*}
		\psi(y_{k,l}) - \psi(y_k^{\ast}) \ge \frac{\underline{\kappa}}2  \|y_{k,l}-y_k^{\ast}\|^2.
	\end{equation*}
	It then follows from~\eqref{equ:FISTAlinearconv} that
	\begin{equation}\label{equ:yklbound}
		\|y_{k,l}-y_k^{\ast}\| \le \sqrt{\frac{2}{\underline{\kappa}}(\psi(y_{k,l}) - \psi(y_k^{\ast}))} \le \frac{2\sqrt{\bar{\kappa}}\|x_k - y_k^{\ast}\|}{\sqrt{\underline{\kappa}}(l+1)}.
	\end{equation}
	
	By the update rule of $z_{k,l}$~\eqref{equ:zkl}, it holds that
	\begin{equation} \label{equ:zklrelation}
		\|z_{k,l}-y_k^{\ast}\| \le \|z_{k,l}-y_{k,l}\| + \|y_{k,l}-y_k^{\ast}\|  \le \left|\frac{\vartheta_{l-1}-1}{\vartheta_l} \right| \|y_{k,l}-y_{k,l-1}\| + \|y_{k,l}-y_k^{\ast}\|.
	\end{equation}
	Since $\vartheta_0=1$ and $\vartheta_l = \frac{1+\sqrt{1+4\vartheta_{l-1}^2}}{2} \ge \vartheta_{l-1}$,
	we get $\left|\frac{\vartheta_{l-1}-1}{\vartheta_l} \right| \le 1$.
	It follows from~\eqref{equ:yklbound} that
	\begin{equation}
		\|y_{k,l}-y_{k,l-1}\|  \le \|y_{k,l}-y_k^{\ast}\| + \|y_{k,l-1}-y_k^{\ast}\| \le \frac{4\sqrt{\bar{\kappa}}\|x_k - y_k^{\ast}\|}{\sqrt{\underline{\kappa}}l}. \label{equ:yklrecursion}
	\end{equation}
	Substituting~\eqref{equ:yklbound} and \eqref{equ:yklrecursion} into \eqref{equ:zklrelation}, we obtain
	\begin{equation*}
		\|z_{k,l}-y_k^{\ast}\| \le \frac{6\sqrt{\bar{\kappa}}\|x_k - y_k^{\ast}\|}{\sqrt{\underline{\kappa}}l}.
	\end{equation*}
	This, together with \eqref{equ:rkl} and \eqref{equ:yklbound}, yields
	\begin{equation}
		\|r_{k,l}\|  \le \left(\|H_k\| + \frac{1}{\alpha}\right)  \|y_{k,l} - z_{k,l-1}\| \le 2 \bar{\kappa}(\|y_{k,l} - y_k^{\ast}\|+\|z_{k,l-1}-y_k^{\ast}\|) \le  \frac{16\bar{\kappa}^{\frac32}\|x_k - y_k^{\ast}\|}{\sqrt{\underline{\kappa}}(l-1)}. \label{eq:residualconvrate}
	\end{equation}
	
	Therefore, if there exists some $l$ such that \eqref{equ:inexactcriterion} is satisfied or $\|d_{k,l} \| \le \epsilon$, then Algorithm~\ref{alg:FISTA} stops and returns to Algorithm~\ref{ZOPN} with $d_k=d_{k,l}$. Otherwise, it holds that
	\begin{equation*}
		(1-\gamma) \| d_{k,l} \| _{H_k}\ge (1-\gamma)\sqrt{\underline{\kappa}}  \| d_{k,l} \| > (1-\gamma)\sqrt{\underline{\kappa}}\epsilon.
	\end{equation*}
	Equation~\eqref{eq:residualconvrate} implies that the left-hand side of \eqref{equ:inexactcriterion} satisfies
	\begin{equation*}
		\|r_{k,l}\|_{H_k^{-1}}\le \frac{1}{\sqrt{\underline{\kappa}}}\|r_{k,l}\| \le   \frac{16\bar{\kappa}^{\frac32}\|x_k - y_k^{\ast}\|}{\underline{\kappa}(l-1)}.
	\end{equation*}
	Hence, the inexact criterion \eqref{equ:inexactcriterion} is triggered when
	\begin{equation*}
		\frac{16\bar{\kappa}^{\frac32}\|x_k - y_k^{\ast}\|}{\underline{\kappa}(l-1)} \le (1-\gamma)\sqrt{\underline{\kappa}}\epsilon,
	\end{equation*}
	which yields the worst-case inner iteration complexity bound \eqref{equ:inneriter}.
\end{proof}

The proof of Lemma \ref{thm:FISTA} further implies that we can obtain $r_k$ via \eqref{equ:rkl}, rather than computing the complicated subdifferential through \eqref{equ:residual}.

\subsection{Iteration and oracle complexities of Algorithm~\ref{ZOPN} for nonconvex objectives}
\label{subsec:globalconv}

We first derive both the iteration and oracle complexities of Algorithm~\ref{ZOPN}
without imposing any convexity assumptions on $f$.

\begin{lemma}\label{thm:globalconv}
	Suppose that Assumptions \ref{as:problem1}, \ref{as:ZOPN-g}, and \ref{as:ZOPN-H} hold.
	Then, for any $T\in\mathbb{N}_+$, the sequence $\{d_k\}$ generated by  Algorithm \ref{ZOPN} satisfies
	\begin{equation}\label{dk}
		\sum_{k=0}^{T-1}\|d_k\|^2 \le \frac{F(x_0)-F_{\mathrm{low}}}{ c_1\underline{t}\underline{\kappa}\gamma} + \frac{nc_2}{c_1\underline{t}\underline{\kappa}\gamma}\sum_{k=0}^{T-1}\Delta_k^2.
	\end{equation}
\end{lemma}

\begin{proof}
	By \eqref{equ:lambda}, \eqref{equ:tlowerbd}, and \eqref{equ:lemdd2},
	\begin{equation*}
		F(x_{k+1})-F(x_k)\le c_1 t_k \Phi_k + nc_2 \Delta_k^2 \le -c_1 \underline{t}\underline{\kappa} \gamma \|d_k\|^2 + nc_2 \Delta_k^2.
	\end{equation*}
	Since $F$ is bounded from below, we have
	\begin{equation*}
		\sum_{k=0}^{T-1}c_1\underline{t} \underline{\kappa} \gamma \|d_k \|^2 \leq F(x_0)-F_{\mathrm{low}} + nc_2  \sum_{k=0}^{T-1}\Delta_k^2.
	\end{equation*}
	This implies that \eqref{dk} holds true.
\end{proof}

Lemma~\ref{thm:globalconv} states that if the sampling radius $\{\Delta_k\}$ is square-summable, then $\{\|d_k\|\}$ is also square-summable, which further yields $\|d_k\|\to 0$ as $k\to \infty$. Combining this with Lemma~\ref{lem:station} and the closedness of $\partial F$, we conclude that any accumulation point of $\{x_k\}$ is a stationary point of \eqref{problem}
if the level set $\{x\in \mathbb R^n\mid F(x)\leq F(x_0)+nc_2\sum_{k=0}^{\infty}\Delta_k^2 \}$ is bounded.
Furthermore, if $F$ is strongly convex, then $\{x_k\}$ converges to the unique minimizer of~\eqref{problem}.

We further derive the worst-case iteration complexity and oracle complexity of Algorithm~\ref{ZOPN}  from Lemma~\ref{thm:globalconv}.
Before proceeding, we introduce the notion of \textit{average number of function evaluations per iteration} to enable a fair comparison when the function evaluation cost is distributed unevenly across iterations. Formally, let $C_T$ denote the total number of function evaluations accumulated over the first $T$ iterations. Assuming the limit $\lim_{T\to\infty} C_T/T$ exists, the average number of function evaluations  per iteration is defined as this limit. For instance, consider the lazy Hessian update strategy in \eqref{eq:lazyH}. The scheme updates the approximate Hessian via finite differences every $n$ iterations, yielding an average number of function evaluations per iteration equal to $\frac{(n+1)(n+2)}{2n}$.

\begin{theorem}\label{cor:globalcomplexity}
	Suppose that Assumptions \ref{as:problem1}, \ref{as:ZOPN-g}, and \ref{as:ZOPN-H} hold.
	Then, for any $\epsilon>0$,
	
	(i) if we set
	\begin{equation}\label{equ:ncmuconst}
		\Delta_k \equiv \Delta \le \frac{\sqrt{c_1\underline{t} \underline{\kappa}\gamma}\epsilon}{2[(L_f+(2-\gamma)\bar{\kappa})\sqrt{nc_2}+\sqrt{c_1\underline{t} \underline{\kappa}\gamma}\kappa_{\mathrm{eg}}]},
	\end{equation}
	then Algorithm \ref{ZOPN} finds an $\epsilon$-stationary point of \eqref{problem} within at most
	\begin{equation*}
		\mathcal{K}_1(\epsilon) := \left \lceil \frac{ 4(F(x_0)-F_{\mathrm{low}}  )  (L_f+ (2-\gamma )\bar{\kappa} )^2 }{c_1\underline{t} \underline{\kappa}\gamma\epsilon^2} \right \rceil
	\end{equation*}
	iterations;
	
	(ii) if $\Delta_k$ satisfies
	\begin{equation}\label{equ:ncmudecay}
		\sum_{k=0}^{\infty}\Delta_k^2 = \frac{S_1}n < \infty \ \  \text{ and } \ \ \Delta_k \le \frac{\epsilon}{2\kappa_{\mathrm{eg}}},\quad\forall k\in\mathbb{N},
	\end{equation}
	then Algorithm \ref{ZOPN} finds an $\epsilon$-stationary point of \eqref{problem} within at most
	\begin{equation*}
		\mathcal{K}_2(\epsilon) := \left \lceil  \frac{4(F(x_0)-F_{\mathrm{low}}+ c_2 S_1)(L_f+(2-\gamma)\bar{\kappa})^2}{c_1\underline{t}\underline{\kappa}\gamma \epsilon^2} \right \rceil
	\end{equation*}
	iterations;
	
	(iii) if the number of function evaluations required for gradient estimation satisfies $\varsigma_g=\mathcal{O}(n)$, and the average number of function evaluations per iteration for Hessian approximation satisfies $\varsigma_H = \mathcal{O}(n)$, then the worst-case oracle complexity of Algorithm~\ref{ZOPN} for finding an $\epsilon$-stationary point of \eqref{problem} is  $\mathcal{O}(n/\epsilon^2)$.
\end{theorem}

\begin{proof}
	By \eqref{dk},
	\begin{equation*}
		\min_{0\le k\le \mathcal{K}(\epsilon)-1} \|d_k \|^2  \le \frac1{\mathcal{K}(\epsilon)} \sum_{k=0}^{\mathcal{K}(\epsilon)-1} \|d_k  \|^2   \le \frac1{\mathcal{K}(\epsilon)} \left( \frac{F(x_0)-F_{\mathrm{low}}}{ c_1\underline{t}\underline{\kappa}\gamma} + \frac{nc_2}{c_1\underline{t}\underline{\kappa}\gamma}\sum_{k=0}^{\mathcal{K}(\epsilon)-1}\Delta_k^2 \right) .
	\end{equation*}
	
	(i) If we set $\Delta_k$ as \eqref{equ:ncmuconst}, then there exists $0\le \bar k\le \mathcal{K}_1(\epsilon)-1$ such that
	\begin{equation*}
		\|d_{\bar k} \| \leq \sqrt{\frac1{\mathcal{K}_1(\epsilon)}  \frac{F(x_0)-F_{\mathrm{low}}}{ c_1\underline{t}\underline{\kappa}\gamma}} + \sqrt{\frac{nc_2}{c_1\underline{t}\underline{\kappa}\gamma}}\Delta \le  \frac{\epsilon}{2(L_f+(2-\gamma)\bar{\kappa} )} +  \sqrt{\frac{nc_2}{c_1\underline{t}\underline{\kappa}\gamma}}\Delta.
	\end{equation*}
	From Lemma \ref{lem:station}, we derive
	\begin{equation*}
		\begin{aligned}
			\mathrm{dist}\left(0,\partial F(x_{\bar k}+d_{\bar k}) \right) &\le (L_f+(2-\gamma)\bar{\kappa}) \|d_{\bar k} \| + \kappa_{\mathrm{eg}} \Delta \\
			& \le \frac{\epsilon}2 + \left( (L_f+(2-\gamma)\bar{\kappa})\sqrt{\frac{nc_2}{c_1\underline{t}\underline{\kappa}\gamma}} + \kappa_{\mathrm{eg}} \right) \Delta \le \epsilon,
		\end{aligned}
	\end{equation*}
	which implies that Algorithm \ref{ZOPN} finds an $\epsilon$-stationary point of \eqref{problem} within at most $\mathcal{K}_1(\epsilon)$ iterations.
	
	(ii) If $\sum_{k=0}^{\infty} \Delta_k^2=S_1/n< \infty$, then there exists $0\le \bar k\le \mathcal{K}_2(\epsilon)-1$ such that
	\begin{equation*}
		\|d_{\bar k} \| \leq \sqrt{\frac1{\mathcal{K}_2(\epsilon)}   \frac{F(x_0)-F_{\mathrm{low}}+c_2S_1}{ c_1\underline{t}\underline{\kappa}\gamma} } \le \frac{\epsilon}{2(L_f+(2-\gamma)\bar{\kappa})}.
	\end{equation*}
	This, together with Lemma \ref{lem:station} and $\Delta_k\le \frac{\epsilon}{2\kappa_{\mathrm{eg}}}$, implies that
	\begin{equation*}
		\mathrm{dist}\left(0,\partial F(x_{\bar k}+d_{\bar k}) \right) \le (L_f+(2-\gamma)\bar{\kappa}) \|d_{\bar k} \| + \kappa_{\mathrm{eg}} \Delta_k  \le
		\epsilon.
	\end{equation*}
	
	(iii) Note that in both cases (i) and (ii),
	the number of function evaluations required for gradient and Hessian approximations via the zeroth-order oracle at each iteration is $\varsigma_g+\varsigma_H$,
	and each backtracking step requires one additional function evaluation.
	Suppose that the line search performs $i$ backtracking steps at a given iteration.
	It follows from \eqref{equ:tlowerbd} that $\beta^i \bar{t} \ge \underline{t}$.
	Thus the upper bound on the number of function evaluations for the line search per iteration is		$ \bar i\le  \lfloor \log_{\beta} ( \underline{t}/{\bar t} )  \rfloor $.
	Therefore the total number of function evaluations required for Algorithm \ref{ZOPN} to produce an $\epsilon$-stationary point is at most $\mathcal{O}(n/\epsilon^2) $.
	The proof is completed.
\end{proof}
	
\subsection{Iteration and oracle complexities of Algorithm~\ref{ZOPN} for strongly convex objectives}\label{subsec:linear}
We first show that Algorithm \ref{ZOPN} converges R-linearly to the minimizer of \eqref{as:problem1} in terms of both objective values and iterates under strong convexity.
Furthermore, we establish its iteration and oracle complexities.

We now state the following assumption,  which is weaker than Assumption~\ref{as:problem2}.

\begin{assumption}\label{ass:3}
	$f$ is convex  and $F$ is  strongly convex.
\end{assumption}

The strong convexity of $F$ implies that
there exists $m>0$ such that for all $x, y\in \mathbb{R}^n $,
\begin{equation*}
	F(x)\ge F(y) + q^{\top}(x-y) + \frac m2  \| x-y  \|^2, \quad \forall q \in \partial F(y).
\end{equation*}
Moreover,  \eqref{problem} has a unique minimizer $x^{\star}$ and
$0\in \partial F(x^{\star})$.  Thus
\begin{equation}\label{equ:SCdis}
	F(x)-F(x^{\star})\ge \frac{m}{2} \|x-x^{\star}  \|^2, \quad \forall x\in \mathbb{R}^n.
\end{equation}

\begin{lemma}\label{lem:squarebd}
	Suppose that Assumptions \ref{as:problem1}, \ref{as:ZOPN-g}, \ref{as:ZOPN-H}, and \ref{ass:3} hold.
	Then Algorithm \ref{ZOPN} satisfies
	\begin{equation}\label{equ:squarebound}
		\|x_k+d_k-x^{\star}  \|^2 \le \frac{2}{m}\left(F(x_k)-F(x^{\star}) + \left(\frac{p+L_f}2 - \underline{\kappa}\gamma \right)\|d_k\|^2 + \frac{\kappa_{\mathrm{eg}}^2}{2p} \Delta_k^2  \right)
	\end{equation}
	for any $p>0$.
\end{lemma}

\begin{proof}
	By \eqref{equ:lemdd2},
	\begin{equation*}\label{h}
		h(x_k+d_k)\le h(x_k)-\gamma d_k^{\top}H_kd_k-g_k^{\top}d_k.
	\end{equation*}
	Since $\nabla f$ is $L_f$-Lipschitz continuous, we have
	\begin{equation*}\label{f}
		f(x_k+d_k)\le f(x_k)+\nabla f(x_k)^{\top}d_k+\frac{L_f}{2} \|d_k  \|^2.
	\end{equation*}
	Summing the above two inequalities yields
	\begin{equation*}
		\begin{aligned}
			F(x_k+d_k)&\le F(x_k)-\gamma d_k^{\top}H_kd_k+ (\nabla f(x_k)-g_k )^{\top}d_k +\frac{L_f}{2} \|d_k  \|^2\\
			&\le F(x_k)-\gamma d_k^{\top}H_kd_k + \frac{1}{2p}  \|\nabla f(x_k)-g_k \|^2 + \frac{p+L_f}2 \|d_k\|^2 \\
			&\le  F(x_k) + \left(\frac{p+L_f}2 - \underline{\kappa}\gamma \right)\|d_k\|^2 + \frac{\kappa_{\mathrm{eg}}^2}{2p} \Delta_k^2,
		\end{aligned}
	\end{equation*}
	where the second inequality follows from Young's inequality for all $p>0$, and the last inequality is due to $H_k\succeq \underline{\kappa}I$ and Assumption~\ref{as:ZOPN-g}.
	It then follows from \eqref{equ:SCdis} that \eqref{equ:squarebound} holds.
\end{proof}

Based on Lemma \ref{lem:squarebd}, we can prove that Algorithm \ref{ZOPN} converges R-linearly to  the minimizer $x^{\star}$ in terms of both objective function values and iterates.

\begin{lemma}\label{thm:linearconv}
	Suppose that Assumptions \ref{as:problem1}, \ref{as:ZOPN-g}, \ref{as:ZOPN-H}, and \ref{ass:3} hold.
	Then, for any $T\in\mathbb{N}_+$, Algorithm~\ref{ZOPN}
	satisfies
	\begin{equation}\label{equ:linearconv_F}
		F(x_T)-F(x^{\star}) \le \tilde{\tau}^{T}(F(x_0)-F(x^{\star}))+ (nc_2+\tilde{\varpi}\kappa_{\mathrm{eg}}^2)\tilde{\tau}^{T-1}\sum_{k=0}^{T-1}\frac{\Delta_k^2}{\tilde{\tau}^k},
	\end{equation}
	where $\tilde{\tau}$, $\tilde{\varpi}$ are constants defined by
	\begin{equation*}
		\tilde{\tau}  := 1 - \frac{c_1\underline{t}\underline{\kappa}\gamma }{4\underline{\kappa}\gamma + \frac{16\bar{\kappa}^2}m + 2L_f} \in (0,1), \text{ and }
		\tilde{\varpi}  := \frac{c_1\underline{t}\underline{\kappa}\gamma ( 1/L_f + 8/m)}{4\underline{\kappa}\gamma + \frac{16\bar{\kappa}^2}m + 2L_f} > 0.
	\end{equation*}
\end{lemma}

\begin{proof}
	Denote $E_k:=F(x_k)-F(x^{\star})$.
	It follows from \eqref{equ:residual}, \eqref{equ:inexactcriterion}, \eqref{equ:squarebound} and the convexity of $h$ that
		\begin{align}\label{equ:thmlinear1}
			&\quad h(x_k+d_k)-h(x^{\star})+g_k^{\top}(x_k+d_k-x^{\star})  \le  (r_k-H_kd_k)^{\top}(x_k+d_k-x^{\star})  \nonumber \\
			& \le   \|r_k-H_kd_k  \|_{H_k^{-1}} \|x_k+d_k-x^{\star}  \|_{H_k}  \le   \sqrt{\bar{\kappa}}\left( \|r_k \|_{H_k^{-1}}+ \|H_kd_k  \|_{H_k^{-1}} \right)  \|x_k+d_k-x^{\star}  \| \nonumber \\
			& \le   2\sqrt{\bar{\kappa}} \| d_k \|_{H_k} \|x_k+d_k-x^{\star}  \|  \le   \frac{4\bar{\kappa}}{m} \| d_k \|_{H_k}^2+\frac{m}{4} \|x_k+d_k-x^{\star}  \|^2 \nonumber \\
			& \le   \frac{4\bar{\kappa}}{m} \| d_k \|_{H_k}^2+\frac12 \left(E_k+ \left(\frac{p+L_f}2 - \underline{\kappa}\gamma \right)\|d_k\|^2 + \frac{\kappa_{\mathrm{eg}}^2}{2p} \Delta_k^2  \right)  \nonumber \\
			& \le   \left(\frac{4\bar{\kappa}^2}{m}+ \frac{p+L_f}4\right) \| d_k \|^2 + \frac12 E_k + \frac{\kappa_{\mathrm{eg}}^2}{4p} \Delta_k^2.
		\end{align}
	
	By \eqref{equ:lambda} and \eqref{equ:tlowerbd},
	\begin{equation}\label{equ:thmlinear2}
		-(g_k^{\top}d_k+h(x_k+d_k)-h(x_k) )=-\Phi_k \le \frac{E_k - E_{k+1} + nc_2\Delta_k^2}{c_1\underline{t}}.
	\end{equation}
	This, together with \eqref{equ:lemdd2}, yields
	\begin{equation}\label{equ:thmlinear3}
		\| d_k \|^2
		\le \frac{E_k - E_{k+1} + nc_2\Delta_k^2}{c_1\underline{t}\underline{\kappa}\gamma}.
	\end{equation}
	Hence, by \eqref{equ:SCdis}, \eqref{equ:thmlinear1}--\eqref{equ:thmlinear3} and Assumption~\ref{as:ZOPN-g},
		\begin{align}\label{Ek}
			E_k &= f(x_k)+h(x_k)-f(x^{\star})-h(x^{\star})  \nonumber \\
			&= f(x_k)-f(x^{\star})+\nabla f(x_k)^{\top} (x^{\star}-x_k )  + h(x_k)-h(x_k+d_k)-g_k^{\top}d_k \nonumber \\
			& \hspace{1.2em} + h(x_k+d_k)-h(x^{\star})+g_k^{\top} (x_k+d_k-x^{\star}  )  + (\nabla f(x_k)-g_k )^{\top}  (x_k-x^{\star}  )  \nonumber \\
			& \le \frac{1}{c_1 \underline{t}} (E_k-E_{k+1}  ) + \frac{nc_2}{c_1\underline{t}}\Delta_k^2 +  \left(\frac{4\bar{\kappa}^2}{m}+\frac{p+L_f}4 \right)\frac{1}{c_1  \underline{t}\underline{\kappa}\gamma} (E_k-E_{k+1} ) \nonumber \\
			& \hspace{1.2em} +  \left(\frac{4\bar{\kappa}^2}{m}+\frac{p+L_f}4 \right)\frac{nc_2}{c_1  \underline{t}\underline{\kappa}\gamma} \Delta_k^2 +\frac12 E_k + \frac{\kappa_{\mathrm{eg}}^2}{4p}\Delta_k^2  \nonumber\\
			& \hspace{1.2em} + \frac{1}{2q} \|\nabla f(x_k)- g_k \|^2 + \frac{q}{2} \|x_k-x^{\star} \|^2 \nonumber\\
			& \le \frac{1}{c_1 \underline{t}}\left(1+\frac{1}{\underline{\kappa}\gamma }\left(\frac{4\bar{\kappa}^2}{m}+\frac{p+L_f}4 \right) \right) (E_k-E_{k+1}  ) + \left( \frac12 + \frac{q}{m} \right) E_k  \nonumber\\
			& \hspace{1.2em} + \frac{nc_2}{c_1 \underline{t}}\left(1+\frac{1}{\underline{\kappa}\gamma }\left(\frac{4\bar{\kappa}^2}{m}+\frac{p+L_f}4 \right) \right) \Delta_k^2 + \left( \frac{1}{4p} + \frac{1}{2q} \right) \kappa_{\mathrm{eg}}^2 \Delta_k^2
		\end{align}
	holds for some $p,q>0$ to be determined.
	
	\eqref{Ek} can be rewritten as
	\begin{equation*}
		E_{k+1}\le \tau(p,q) E_k + (nc_2 + \varpi(p,q)\kappa_{\mathrm{eg}}^2)\Delta_k^2,
	\end{equation*}
	where
	\begin{equation*}
		\tau(p,q) := 1 - \frac{c_1\underline{t}\underline{\kappa}\gamma \left(\frac12 - \frac{q}m \right)}{\underline{\kappa}\gamma + \frac{4\bar{\kappa}^2}m + \frac{p+L_f}4} \quad \text{ and } \quad
		\varpi(p,q) :=  \frac{c_1\underline{t}\underline{\kappa}\gamma  \left( \frac{1}{4p} + \frac{1}{2q} \right)}{\underline{\kappa}\gamma + \frac{4\bar{\kappa}^2}m + \frac{p+L_f}4}.
	\end{equation*}
	It is difficult to simultaneously minimize $\tau(p,q)$ and $\varpi(p,q)$. We simply set $p=L_f$ and $q=m/4$, which suffices for our subsequent analysis. In such case, we have
	\begin{equation*}
		\begin{aligned}
			\tilde{\tau} &= \tau\left(L_f, \frac{m}4\right) = 1 - \frac{c_1\underline{t}\underline{\kappa}\gamma }{4\underline{\kappa}\gamma + \frac{16\bar{\kappa}^2}m + 2L_f} \in (0,1), \\
			\tilde{\varpi} &= \varpi\left( L_f, \frac{m}4\right) =   \frac{c_1\underline{t}\underline{\kappa}\gamma ( 1/L_f + 8/m)}{4\underline{\kappa}\gamma + \frac{16\bar{\kappa}^2}m + 2L_f} >0.
		\end{aligned}
	\end{equation*}
	Then \eqref{equ:linearconv_F} holds true by applying $E_{k+1}\le \tilde{\tau}E_k + (nc_2+\tilde{\varpi}\kappa_{\mathrm{eg}}^2)\Delta_k^2$ recursively.
\end{proof}

If the sampling radius decays to zero with rate $\sum_{k=0}^{\infty}\Delta_k^2\tilde{\tau}^{-k}=S_2/n<\infty$, then
\begin{equation}\label{eq:Flinear}
	F(x_T) - F(x^{\star}) \le \tilde{\tau}^T \left( F(x_0)-F(x^{\star}) + \left( c_2 + \frac{\tilde{\varpi} \kappa_{\mathrm{eg}}^2}n \right)\frac{S_2}{\tilde{\tau}} \right).
\end{equation}
Therefore, it follows from \eqref{equ:SCdis} that
\begin{equation}\label{eq:xlinear}
	\|x_T - x^{\star} \| \le \tilde{\tau}^{T/2} \sqrt{\frac2m \left( F(x_0)-F(x^{\star}) + \left( c_2 + \frac{\tilde{\varpi} \kappa_{\mathrm{eg}}^2}n \right)\frac{S_2}{\tilde{\tau}} \right)},
\end{equation}
implying that the sequence $\{\|x_k - x^{\star}\|\}$ is summable.
Moreover, \eqref{eq:Flinear} and \eqref{eq:xlinear} indicate that Algorithm \ref{ZOPN} converges R-linearly to the minimizer in terms of objective values and iterates, respectively, provide that the sampling radius is exponentially weighted square-summable.

The iteration and oracle complexities of Algorithm \ref{ZOPN} can be derived from Lemma \ref{thm:linearconv}.

\begin{theorem}\label{cor:linearNF}
	Suppose that Assumptions \ref{as:problem1}, \ref{as:ZOPN-g}, \ref{as:ZOPN-H}, and \ref{ass:3} hold.
	Then, for any $\epsilon>0$,
	
	(i) if we set
	\begin{equation}\label{equ:scmuconst}
		\Delta_k \equiv \Delta \le \sqrt{\frac{(1-\tilde{\tau})\epsilon}{2(nc_2+\tilde{\varpi}\kappa_{\mathrm{eg}}^2)}},
	\end{equation}
	then the number of iterations required for
	Algorithm~\ref{ZOPN} to obtain an $\epsilon$-optimal solution of \eqref{problem} in terms of objective values (i.e., $F(x_k)-F(x^{\star})\le \epsilon$) is at most
	\begin{equation*}
		\tilde{\mathcal{K}}_1(\epsilon) = \left \lceil  \mathrm{log}_{\tilde \tau}\left( \frac{\epsilon}{2(F(x_0)-F(x^{\star}))} \right) \right \rceil;
	\end{equation*}
	
	(ii) if $\Delta_k$ satisfies
	\begin{equation}\label{equ:scmudecay}
		\sum_{k=0}^{\infty} \frac{\Delta_k^2}{\tilde{\tau}^k} = \frac{S_2}n <  \infty,
	\end{equation}
	then Algorithm \ref{ZOPN} obtains an $\epsilon$-optimal solution of \eqref{problem} in terms of objective values   within at most
	\begin{equation*}
		\tilde{\mathcal{K}}_2(\epsilon) = \left \lceil  \mathrm{log}_{\tilde \tau}\left( \frac{\epsilon}{F(x_0)-F(x^{\star}) +  (c_2 + \tilde{\varpi} \kappa_{\mathrm{eg}}^2/n)S_2 / \tilde{\tau}}  \right) \right \rceil
	\end{equation*}
	iterations;
	
	(iii)
	if the number of function evaluations required for gradient estimation satisfies $\varsigma_g=\mathcal{O}(n)$, and the average number of function evaluations per iteration required for Hessian approximation satisfies $\varsigma_H = \mathcal{O}(n)$, then Algorithm~\ref{ZOPN} requires at most $\mathcal{O}( n\mathrm{log}(1/\epsilon ))$  function evaluations  to produce an  $\epsilon$-optimal solution of \eqref{problem}.
\end{theorem}

\begin{proof}
	(i) If $\Delta_k$ is set according to \eqref{equ:scmuconst}, then \eqref{equ:linearconv_F} immediately implies  that
	\begin{equation*}
		\begin{aligned}
			F\left( x_{\tilde{\mathcal{K}}_1(\epsilon)} \right) - F(x^{\star}) & \le \tilde{\tau}^{\tilde{\mathcal{K}}_1(\epsilon)}(F(x_0)-F(x^{\star})) + (nc_2+\tilde{\varpi}\kappa_{\mathrm{eg}}^2)\Delta^2 \tilde{\tau}^{\tilde{\mathcal{K}}_1(\epsilon)-1} \sum_{k=0}^{\tilde{\mathcal{K}}_1(\epsilon)-1}\frac{1}{\tilde{\tau}^k} \\
			&\le \tilde{\tau}^{\tilde{\mathcal{K}}_1(\epsilon)}(F(x_0)-F(x^{\star})) + \frac{nc_2+\tilde{\varpi}\kappa_{\mathrm{eg}}^2}{1-\tilde{\tau}}\Delta^2 \le \epsilon.
		\end{aligned}
	\end{equation*}
	
	(ii) If $\Delta_k$ satisfies \eqref{equ:scmudecay}, then \eqref{equ:linearconv_F} implies that
	\begin{equation*}
		F\left( x_{\tilde{\mathcal{K}}_2(\epsilon)} \right) - F(x^{\star})  \le \tilde{\tau}^{\tilde{\mathcal{K}}_2(\epsilon)} \left( F(x_0)-F(x^{\star}) + \left( c_2 + \frac{\tilde{\varpi} \kappa_{\mathrm{eg}}^2}n \right)\frac{S_2}{\tilde{\tau}} \right)  \le \epsilon.
	\end{equation*}	
	
	(iii) In both cases (i) and (ii),
	following the same argument as in the proof of Theorem \ref{cor:globalcomplexity}, we derive that the worst-case oracle complexity of Algorithm~\ref{ZOPN} is $  \mathcal{O} ( n\mathrm{log} (1/\epsilon ))$.
\end{proof}

\section{Local Superlinear Convergence of Algorithm~\ref{ZOPN}}\label{sec:Local}

Equipped with the global convergence results established in Section~\ref{sec:Conv}, we now turn to investigate the local convergence rate of Algorithm~\ref{ZOPN}.

\subsection{Local R-superlinear convergence analysis.}\label{subsec:local}

Throughout this section, we assume that the initial step length in the line search scheme \eqref{equ:lambda} is set to $\bar t=1$, and impose the following assumption.

\begin{assumption}\label{as:local}
	The subproblem \eqref{equ:subproblem} is solved exactly, i.e., $\gamma=1$ in \eqref{equ:inexactcriterion}.
\end{assumption}

We first prove that the unit step size is always accepted in the line search scheme \eqref{equ:lambda} after sufficiently many iterations.

\begin{lemma}\label{lem:unitstep}
	Suppose that Assumptions~\ref{as:problem1}--\ref{as:ZOPN-DM}  and \ref{as:local} hold. If the line search parameter
	$
	c_2 \ge \frac{\kappa_{\mathrm{eg}}^2}{n\underline{\kappa}(1-2c_1)},
	$
	and the sampling radius sequence $\{\Delta_k\}$ is square-summable,
	then the unit step size $t_k=1$ satisfies~\eqref{equ:lambda} for all sufficiently large $k$.
\end{lemma}

\begin{proof}
	It follows from Assumption~\ref{as:problem2} that
		\begin{align}\label{eq:unitstep1}
			F(x_k+d_k) & \le f(x_k) + \nabla f(x_k)^{\top}d_k + \frac12 d_k^{\top}\nabla^2 f(x_k)d_k + \frac{L_H}{6}\|d_k\|^3 + h(x_k+d_k)  \nonumber \\
			& = F(x_k) + \Phi_k + (\nabla f(x_k)-g_k)^{\top}d_k + \frac12 d_k^{\top}\nabla^2 f(x_k)d_k + \frac{L_H}{6}\|d_k\|^3  \nonumber \\
			& \le F(x_k) + \Phi_k + \frac{1}{1-2c_1}\|\nabla f(x_k)-g_k\|_{H_k^{-1}}^2 + \frac{1-2c_1}{4} \|d_k\|_{H_k}^2   \nonumber \\
			& \hspace{1.2em} + \frac12 d_k^{\top}\nabla^2 f(x_k)d_k + \frac{L_H}{6}\|d_k\|^3  \nonumber \\
			& \le F(x_k) + \frac{2c_1+1}{4} \Phi_k + \frac{3-2c_1}{4}\Phi_k + \frac12 d_k^{\top}\left(\nabla^2 f(x_k) + \frac{1-2c_1}{2} H_k \right) d_k \nonumber \\
			& \hspace{1.2em} + \frac{\kappa_{\mathrm{eg}}^2}{\underline{\kappa}(1-2c_1) }\Delta_k^2 - \frac{L_H}{6\underline{\kappa}}\|d_k\|\Phi_k  \nonumber \\
			& \le F(x_k) + \left( \frac{2c_1+1}{4} - \frac{L_H}{6\underline{\kappa}}\|d_k\|\right) \Phi_k + \frac12 d_k^{\top}\left(\nabla^2 f(x_k) - H_k \right) d_k + nc_2 \Delta_k^2,
		\end{align}
	where the second inequality uses Young's inequality
	\begin{equation*}
		(\nabla f(x_k)-g_k)^{\top}d_k \le \frac{1}{2p}\|\nabla f(x_k)-g_k\|_{A^{-1}}^2 + \frac{p}{2} \|d_k\|_{A}^2
	\end{equation*}
	with $A=H_k\succeq\underline{\kappa}I$ and $p=1/2 - c_1>0$, the third inequality follows from Assumptions~\ref{as:ZOPN-g} and \ref{as:ZOPN-H}, along with
	\begin{equation}\label{eq:unitstep2}
		\|d_k\|^2 \le \frac{1}{\underline{\kappa}}d_k^{\top}H_k d_k \le -\frac{1}{\underline{\kappa}}\Phi_k
	\end{equation}
	by virtue of \eqref{equ:lemdd2} and Assumption~\ref{as:local}, and the last inequality is due to \eqref{equ:lemdd2} and the definition of $c_2$.
	
	By Assumptions~\ref{as:problem2} and \ref{as:ZOPN-DM}, we deduce from \eqref{eq:unitstep2} that
	\begin{equation*}
		\begin{aligned}
			\frac12 d_k^{\top}\left(\nabla^2 f(x_k) - H_k \right) d_k
			& = \frac12 d_k^{\top}\left(\nabla^2 f(x_k) - \nabla^2 f(x^{\star}) \right) d_k + \frac12 d_k^{\top}\left(\nabla^2 f(x^{\star}) - H_k \right) d_k \\
			& \le \frac12 \|\nabla^2 f(x_k) - \nabla^2 f(x^{\star})\| \|d_k\|^2 + \frac12 \|(\nabla^2 f(x^{\star}) - H_k)d_k\| \|d_k\| \\
			& \le \frac{L_H}{2}\|x_k - x^{\star}\| \|d_k\|^2 + \frac{\varrho_k}{2}\|d_k\|^2 \\
			& \le -\frac{L_H \|x_k - x^{\star}\| + \varrho_k}{2\underline{\kappa}} \Phi_k.
		\end{aligned}
	\end{equation*}
	Substituting it into \eqref{eq:unitstep1}, we obtain
	\begin{equation*}
		F(x_k + d_k) \le F(x_k) + \left( \frac{2c_1+1}{4} - \frac{L_H( \|d_k\| + 3\|x_k - x^{\star}\|) + 3 \varrho_k }{6\underline{\kappa}}\right) \Phi_k  + nc_2 \Delta_k^2.
	\end{equation*}
	
	Since $\{\Delta_k\}$ is square-summable, we deduce from Lemma~\ref{lem:station}, Lemma~\ref{thm:globalconv}, and Assumptions~\ref{as:problem1} and~\ref{as:problem2} that $\|d_k\|\to 0$ and $x_k\to x^{\star}$ as $k\to\infty$.
	Combining the fact that $\varrho_k\to 0$ from Assumption~\ref{as:ZOPN-DM}, it holds for sufficiently large $k$ that,
	\begin{equation*}
		\frac{L_H( \|d_k\| + 3\|x_k - x^{\star}\|) + 3 \varrho_k }{6\underline{\kappa}} \le \frac{1 - 2c_1}{4},
	\end{equation*}
	which implies
	\begin{equation*}
		F(x_k + d_k) \le F(x_k) + c_1 \Phi_k  + nc_2 \Delta_k^2,
	\end{equation*}
	and hence the unit step size $t_k=1$ is always accepted.
\end{proof}

The next auxiliary result helps us establish the relationship between $d_k$ in \eqref{equ:subproblem} and the search direction generated by the proximal Newton method with exact gradient and Hessian.

\begin{lemma}\label{prop:ddis}
	Given vectors $q_1,q_2\in\mathbb{R}^n$ and matrices $Q_1,Q_2\in\mathbb{S}_{++}^n$, let $d_1$ and $d_2$ be the search directions generated by solving
	\begin{equation*}
		d_i = \operatornamewithlimits{argmin}_{d\in \mathbb{R}^n}  q_i^{\top} d + \frac{1}{2}d^{\top}Q_i d + h(x + d), \quad i=1,2.
	\end{equation*}
	Then we have
	\begin{equation}\label{eq:ddis}
		\|d_1 - d_2\|_{Q_1}^2 \le 2\|q_1 - q_2\|_{Q_1^{-1}}^2 + 2\|(Q_2 - Q_1)d_2\|_{Q_1^{-1}}^2.
	\end{equation}
\end{lemma}

\begin{proof}
	From the optimality condition, we also have
	\begin{equation*}
		d_i = \operatornamewithlimits{argmin}_{d\in \mathbb{R}^n}  q_i^{\top} d + d_i^{\top}Q_i d + h(x + d), \quad i=1,2.
	\end{equation*}
	Thus,
	\begin{equation*}
		\begin{aligned}
			q_1^{\top}d_1 + d_1^{\top}Q_1 d_1 + h(x + d_1) & \le q_1^{\top} d_2 + d_1^{\top}Q_1 d_2 + h(x + d_2), \\
			q_2^{\top} d_2 + d_2^{\top}Q_2 d_2 + h(x + d_2) & \le q_2^{\top} d_1 + d_2^{\top}Q_2 d_1 + h(x + d_1).
		\end{aligned}
	\end{equation*}
	Summing the above two inequalities yields
	\begin{equation}\label{eq:ddis1}
		(q_1 - q_2)^{\top}(d_1 - d_2) + d_1^{\top}Q_1 d_1 + d_2^{\top}Q_2 d_2 - d_1^{\top}(Q_1 + Q_2) d_2 \le 0.
	\end{equation}
	
	By Young's inequality,
	\begin{equation}
		(q_1 - q_2)^{\top}(d_1 - d_2)  \ge - \left| (q_1 - q_2)^{\top}(d_1 - d_2)\right|  \ge - \|q_1 - q_2 \|_{Q_1^{-1}}^2 - \frac14 \|d_1 - d_2\|_{Q_1}^2. \label{eq:ddis2}
	\end{equation}
	Completing the square and rearranging terms, we apply Young's inequality again to obtain
		\begin{align}\label{eq:ddis3}
			d_1^{\top}Q_1 d_1 + d_2^{\top}Q_2 d_2 - d_1^{\top}(Q_1 + Q_2) d_2
			& = \|d_1 - d_2\|_{Q_1}^2 - (d_1 - d_2)^{\top}(Q_2 - Q_1)d_2  \nonumber \\
			& \ge \|d_1 - d_2\|_{Q_1}^2 - \|(Q_2 - Q_1)d_2\|_{Q_1^{-1}}^2 - \frac14 \|d_1 - d_2\|_{Q_1}^2  \nonumber \\
			& = \frac34 \|d_1 - d_2\|_{Q_1}^2 - \|(Q_2 - Q_1)d_2\|_{Q_1^{-1}}^2.
		\end{align}
	Substituting \eqref{eq:ddis2} and \eqref{eq:ddis3} into \eqref{eq:ddis1}, we immediately arrive at \eqref{eq:ddis}.
\end{proof}

This lemma enables us to establish the main result of this section.

\begin{theorem}\label{thm:localrate}
	Suppose that all assumptions of Lemma~\ref{lem:unitstep} hold. Then it holds for sufficiently large $k$ that
	\begin{equation}\label{eq:localrate}
			\|x_{k+1} - x^{\star} \| \le \frac{\sqrt{2}\kappa_{\mathrm{eg}}}{\mu - \sqrt{2}\varrho_k}\Delta_k  + \frac{\sqrt{2}(L_H\|d_k\|+\varrho_k)}{\mu - \sqrt{2}\varrho_k} \|x_k - x^{\star} \|   + \frac{L_H}{2\mu}\left( 1+ \frac{\sqrt{2}\varrho_k}{\mu - \sqrt{2}\varrho_k}\right) \|x_k - x^{\star} \|^2.
	\end{equation}
	Furthermore,  if $\{\Delta_k\}$ converges Q-superlinearly to 0, then $\{x_k\}$ converges R-superlinearly to $x^{\star}$.
	
\end{theorem}

\begin{proof}
	Denote $d_k^{\mathrm{N}}$ as the search direction generated by the proximal Newton method with exact gradient $\nabla f(x_k)$ and Hessian $\nabla^2 f(x_k)$. Then, it follows from the Q-quadratic convergence of the proximal Newton method (cf. \cite[Theorem~3.4]{PNM2014}) that
	\begin{equation}\label{eq:local1}
		\|x_k + d_k^{\mathrm{N}} - x^{\star}\| \le \frac{L_H}{2\mu} \|x_k - x^{\star} \|^2.
	\end{equation}
	
	It follows from Lemma~\ref{lem:unitstep} that $x_{k+1}=x_k+d_k$ holds for sufficiently large $k$. By \eqref{eq:local1},
	\begin{equation}
		\|x_{k+1} - x^{\star}\|  = \|x_k + d_k - x^{\star} \|   \le \|x_k + d_k^{\mathrm{N}} - x^{\star} \| + \| d_k^{\mathrm{N}} - d_k\|   \le \frac{L_H}{2\mu} \|x_k - x^{\star} \|^2  + \| d_k^{\mathrm{N}} - d_k\|. \label{eq:local2}
	\end{equation}
	By Lemma~\ref{prop:ddis},
	\begin{equation*}
		\| d_k^{\mathrm{N}} - d_k \|_{\nabla^2 f(x_k)}^2   \le 2 \|\nabla f(x_k) - g_k\|_{\nabla^2 f(x_k)^{-1}}^2  + 2\|(H_k - \nabla^2 f(x_k))d_k \|_{\nabla^2 f(x_k)^{-1}}^2.
	\end{equation*}
	This, together with Assumption~\ref{as:problem2} and the fact that $a^2+b^2 \le (a+b)^2$ for any $a,b\ge 0$, yields
	\begin{equation}\label{eq:local3}
		\| d_k^{\mathrm{N}} - d_k \| \le \frac{\sqrt{2}}{\mu}\left(\|\nabla f(x_k) - g_k\| + \|(H_k - \nabla^2 f(x_k))d_k \| \right) .
	\end{equation}
	It follows from Assumptions~\ref{as:problem2} and \ref{as:ZOPN-DM} that
		\begin{align}\label{eq:local4}
			\|(H_k - \nabla^2 f(x_k))d_k \| & \le \|(\nabla^2 f(x^{\star}) - \nabla^2 f(x_k))d_k \| + \|(H_k - \nabla^2 f(x^{\star}))d_k \|  \nonumber \\
			& \le L_H \|d_k\| \|x_k - x^{\star}\| + \varrho_k \|d_k\| \nonumber \\
			& \le L_H \|d_k\| \|x_k - x^{\star}\| + \varrho_k \|d_k - d_k^{\mathrm{N}}\| + \varrho_k \|d_k^{\mathrm{N}}\|.
		\end{align}
	By \eqref{eq:local1},
	\begin{equation}
		\|d_k^{\mathrm{N}}\| \le \|x_k + d_k^{\mathrm{N}} - x^{\star} \| + \|x_k - x^{\star}\|   \le \frac{L_H}{2\mu} \|x_k - x^{\star} \|^2 + \|x_k - x^{\star}\|. \label{eq:local5}
	\end{equation}
	Since $\varrho_k\to 0$, $\varrho_k < \mu/\sqrt{2}$ holds for sufficiently large $k$.
	Substituting \eqref{eq:local4} and \eqref{eq:local5} into \eqref{eq:local3} yields
	\begin{equation}
		\begin{aligned}
			\| d_k^{\mathrm{N}} - d_k \| & \le \frac{\sqrt{2}}{\mu - \sqrt{2}\varrho_k}\|\nabla f(x_k) - g_k\| + \frac{\sqrt{2}}{\mu - \sqrt{2}\varrho_k}(L_H\|d_k\|+\varrho_k)\|x_k - x^{\star}\|   \\
			&\hspace{1.2em} + \frac{L_H}{2\mu}\frac{\sqrt{2}\varrho_k}{\mu-\sqrt{2}\varrho_k}\|x_k - x^{\star}\|^2. \label{eq:local6}
		\end{aligned}
	\end{equation}
	Combining \eqref{eq:local6}, \eqref{eq:local2} and Assumption~\ref{as:ZOPN-g}, we obtain \eqref{eq:localrate}.
	
	Note that both $\|d_k\|$ and $\varrho_k$ decay to zero as $k\to \infty$ (cf. Assumption~\ref{as:ZOPN-DM} and Lemma~\ref{thm:globalconv}), \eqref{eq:localrate} implies
	\begin{equation}\label{eq:xdelta}
		\|x_{k+1} - x^{\star} \| \le \frac{\sqrt{2}\kappa_{\mathrm{eg}}}{\mu - \sqrt{2}\varrho_k}\Delta_k  +o(\|x_k - x^{\star} \|).
	\end{equation}
	Since $\{\Delta_k\}$ converges Q-superlinearly to 0, for any $\epsilon>0$, there exists  $K_1$ such that $\Delta_k/\Delta_{k-1}\le \epsilon/2 $ for all $k\ge K_1$.
	Meanwhile, there exists $K_2$ such that $o(\|x_k - x^{\star} \|)/\|x_k - x^{\star} \|\le \epsilon$ for all $k\ge K_2$. Hence, for any $k\ge \max\{K_1,K_2\}$, it follows from \eqref{eq:xdelta} that
	\begin{equation*}
		\begin{aligned}
			\|x_{k+1} - x^{\star} \| + \frac{\sqrt{2}\kappa_{\mathrm{eg}}}{\mu - \sqrt{2}\varrho_k}\Delta_k  &\le \frac{2\sqrt{2}\kappa_{\mathrm{eg}}}{\mu - \sqrt{2}\varrho_k}\Delta_k + o(\|x_k - x^{\star} \|) \\ &\le  \epsilon\left(\|x_k - x^{\star} \| + \frac{\sqrt{2}\kappa_{\mathrm{eg}}}{\mu - \sqrt{2}\varrho_k}\Delta_{k-1}\right).
		\end{aligned}
	\end{equation*}
	This implies that $\left\{\|x_{k+1} - x^{\star} \| + \frac{\sqrt{2}\kappa_{\mathrm{eg}}}{\mu - \sqrt{2}\varrho_k}\Delta_k\right\}$ is Q-superlinearly convergent to 0. Thus $\{\|x_{k+1} - x^{\star} \|\}$ is R-superlinearly convergent to 0. Therefore $\{x_k\}$
	converges R-superlinearly  to $x^{\star}$.
\end{proof}

From Theorem~\ref{thm:localrate}, $\{x_k\}$ converges locally R-superlinearly to $x^{\star}$ whenever the sampling radius $\Delta_k$ converges Q-superlinearly to 0.
A typical choice is $\Delta_k=q^{2^k}$ for $q\in(0,1)$.
Nevertheless, Q-superlinear convergence of $\{x_k\}$ cannot be guaranteed.

\subsection{Hessian approximations}\label{subsec:Hessian}

We impose two assumptions on the approximate Hessian in Assumptions~\ref{as:ZOPN-H} and \ref{as:ZOPN-DM}. In this section, we discuss two widely used Hessian approximation schemes and verify whether they satisfy Assumptions~\ref{as:ZOPN-H} and \ref{as:ZOPN-DM}.

\paragraph{\textbf{BFGS}}
As one of the most popular Hessian approximations, BFGS approximates the Hessian via
\begin{equation}\label{eq:bfgs}
	H_k = H_{k-1} + \frac{y_{k-1}y_{k-1}^{\top}}{y_{k-1}^{\top}s_{k-1}} - \frac{H_{k-1}s_{k-1}(H_{k-1}s_{k-1})^{\top}}{s_{k-1}^{\top}H_{k-1}s_{k-1}},
\end{equation}
where $s_{k-1}=x_k - x_{k-1}$, $y_{k-1}=g_k - g_{k-1}$.
A notable advantage of the BFGS update is that it does not require additional function evaluations to construct the Hessian approximation.

The following result given in \cite{Byrd1989} is a fundamental tool for our analysis of \eqref{eq:bfgs}.

\begin{lemma}[{\cite[Theorem~3.2]{Byrd1989}}]\label{lem:dmcondition}
	Let $H_k$ be generated by the BFGS formula \eqref{eq:bfgs}, where $H_0\in\mathbb{S}_{++}^n$, and where $y_k^{\top}s_k>0$ for all $k\in\mathbb{N}$. Furthermore, assume that $\{s_k\}$ and $\{y_k\}$ are such that
	\begin{equation*}
		\frac{\|y_k-G^{\ast}s_k\|}{\|s_k\|} \le \epsilon_k,
	\end{equation*}
	for some $G^{\ast}\in\mathbb{S}_{++}^n$ and for some sequence $\{\epsilon_k\}$ satisfying $\sum_{k=0}^{\infty}\epsilon_k < \infty$. Then,
	\begin{equation*}
		\lim_{k\to\infty} \frac{\|(H_k-G^{\ast})s_k\|}{\|s_k\|} = 0.
	\end{equation*}
\end{lemma}

\begin{theorem}\label{thm:bfgs}
	Suppose that Assumptions~\ref{as:problem1}--\ref{as:ZOPN-H} hold.
	For the BFGS update \eqref{eq:bfgs}, suppose the sampling radius satisfies
	\begin{equation}\label{eq:radiusbfgs}
		\Delta_k \le \min\left\lbrace \Delta_{k-1},  ( \tilde \tau/2  )^{k/2},  \frac{\mu\underline{t}}{4\kappa_{\mathrm{eg}}}\|d_k\| ,C\|d_k\|^3\right\rbrace,
	\end{equation}
	where $\tilde \tau\in(0,1)$ is given in Lemma~\ref{thm:linearconv}, $\underline{t}>0$ is defined in Lemma~\ref{lem:stepsizeuppbd}, and $C>0$ is a constant independent of $k$.
	Then $\{H_k\}$ satisfies the Dennis--Mor\'e condition \eqref{eq:DM}.
\end{theorem}

\begin{proof}
	As $f$ is $\mu$-strongly convex, it follows from Assumption~\ref{as:ZOPN-g} and \eqref{eq:radiusbfgs} that
		\begin{align}\label{eq:curvpos}
			y_k^{\top}s_k & = (g_{k+1}-g_k)^{\top}s_k  \nonumber \\
			& = (\nabla f(x_{k+1})- \nabla f(x_k))^{\top}s_k + (g_{k+1} - \nabla f(x_{k+1}))^{\top}s_k + (\nabla f(x_k) - g_k)^{\top}s_k  \nonumber \\
			& \ge \mu\|s_k\|^2 - (\|g_{k+1} - \nabla f(x_{k+1})\| + \|\nabla f(x_k) - g_k\|)\|s_k\|  \nonumber \\
			& \ge \|s_k\|(\mu\|s_k\|-2\kappa_{\mathrm{eg}}\Delta_k) \ge \|s_k\| \left( \mu\|s_k\|-2\kappa_{\mathrm{eg}}\frac{\mu t_k}{4\kappa_{\mathrm{eg}}}\|d_k\| \right) \ge \frac{\mu}{2}\|s_k\|^2 > 0.
		\end{align}
	
	Furthermore, by Assumptions~\ref{as:problem2}, \ref{as:ZOPN-g} and \eqref{eq:radiusbfgs},
	\begin{equation*}
		\begin{aligned}
			& \quad \frac{\|y_k-\nabla^2 f(x^{\star})s_k\|}{\|s_k\|} = \frac{\|g_{k+1} - g_k - \nabla^2 f(x^{\star})s_k\|}{\|s_k\|} \\
			&\le \frac{\|\nabla f(x_{k+1}) - \nabla f(x_k)- \nabla^2 f(x^{\star})s_k\|}{\|s_k\|}  + \frac{\| g_{k+1} - \nabla f(x_{k+1})\|}{\|s_k\|} + \frac{\| \nabla f(x_k) - g_k\|}{\|s_k\|} \\
			& \le \left\| \int_0^1 \nabla^2 f(x_k + \xi s_k)s_k \mathrm{d}\xi - \nabla^2 f(x^{\star})s_k\right\| \Big/ \|s_k\| + \frac{2\kappa_{\mathrm{eg}}\Delta_k}{\|s_k\|} \\
			& \le L_H\max\{\|x_{k+1} - x^{\star}\|, \|x_k - x^{\star}\|\} + \frac{2C\kappa_{\mathrm{eg}}}{\underline{t}}\|d_k\|^2.
		\end{aligned}
	\end{equation*}
	
	By virtue of \eqref{eq:radiusbfgs}, we obtain
	\begin{equation*}
		\sum_{k=0}^{\infty}\Delta_k^2 < \infty  \text{ and }   \sum_{k=0}^{\infty} \Delta_k^2 \tilde \tau^{-k} < \infty.
	\end{equation*}
	Thus, it follows from Lemmas~\ref{thm:globalconv} and \ref{thm:linearconv} that
	\begin{equation*}
		\sum_{k=0}^{\infty} \|d_k\|^2 < \infty  \text{ and }  \sum_{k=0}^{\infty} \|x_k - x^{\star} \| < \infty,
	\end{equation*}
	which implies
	\begin{equation*}
		\sum_{k=0}^{\infty} \frac{\|y_k-\nabla^2 f(x^{\star})s_k\|}{\|s_k\|} < \infty.
	\end{equation*}
	All the hypotheses of Lemma~\ref{lem:dmcondition} are satisfied. Therefore, we have
	\begin{equation*}
		\lim_{k\to\infty} \frac{\|(H_k- \nabla^2 f(x^{\star}))s_k\|}{\|s_k\|} = 0,
	\end{equation*}
	i.e.,  $\{H_k\}$ satisfies the Dennis--Mor\'e condition \eqref{eq:DM}.
\end{proof}

In Theorem~\ref{thm:bfgs}, we directly assume that the eigenvalues of $\{H_k\}$ are bounded. In fact, it is shown in \cite[Online supplement]{ZOQN2023} that Assumption~\ref{as:ZOPN-H} holds algorithmically for the BFGS update~\eqref{eq:bfgs} due to \eqref{eq:curvpos}. In addition, the sampling radius condition \eqref{eq:radiusbfgs} seems unreasonable at first glance, as it involves $\|d_k\|$, whereas $d_k$ is determined after $\Delta_k$ in Algorithm~\ref{ZOPN}. In fact, \eqref{eq:radiusbfgs} serves primarily as a theoretical condition and cannot be directly enforced in the algorithmic loop.
Furthermore, since $\|d_k\|$ is commonly adopted as the stopping criterion of Algorithm~\ref{ZOPN}, one can ensure \eqref{eq:radiusbfgs} holds in practical implementation. Given a prescribed tolerance $\varepsilon>0$, initializing $\Delta_0\le \varepsilon^3$ guarantees the bound $\Delta_k\le C\|d_k\|^3$. This is because $\|d_k\|$ never falls below $\varepsilon$ before the algorithm terminates.

\paragraph{\textbf{Lazy Hessian}}
Another popular Hessian approximation based on zeroth-order oracle is the lazy Hessian update (cf. \cite{ZOCNM2023})
\begin{equation}\label{eq:lazyH}
	[H_k]_{ij} = \frac{f(x_k+\Delta_k e_i + \Delta_k e_j) - f(x_k+\Delta_k e_i) - f(x_k+\Delta_k e_j) + f(x_k)}{\Delta_k^2},
\end{equation}
where $i,j\in[n]$. The Lazy Hessian scheme updates the approximate Hessian via \eqref{eq:lazyH} every $n$ iterations to reduce the high computational cost of function evaluations. Accordingly, the average number of function evaluations per iteration is $\Theta(n)$.

\begin{theorem}\label{thm:lazyH}
	Suppose that Assumptions~\ref{as:problem1}--\ref{as:ZOPN-g} hold. If the sampling radius satisfies
	\begin{equation}\label{eq:radiuslazyH}
		\Delta_k \le \frac{3\mu}{2(\sqrt{2}+1)nL_H} \quad \text{ and } \quad
		\sum_{k=0}^{\infty}\Delta_k^2<\infty,
	\end{equation}
	then the lazy Hessian update \eqref{eq:lazyH} satisfies Assumptions~\ref{as:ZOPN-H} and \ref{as:ZOPN-DM}.
\end{theorem}

\begin{proof}
	For any $k\in \mathbb{N}$, we have
	$	k = n \left \lfloor \frac{k}n \right \rfloor + k - n \left \lfloor \frac{k}n \right \rfloor $.
	Since the lazy Hessian scheme updates the approximate Hessian via \eqref{eq:lazyH} every $n$ iterations, we have
	$	H_k \equiv H_{n \left \lfloor \frac{k}n \right \rfloor}$ for all $ k\in\mathbb{N} $.
	
	It follows from \cite[Lemma~6]{ZOCNM2023} that
	\begin{equation*}
		\left\| H_k - \nabla^2 f \left( x_{n \left \lfloor \frac{k}n \right \rfloor} \right)  \right\|  = \left\| H_{n \left \lfloor \frac{k}n \right \rfloor} - \nabla^2 f \left( x_{n \left \lfloor \frac{k}n \right \rfloor} \right)  \right\|  \le \frac{(\sqrt{2}+1)nL_H}{3}\Delta_{n \left \lfloor \frac{k}n \right \rfloor}.
	\end{equation*}
	Since $\nabla^2 f \succeq \mu I$, Assumption~\ref{as:ZOPN-H} is readily satisfied under the sampling radius condition \eqref{eq:radiuslazyH}.
	
	Similarly, we have
	\begin{equation*}
		\begin{aligned}
			\left\| H_k - \nabla^2 f(x^{\star}) \right\| & = \left\| H_{n \left \lfloor \frac{k}n \right \rfloor} - \nabla^2 f (x^{\star}) \right\| \\
			& \le \left\| H_{n \left \lfloor \frac{k}n \right \rfloor} - \nabla^2 f \left( x_{n \left \lfloor \frac{k}n \right \rfloor} \right)  \right\| + \left\| \nabla^2 f \left( x_{n \left \lfloor \frac{k}n \right \rfloor} \right) - \nabla^2 f(x^{\star})  \right\| \\
			& \le \frac{(\sqrt{2}+1)nL_H}{3}\Delta_{n \left \lfloor \frac{k}n \right \rfloor} + L_H \left\|  x_{n \left \lfloor \frac{k}n \right \rfloor} - x^{\star} \right\| .
		\end{aligned}
	\end{equation*}
	This, together with Lemma~\ref{lem:station} and  Lemma~\ref{thm:globalconv}, implies that
	\begin{equation*}
		\frac{\|(H_k - \nabla^2 f(x^{\star}))(x_{k+1}-x_k)\|}{\|x_{k+1} - x_k\|} \le \| H_k - \nabla^2 f(x^{\star})\| \to 0 \quad \text{ as }\ k\to \infty,
	\end{equation*}
	which verifies Assumption~\ref{as:ZOPN-DM}.
\end{proof}

\section{Numerical Experiments}\label{sec:experiment}

In this section, we test Algorithm \ref{ZOPN} (ZOPN) and compare it with four zeroth-order proximal gradient methods:
SS-ProxGD (cf. \cite{ZOPG2023,UniZProxSG2020,GE2022}) utilizing a two-point spherical smoothing to estimate the gradient,
GS-ProxGD (cf. \cite{ZOPG2023,GE2022}) which estimates the gradient via two-point Gaussian smoothing,
DGS-ProxGD (cf. \cite{DSZProxSG2021,ZOPG2023}) employing a double Gaussian smoothing
\begin{equation*}
	\nabla^{\mathrm{DGS}}_{\Delta_a,\Delta_b}f(x):=\frac{ (f(x+\Delta_a u_1+\Delta_b u_2)-f(x+\Delta_a u_1)  )u_2 }{\Delta_b},
\end{equation*}
where $u_1, u_2\sim \mathcal{N}(0,I)$ and $\Delta_a, \Delta_b>0$ are sampling radii,
and FD-ProxGD which approximates the gradient via the forward difference.
The experiments are implemented on a 64-bit laptop with an Intel Core i7-13700H 2.4GHz CPU and 16GB RAM, using MATLAB R2023a.

In Algorithm \ref{ZOPN}, we use the standard forward difference to estimate the gradient, and the approximate Hessian $H_k$ is updated by \eqref{eq:bfgs} if
\begin{equation}\label{equ:BFGSbdd}
	y_{k-1}^{\top}s_{k-1} \ge 10^{-9}  \|s_{k-1} \|^2,
\end{equation}
such that $H_k$ satisfies Assumption~\ref{as:ZOPN-H} algorithmically (cf. \cite[Lemma 3.3]{BFGSbdd2016} and \cite[Online supplement]{ZOQN2023}), otherwise it keeps unchanged.
The initial matrix $H_0$ is chosen as the identity matrix.
We apply Algorithm \ref{alg:FISTA} (FISTA) to solve \eqref{equ:subproblem} with $\gamma=0.9$ and $10^3$ maximum number of inner iterations.

We set $\Delta_k\equiv \Delta = 5\times 10^{-10}$ for SS-ProxGD, GS-ProxGD, FD-ProxGD and ZOPN, and $\Delta_a=5\times 10^{-7}$, $\Delta_b= \Delta = 5\times 10^{-10}$ for DGS-ProxGD as in \cite{ZOPG2023}.
We report results for the best versions of SS-ProxGD, GS-ProxGD, DGS-ProxGD and FD-ProxGD based on tuning the constant step size for each problem (i.e., by considering  $t=2^j, j\in \{  -15,-14,\ldots,9,10 \}$).
For all experiments we report the mean results across 10 different random runs for SS-ProxGD, GS-ProxGD and DGS-ProxGD.

The default parameters of Algorithm \ref{ZOPN} are set as $c_1=10^{-4}$, $c_2=10^{-8}$, $\beta=0.5$, and $\bar{t}=1$.
To further explore the potential of the proposed method, we also report numerical results obtained by tuning the initial step size $\bar{t}$ for instances where the default configuration yields unsatisfactory performance.
The experimental results demonstrate that the default parameter setting is sufficiently stable and robust, and performs well on the vast majority of test problems.

All algorithms stop at iteration $k$ if
\begin{equation}\label{equ:stopcriterion}
	F(x_k)-F^{\star}< 10^{-16} \text{ or }  \mathrm{NF}>300(n+1),
\end{equation}
where $F^{\star}$ is obtained by the first-order proximal gradient method, and NF stands for the number of function evaluations.
The initial point for all tested algorithms is set to be the origin.

We test the algorithms on the  LASSO problem, the $\ell_1$-regularized logistic regression,
the $\ell_2$-regularized logistic regression, the elastic net-regularized binary classification
as well as the nonconvex support vector machine problem.
Except for the first problem, we use the datasets from the
LIBSVM website\footnote{\href{https://www.csie.ntu.edu.tw/\textasciitilde cjlin/libsvmtools/datasets/binary.html}{https://www.csie.ntu.edu.tw/\textasciitilde cjlin/libsvmtools/datasets/binary.html}}.
The detailed information of the datasets is presented in Table \ref{tab:data}.

\begin{table}[H]
	\centering
	\begin{tabular}{lcc}
		\hline
		Dataset & Sample size $p$ & Dimension $n$  \\  \hline
		\texttt{a1a} & 30956 & 123 \\
		\texttt{a4a} & 27780 & 123 \\
		\texttt{a9a} & 16281 & 122 \\
		\texttt{covtype} & 581012 & 54 \\
		\texttt{heart} & 270 & 13 \\
		\texttt{ijcnn1} & 49990 & 22 \\
		\texttt{mushrooms} & 8124 & 112 \\
		\texttt{sonar} & 208 & 60 \\
		\texttt{svmguide3} & 41 & 22 \\
		\texttt{w1a} & 47272 & 300 \\
		\texttt{w4a} & 42383 & 300 \\
		\texttt{w8a} & 14951 & 300 \\ \hline
	\end{tabular}
	\caption{Datasets information.}
	\label{tab:data}
\end{table}

\subsection{LASSO problem}\label{subsec:lasso}

Consider the LASSO problem \cite{APSR2022,IPZOPM2024}
\begin{equation}\label{equ:lasso}
	\min_{x\in \mathbb{R}^n} \frac12 \| Ax-b \|^2 + \zeta  \| x \|_1,
\end{equation}
where $A\in \mathbb{R}^{p\times n}$ is a column-unit Gaussian matrix, $x\in \mathbb{R}^n$ is a Gaussian vector with a sparsity level 0.01,
and $b\in \mathbb{R}^p$ is obtained by $Ax$ disturbed with a Gaussian white noise of level $10^{-4}$.
We set $\zeta=5\times 10^{-3}$ and $p=0.4n$ with $n=500, 1000, 2000, 4000$, respectively.

We test all algorithms on the problem. The results are presented in Figure \ref{fig:LassoNF}.
It can be seen that the performances of SS-ProxGD, DGS-ProxGD and GS-ProxGD are similar.
They fail to achieve the desired accuracy within the predetermined computational efforts.
FD-ProxGD usually performs better than SS-ProxGD, DGS-ProxGD and GS-ProxGD,
and ZOPN performs best among all the algorithms and requires much fewer number of function evaluations to get the optimal solution.
The reason is that ZOPN takes into account the second-order information of the problem.

\begin{figure}[htbp!]
	\centering
	\begin{subfigure}{0.24\linewidth}
		\centering
		\includegraphics[width=1\linewidth]{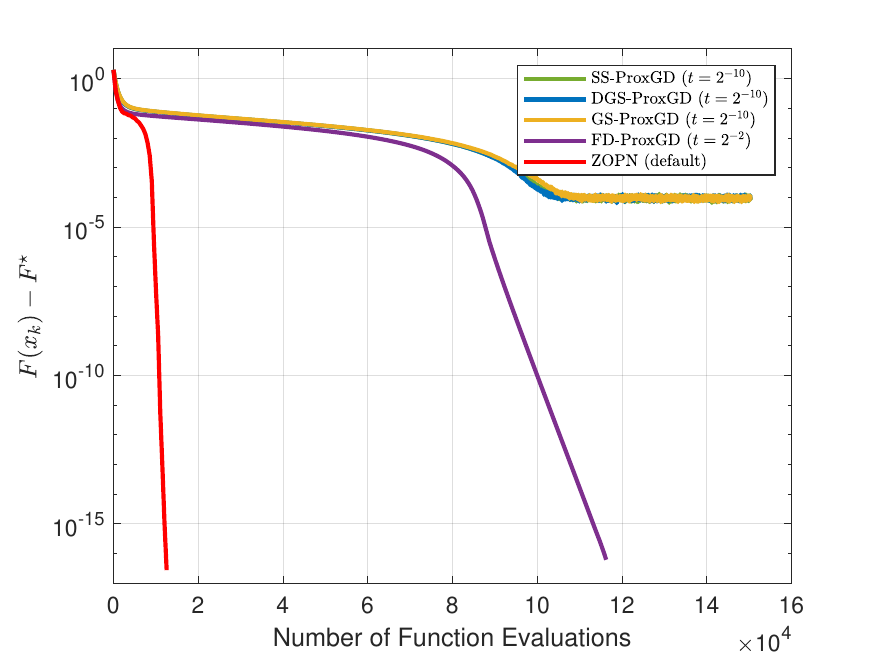}
		\caption{$n = 500$}
		\label{Lasso5e2}
	\end{subfigure}
	\centering
	\begin{subfigure}{0.24\linewidth}
		\centering
		\includegraphics[width=1\linewidth]{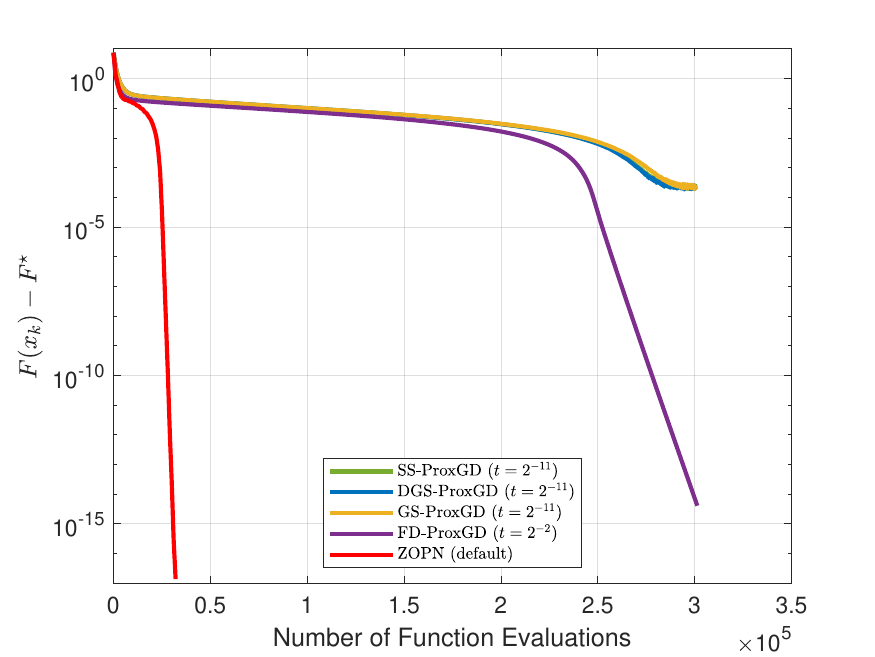}
		\caption{$n = 1000$}
		\label{Lasso1e3}
	\end{subfigure}
	\centering
	\begin{subfigure}{0.24\linewidth}
		\centering
		\includegraphics[width=1\linewidth]{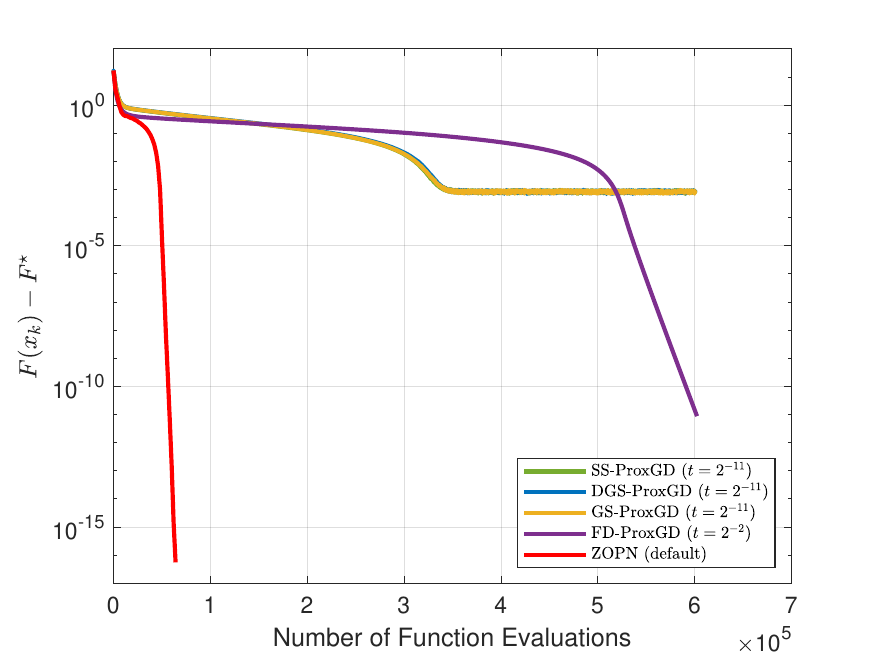}
		\caption{$n = 2000$}
		\label{Lasso2e3}
	\end{subfigure}
	\centering
	\begin{subfigure}{0.24\linewidth}
		\centering
		\includegraphics[width=1\linewidth]{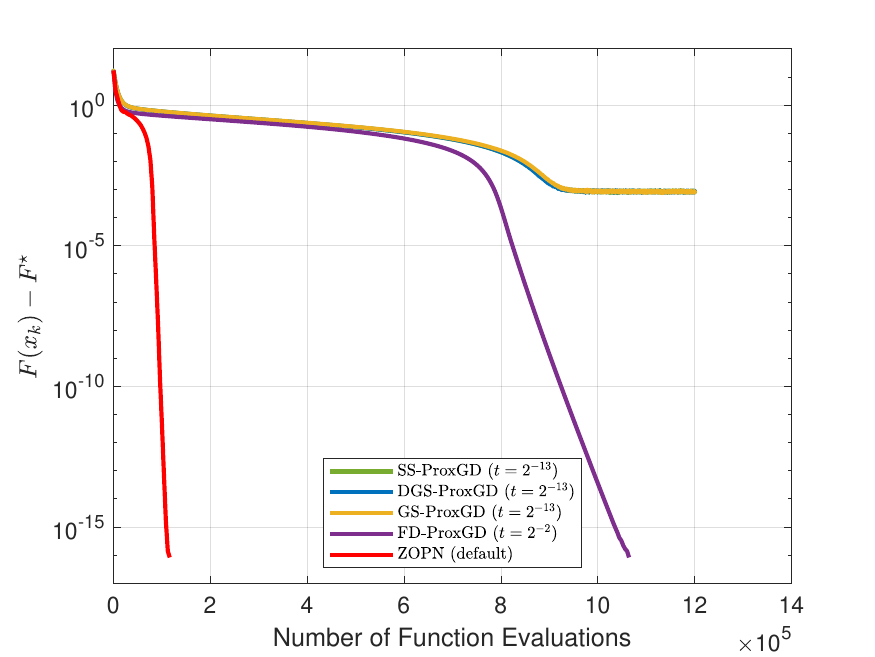}
		\caption{$n = 4000$}
		\label{Lasso4e3}
	\end{subfigure}
	\caption{Objective value versus NF on the LASSO problem.}
	\label{fig:LassoNF}
\end{figure}
 

The results in Figure~\ref{fig:LassoNF} show fast superlinear-like decay for the LASSO problem.
We use this example to verify the local R-superlinear convergence stated in Theorem~\ref{thm:localrate}. We test two variants of Algorithm~\ref{ZOPN}: ZOPN-BFGS, combining forward-difference-based gradients with the BFGS Hessian approximation \eqref{eq:bfgs}, and ZOPN-LazyH, sharing the same gradient estimator with the lazy Hessian strategy \eqref{eq:lazyH}.

The true sparse vector $x$ has a sparsity level of 0.1.
We set the line search parameter $c_2=1$ and let the sampling radius decay as $\Delta_k = \max\{  10^{-10},  \min \{ 10^{-3}, 0.99^{2^k}\} \}$. Algorithm~\ref{alg:FISTA} runs with
$10^4$ inner iterations. The initial point is $x_0=x^{\star}+\sigma u$, where $\sigma=1/n$, $u\sim\mathcal{N}(0,I)$, and $x^{\star}$ is produced by the first-order proximal gradient method for~\eqref{equ:lasso}. All other configurations remain the same as above.
We track the absolute error
$\varepsilon_k:=\|x_k - x^{\star}\|$ and the root error $\mathrm{root}_k:=\varepsilon_k^{1/(k+1)}$.
Both algorithms stop when $\varepsilon_k<10^{-6}$.

The results are summarized in Table~\ref{tab:eg2} and Figure~\ref{fig:eg2}, where a dash ``--" denotes early algorithm termination. ZOPN-LazyH stagnates over iterations 2--9. Under the radius updating rule, $\Delta_k$ remains at $10^{-3}$ for $k=0,\ldots,9$ and begins to decay only at
$k=10$. The algorithm thus fails to advance within the neighborhood induced by this fixed large sampling radius. Once $\Delta_k$ decreases at $k=10$, the error drops rapidly.
By contrast, ZOPN-BFGS reduces the error steadily across all iterations.
Its root error declines markedly in later iterations, verifying the local R-superlinear convergence property. These numerical observations agree well with the local convergence theory in Section~\ref{sec:Local}.

\begin{table}[hbt!]
	\centering
	\resizebox{\textwidth}{!}{
		\begin{tabular}{lllllllllllllllllllllllll}
			\hline
			& & \multicolumn{7}{l}{$n=10$} & & \multicolumn{7}{l}{$n=20$} & &
			\multicolumn{7}{l}{$n=50$} \\
			\cline{3-9} \cline{11-17} \cline{19-25}
			$k$ & & \multicolumn{4}{l}{ZOPN-BFGS} & \multicolumn{3}{l}{ZOPN-LazyH} & & \multicolumn{4}{l}{ZOPN-BFGS} & \multicolumn{3}{l}{ZOPN-LazyH} & & \multicolumn{4}{l}{ZOPN-BFGS} & \multicolumn{3}{l}{ZOPN-LazyH}  \\
			\cline{3-5} \cline{7-9} \cline{11-13} \cline{15-17} \cline{19-21} \cline{23-25}
			& & NF & $\varepsilon_k$ & $\mathrm{root}_k$ & & NF & $\varepsilon_k$ & $\mathrm{root}_k$ & & NF & $\varepsilon_k$ & $\mathrm{root}_k$ & & NF & $\varepsilon_k$ & $\mathrm{root}_k$ & & NF & $\varepsilon_k$ & $\mathrm{root}_k$ & & NF & $\varepsilon_k$ & $\mathrm{root}_k$ \\
			\hline
			0 & & 0 & 3.70e-01 & 3.70e-01 & & 0 & 3.70e-01 & 3.70e-01 & & 0 & 2.38e-01 & 2.38e-01 & & 0 & 2.38e-01 & 2.38e-01 & & 0 & 1.25e-01 & 1.25e-01 & & 0 & 1.25e-01 & 1.25e-01 \\
			1 & & 12 & 2.46e-01 & 4.96e-01 & & 66 & 5.00e-04 & 2.24e-02 & & 22 & 1.79e-01 & 4.23e-01 & & 231 & 1.17e-02 & 1.08e-01 & & 52 & 1.08e-01 & 3.28e-01 & & 1326 & 1.22e-03 & 3.50e-02 \\
			\vdots & & \vdots & \vdots & \vdots & & \vdots & \vdots & \vdots & & \vdots & \vdots & \vdots & & \vdots & \vdots & \vdots & & \vdots & \vdots & \vdots & & \vdots & \vdots & \vdots \\
			9 & & 100 & 7.30e-02 & 7.70e-01 & & 154 & 5.00e-04 & 4.68e-01 & & 192 & 5.71e-02 & 7.51e-01 & & 399 & 1.17e-02 & 6.41e-01 & & 460 & 3.00e-02 & 7.04e-01 & & 1734 & 1.22e-03 & 5.11e-01 \\
			10 & & 111 & 7.22e-02 & 7.87e-01 & & 165 & 1.70e-05 & 3.68e-01 & & 213 & 2.85e-02 & 7.24e-01 & & 420 & 3.97e-04 & 4.91e-01 & & 511 & 2.49e-02 & 7.15e-01 & & 1785 & 5.41e-05 & 4.09e-01 \\
			11 & & 122 & 7.15e-02 & 8.03e-01 & & 231 & 4.71e-06 & 3.60e-01 & & 234 & 2.17e-02 & 7.27e-01 & & 441 & 2.37e-09 & 1.91e-01 & & 562 & 1.90e-02 & 7.19e-01 & & 1836 & 5.55e-09 & 2.05e-01 \\
			12 & & 133 & 6.96e-02 & 8.15e-01 & & 242 & 1.32e-06 & 3.53e-01 & & 255 & 8.64e-03 & 6.94e-01 & & -- & -- & -- & & 613 & 1.40e-02 & 7.20e-01 & & -- & -- & -- \\
			13 & & 144 & 6.69e-02 & 8.24e-01 & & 253 & 3.69e-07 & 3.47e-01 & & 276 & 7.81e-03 & 7.07e-01 & & -- & -- & -- & & 664 & 8.17e-03 & 7.09e-01 & & -- & -- & -- \\
			14 & & 155 & 6.32e-02 & 8.32e-01 & & -- & -- & -- & & 297 & 6.94e-03 & 7.18e-01 & & -- & -- & -- & & 715 & 5.34e-03 & 7.06e-01 & & -- & -- & -- \\
			15 & & 166 & 5.67e-02 & 8.36e-01 & & -- & -- & -- & & 318 & 6.53e-03 & 7.30e-01 & & -- & -- & -- & & 766 & 1.69e-03 & 6.71e-01 & & -- & -- & -- \\
			16 & & 177 & 3.79e-02 & 8.25e-01 & & -- & -- & -- & & 339 & 5.43e-03 & 7.36e-01 & & -- & -- & -- & & 817 & 8.01e-04 & 6.57e-01 & & -- & -- & -- \\
			17 & & 188 & 1.84e-02 & 8.01e-01 & & -- & -- & -- & & 360 & 3.82e-03 & 7.34e-01 & & -- & -- & -- & & 868 & 1.77e-04 & 6.19e-01 & & -- & -- & -- \\
			18 & & 199 & 4.77e-03 & 7.55e-01 & & -- & -- & -- & & 381 & 1.36e-03 & 7.07e-01 & & -- & -- & -- & & 919 & 7.55e-05 & 6.07e-01 & & -- & -- & -- \\
			19 & & 210 & 1.23e-03 & 7.15e-01 & & -- & -- & -- & & 402 & 2.38e-04 & 6.59e-01 & & -- & -- & -- & & 970 & 2.09e-05 & 5.83e-01 & & -- & -- & -- \\
			20 & & 221 & 2.05e-04 & 6.67e-01 & & -- & -- & -- & & 423 & 1.44e-04 & 6.56e-01 & & -- & -- & -- & & 1021 & 1.28e-05 & 5.85e-01 & & -- & -- & -- \\
			21 & & 232 & 2.20e-07 & 4.98e-01 & & -- & -- & -- & & 444 & 2.34e-05 & 6.16e-01 & & -- & -- & -- & & 1072 & 6.18e-06 & 5.80e-01 & & -- & -- & -- \\
			22 & & -- & -- & -- & & -- & -- & -- & & 465 & 6.57e-07 & 5.39e-01 & & -- & -- & -- & & 1123 & 2.19e-06 & 5.67e-01 & & -- & -- & -- \\
			23 & & -- & -- & -- & & -- & -- & -- & & -- & -- & -- & & -- & -- & -- & & 1174 & 3.59e-07 & 5.39e-01 & & -- & -- & -- \\
			\hline
		\end{tabular}
	}
	\caption{Numerical results of Algorithm~\ref{ZOPN} on the LASSO problem.}
	\label{tab:eg2}
\end{table}

\begin{figure}[htbp!]
	\centering
	\begin{subfigure}{0.32\linewidth}
		\centering
		\includegraphics[width=1\linewidth]{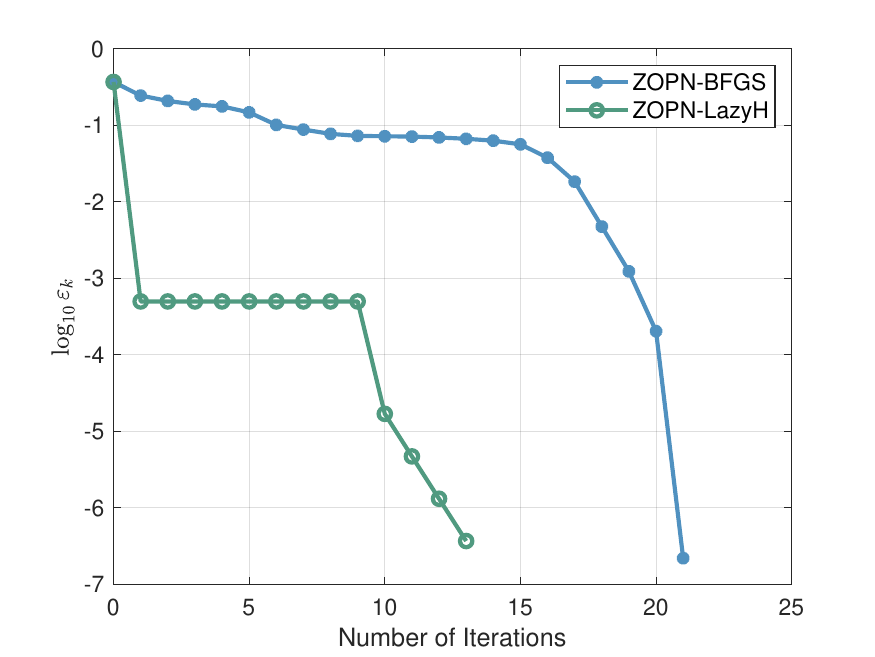}
		\caption{$n=10$}
	\end{subfigure}
	\centering
	\begin{subfigure}{0.32\linewidth}
		\centering
		\includegraphics[width=1\linewidth]{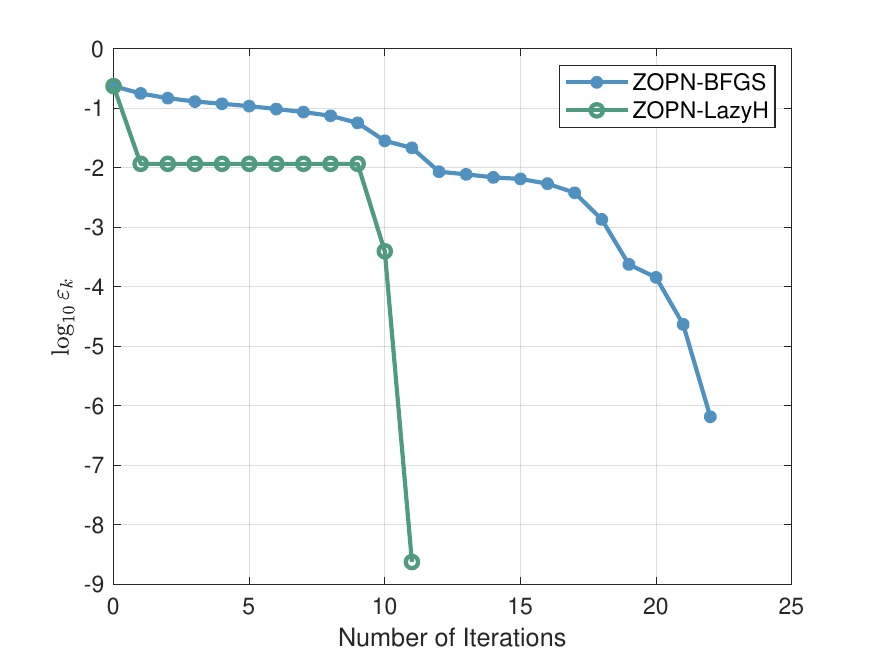}
		\caption{$n=20$}
	\end{subfigure}
	\centering
	\begin{subfigure}{0.32\linewidth}
		\centering
		\includegraphics[width=1\linewidth]{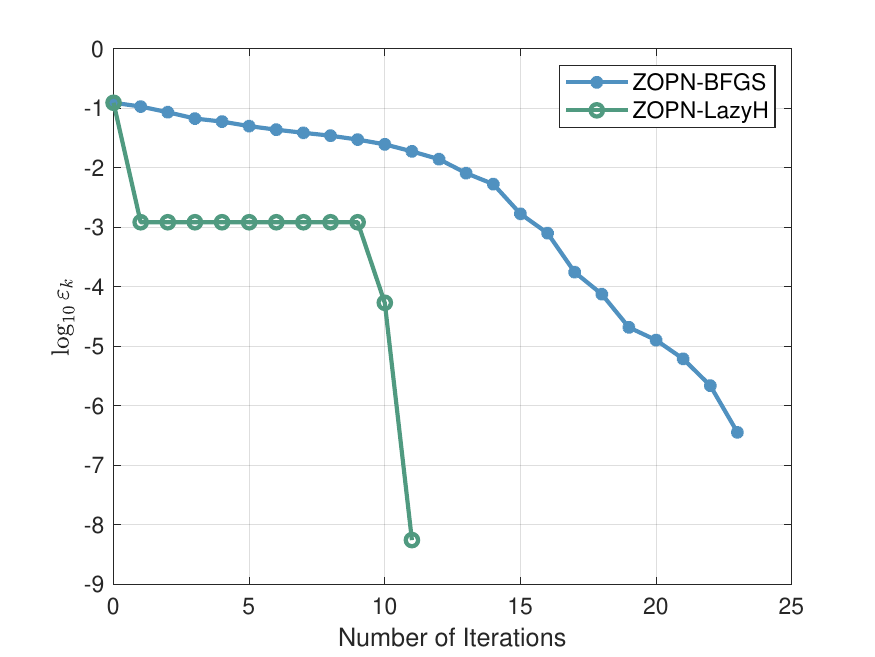}
		\caption{$n=50$}
	\end{subfigure}
	\caption{Evolution of \(\varepsilon_k\) versus iteration $k$ on the LASSO problem.}
	\label{fig:eg2}
\end{figure}

\subsection{$\ell_1$-regularized logistic regression}\label{subsec:L1logistic}
Consider the $\ell_1$-regularized logistic regression problem \cite{Inexact2021,PNM2014,Globalized2021,APSR2022,NIM2016}
\begin{equation*}
	\min_{x\in \mathbb{R}^n}\frac1p \sum_{i=1}^p \mathrm{log} (1+\exp (-b_i a_i^{\top}x  ) ) + \zeta \| x \|_1,
\end{equation*}
where we set $\zeta=10^{-3}$.

The results are summarized in Figure \ref{fig:L1logNF}.
SS-ProxGD, DGS-ProxGD and GS-ProxGD perform almost same on all datasets.
They outperform FD-ProxGD but underperform ZOPN on \texttt{a1a}, \texttt{a4a}, \texttt{a9a}, \texttt{covtype}, \texttt{mushrooms} and \texttt{sonar}.
It can also be observed that ZOPN performs best except for \texttt{w1a} and \texttt{w4a}.
For \texttt{w8a}, ZOPN efficiently obtains a low-accuracy solution with fewer function evaluations, while FD-ProxGD gradually achieves higher accuracy as the number of function evaluations increases. The two methods attain comparable final accuracy under the same computational budget.

\begin{figure}[htbp!]
	\centering
	\begin{subfigure}{0.24\linewidth}
		\centering
		\includegraphics[width=1\linewidth]{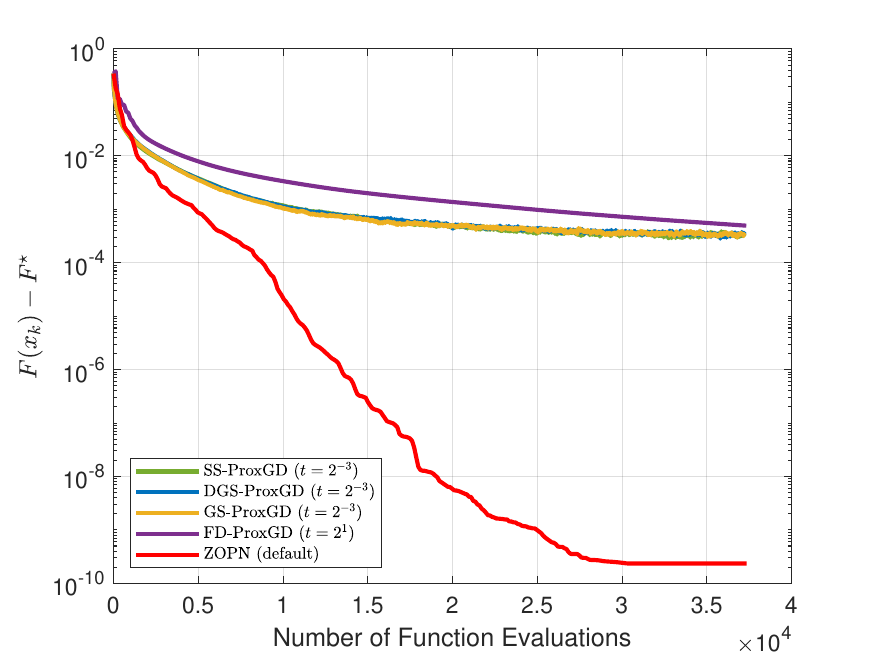}
		\caption{\texttt{a1a}}
		\label{L1loga1a}
	\end{subfigure}
	\centering
	\begin{subfigure}{0.24\linewidth}
		\centering
		\includegraphics[width=1\linewidth]{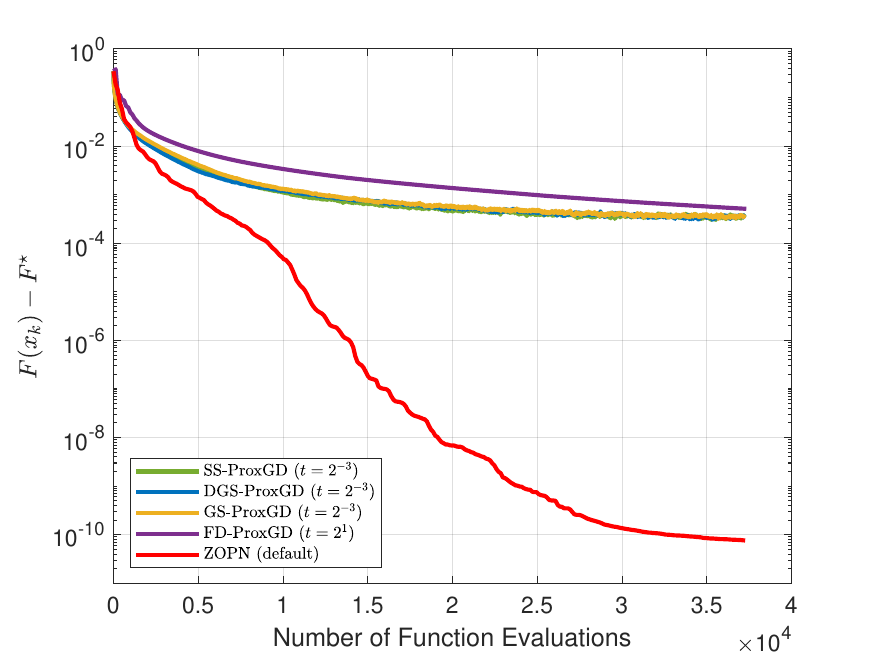}
		\caption{\texttt{a4a}}
		\label{L1loga4a}
	\end{subfigure}
	\centering
	\begin{subfigure}{0.24\linewidth}
		\centering
		\includegraphics[width=1\linewidth]{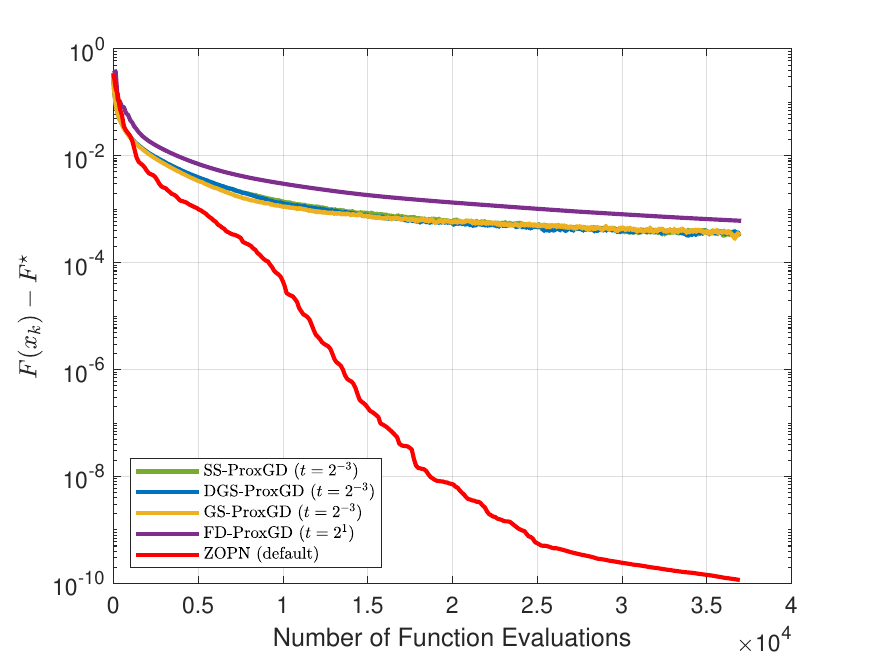}
		\caption{\texttt{a9a}}
		\label{L1loga9a}
	\end{subfigure}
	\centering
	\begin{subfigure}{0.24\linewidth}
		\centering
		\includegraphics[width=1\linewidth]{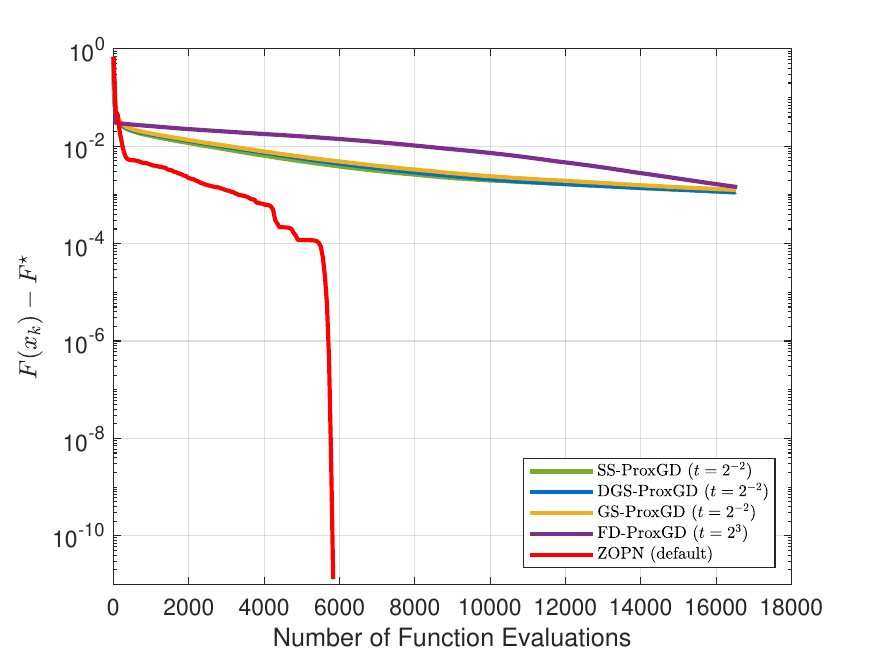}
		\caption{\texttt{covtype}}
		\label{L1logcovtype}
	\end{subfigure}
	
	\centering
	\begin{subfigure}{0.24\linewidth}
		\centering
		\includegraphics[width=1\linewidth]{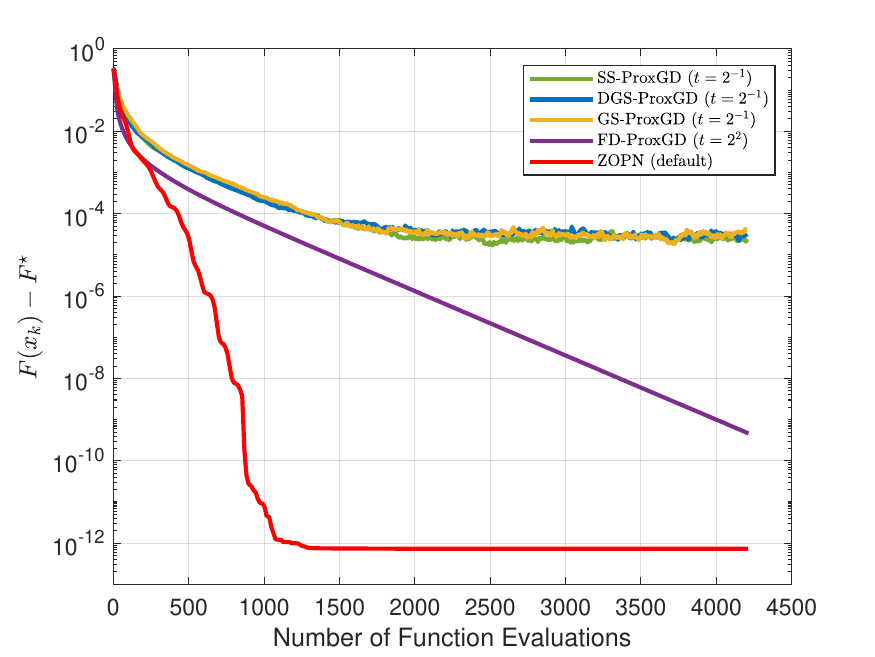}
		\caption{\texttt{heart}}
		\label{L1logheart}
	\end{subfigure}
	\centering
	\begin{subfigure}{0.24\linewidth}
		\centering
		\includegraphics[width=1\linewidth]{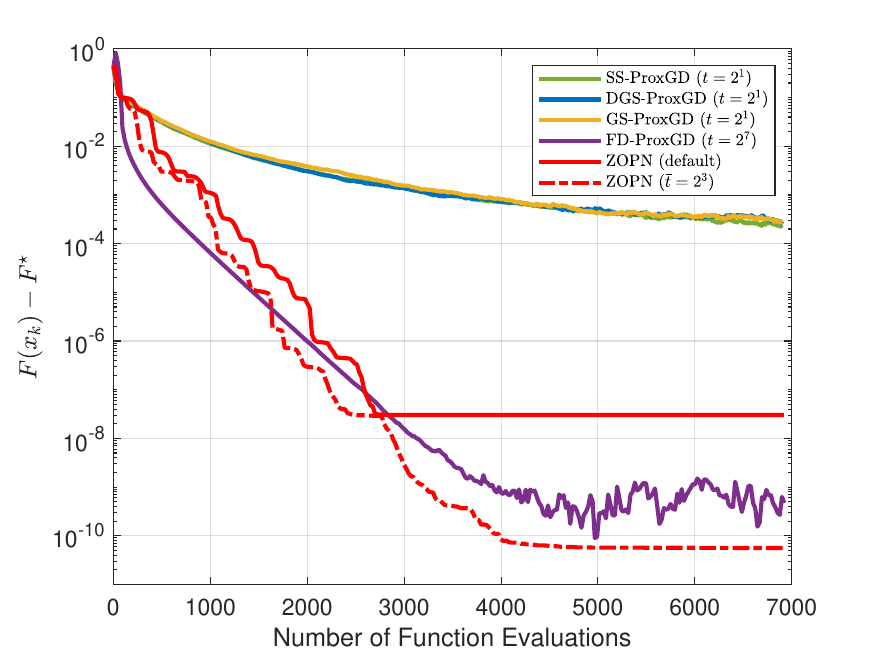}
		\caption{\texttt{ijcnn1}}
		\label{L1logijcnn}
	\end{subfigure}
	\centering
	\begin{subfigure}{0.24\linewidth}
		\centering
		\includegraphics[width=1\linewidth]{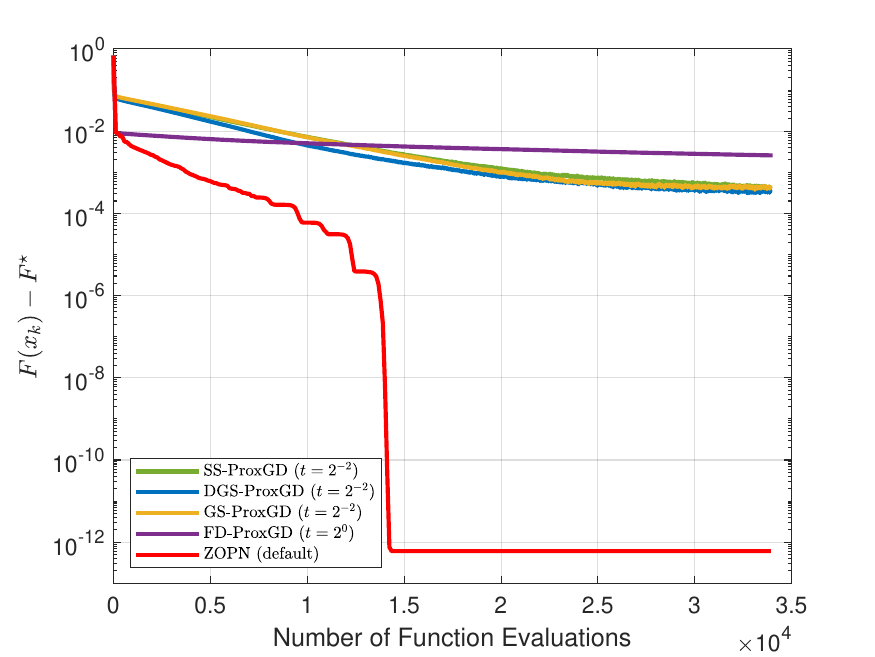}
		\caption{\texttt{mushrooms}}
		\label{L1logmushroom}
	\end{subfigure}
	\centering
	\begin{subfigure}{0.24\linewidth}
		\centering
		\includegraphics[width=1\linewidth]{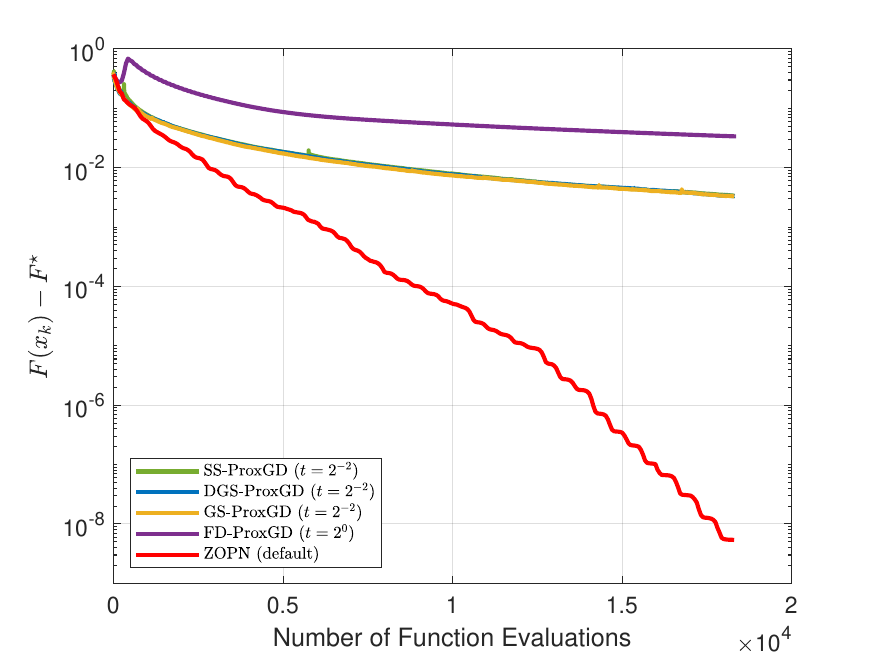}
		\caption{\texttt{sonar}}
		\label{L1logsonar}
	\end{subfigure}
	
	\centering
	\begin{subfigure}{0.24\linewidth}
		\centering
		\includegraphics[width=1\linewidth]{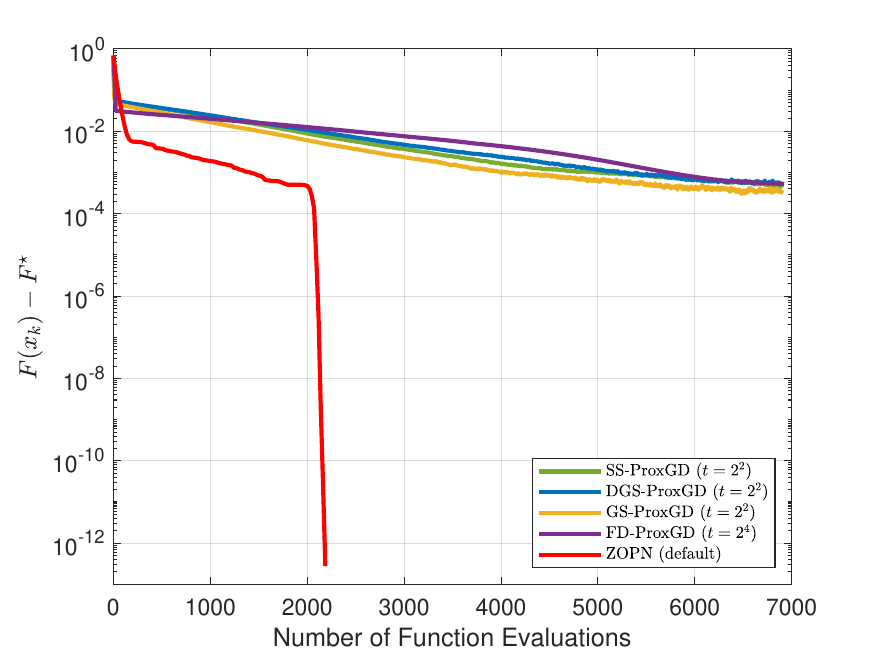}
		\caption{\texttt{svmguide3}}
		\label{L1logsvmguide}
	\end{subfigure}
	\centering
	\begin{subfigure}{0.24\linewidth}
		\centering
		\includegraphics[width=1\linewidth]{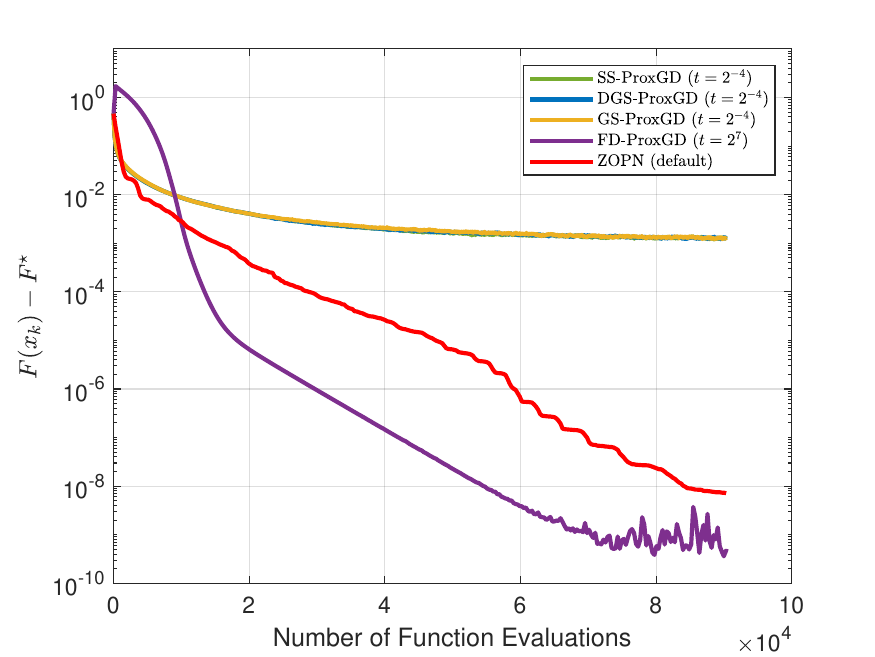}
		\caption{\texttt{w1a}}
		\label{L1logw1a}
	\end{subfigure}
	\centering
	\begin{subfigure}{0.24\linewidth}
		\centering
		\includegraphics[width=1\linewidth]{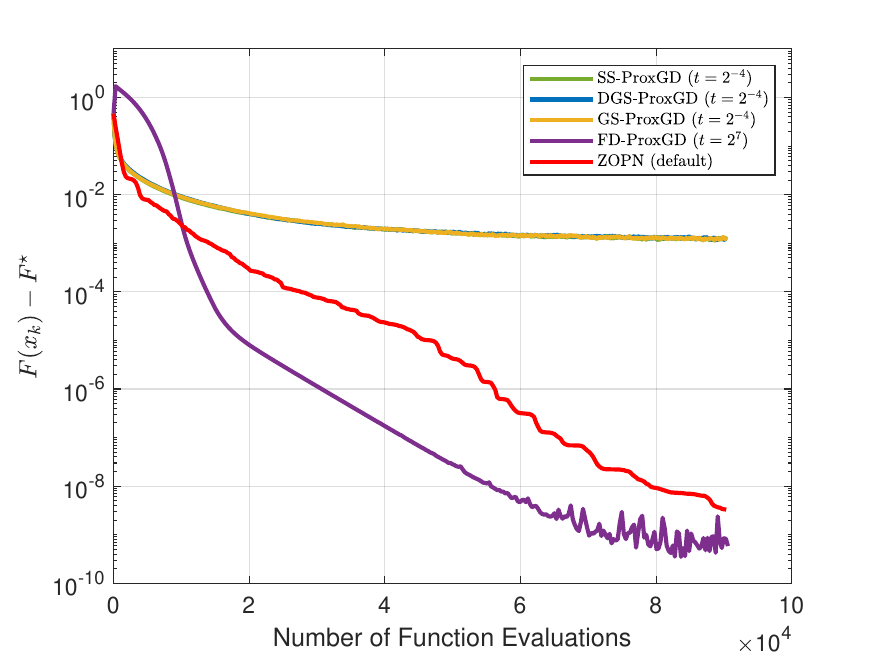}
		\caption{\texttt{w4a}}
		\label{L1logw4a}
	\end{subfigure}
	\centering
	\begin{subfigure}{0.24\linewidth}
		\centering
		\includegraphics[width=1\linewidth]{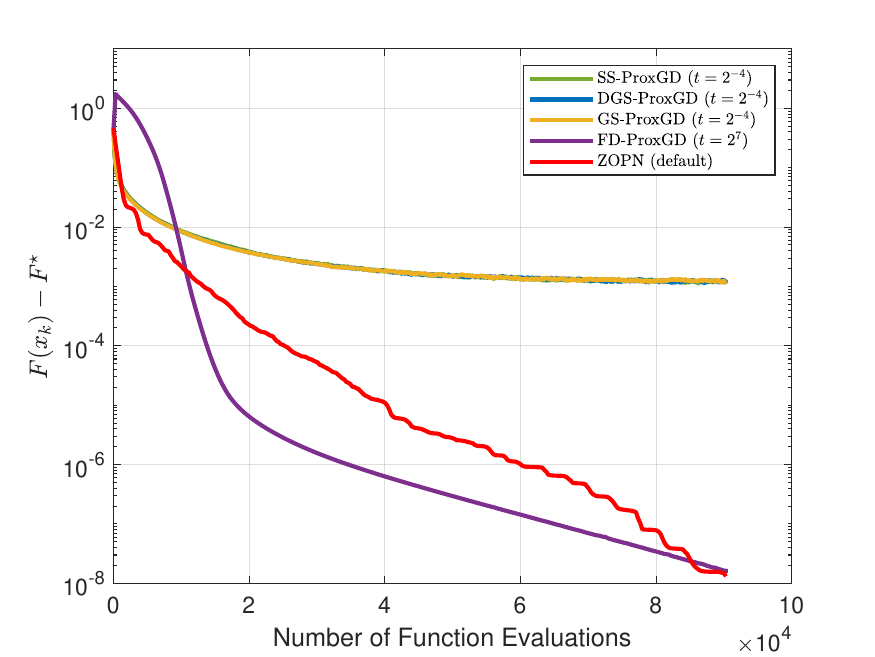}
		\caption{\texttt{w8a}}
		\label{L1logw8a}
	\end{subfigure}
	\caption{Objective value versus NF on the  $\ell_1$-regularized logistic regression problem.}
	\label{fig:L1logNF}
\end{figure}

\subsection{$\ell_2$-regularized logistic regression}\label{subsec:L2logistic}
Consider the $\ell_2$-regularized logistic regression problem \cite{NIM2016}
\begin{equation*}
	\min_{x\in \mathbb{R}^n}\frac1p \sum_{i=1}^p \mathrm{log} (1+\exp (-b_i a_i^{\top}x )  ) + \frac{\zeta}{2}  \| x \|^2,
\end{equation*}
where we set $\zeta=10^{-3}$.

\begin{figure}[htbp!]
	\centering
	\begin{subfigure}{0.24\linewidth}
		\centering
		\includegraphics[width=1\linewidth]{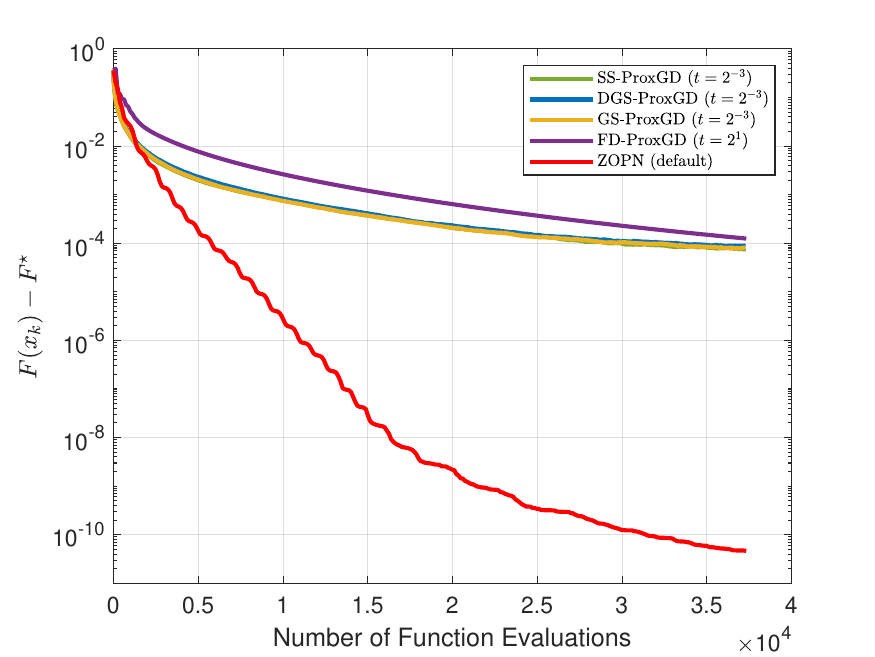}
		\caption{\texttt{a1a}}
		\label{L2loga1a}
	\end{subfigure}
	\centering
	\begin{subfigure}{0.24\linewidth}
		\centering
		\includegraphics[width=1\linewidth]{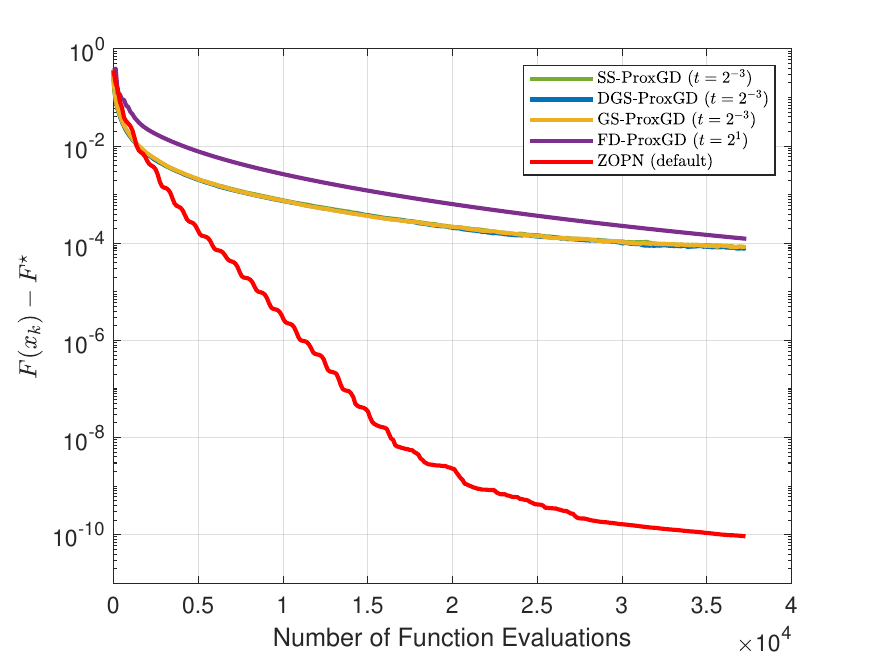}
		\caption{\texttt{a4a}}
		\label{L2loga4a}
	\end{subfigure}
	\centering
	\begin{subfigure}{0.24\linewidth}
		\centering
		\includegraphics[width=1\linewidth]{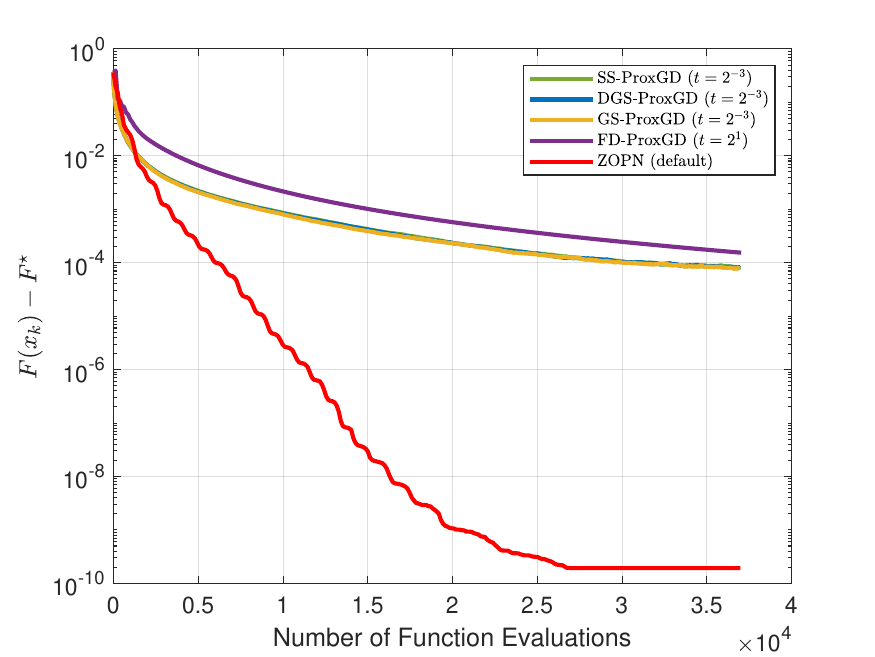}
		\caption{\texttt{a9a}}
		\label{L2loga9a}
	\end{subfigure}
	\centering
	\begin{subfigure}{0.24\linewidth}
		\centering
		\includegraphics[width=1\linewidth]{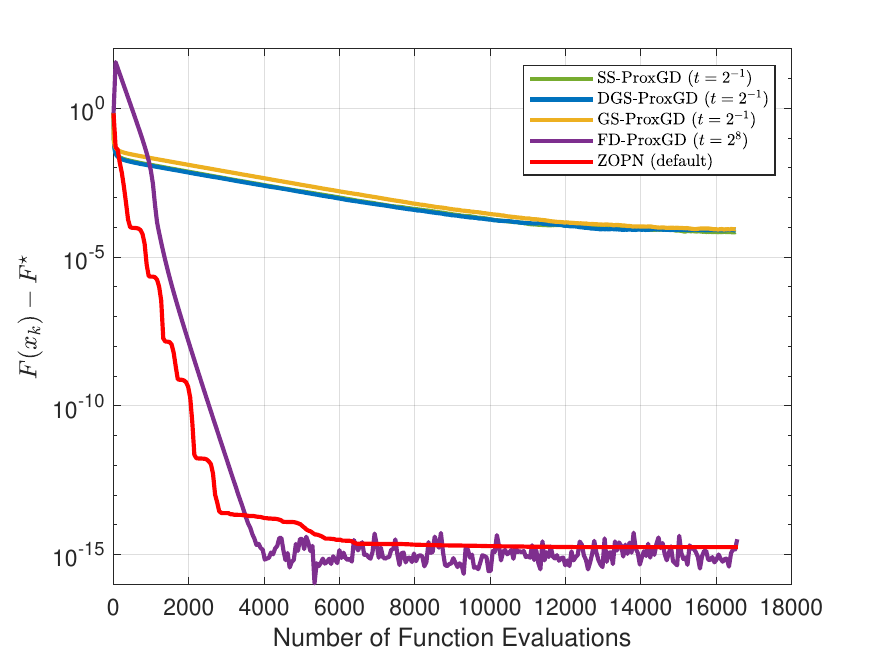}
		\caption{\texttt{covtype}}
		\label{L2logcovtype}
	\end{subfigure}
	
	\centering
	\begin{subfigure}{0.24\linewidth}
		\centering
		\includegraphics[width=1\linewidth]{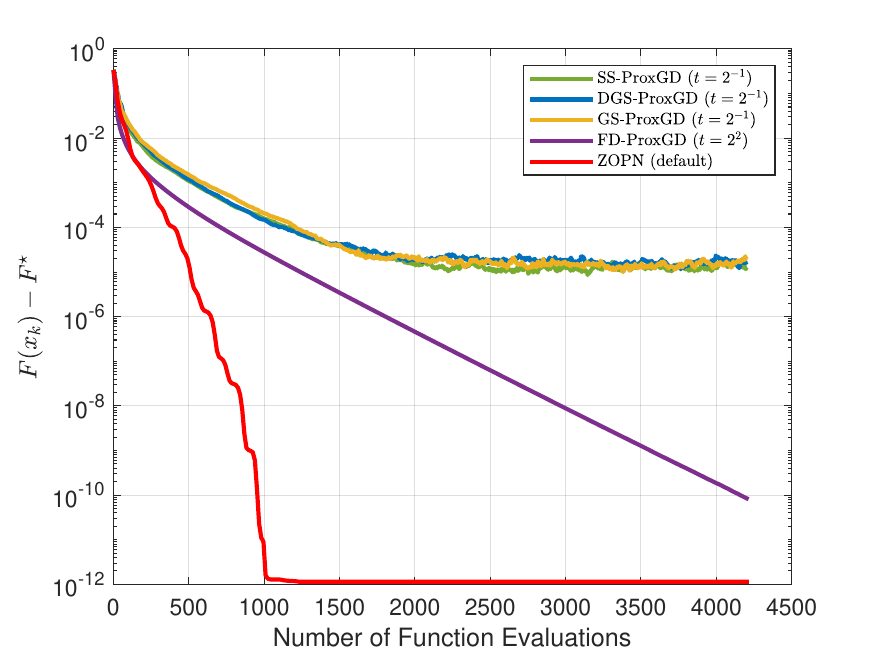}
		\caption{\texttt{heart}}
		\label{L2logheart}
	\end{subfigure}
	\centering
	\begin{subfigure}{0.24\linewidth}
		\centering
		\includegraphics[width=1\linewidth]{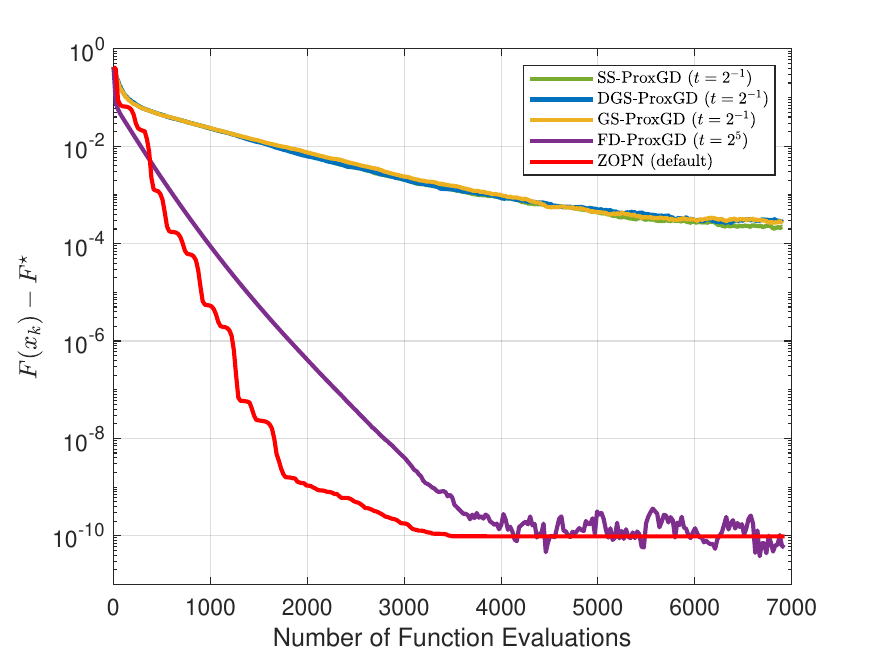}
		\caption{\texttt{ijcnn1}}
		\label{L2logijcnn}
	\end{subfigure}
	\centering
	\begin{subfigure}{0.24\linewidth}
		\centering
		\includegraphics[width=1\linewidth]{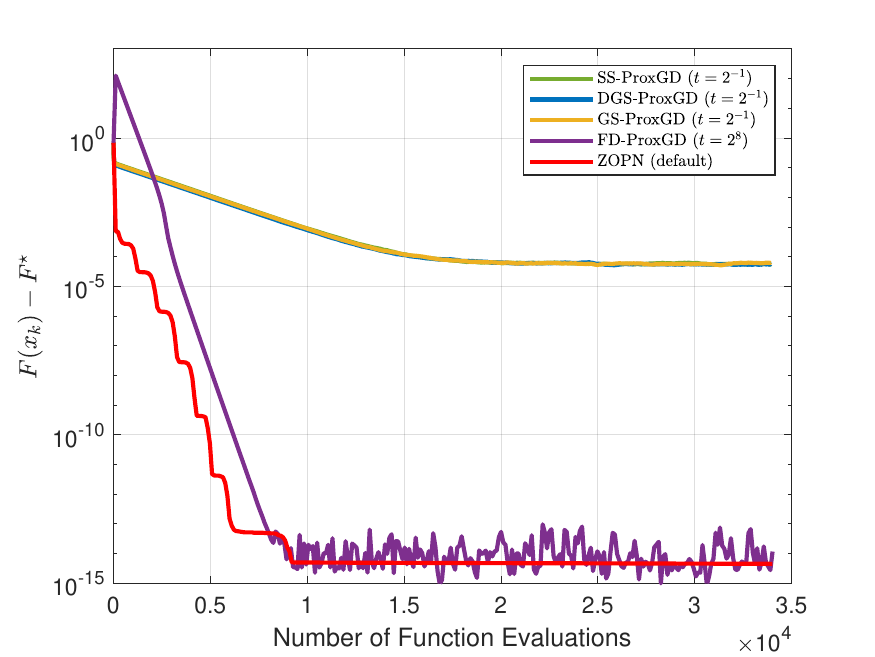}
		\caption{\texttt{mushrooms}}
		\label{L2logmushroom}
	\end{subfigure}
	\centering
	\begin{subfigure}{0.24\linewidth}
		\centering
		\includegraphics[width=1\linewidth]{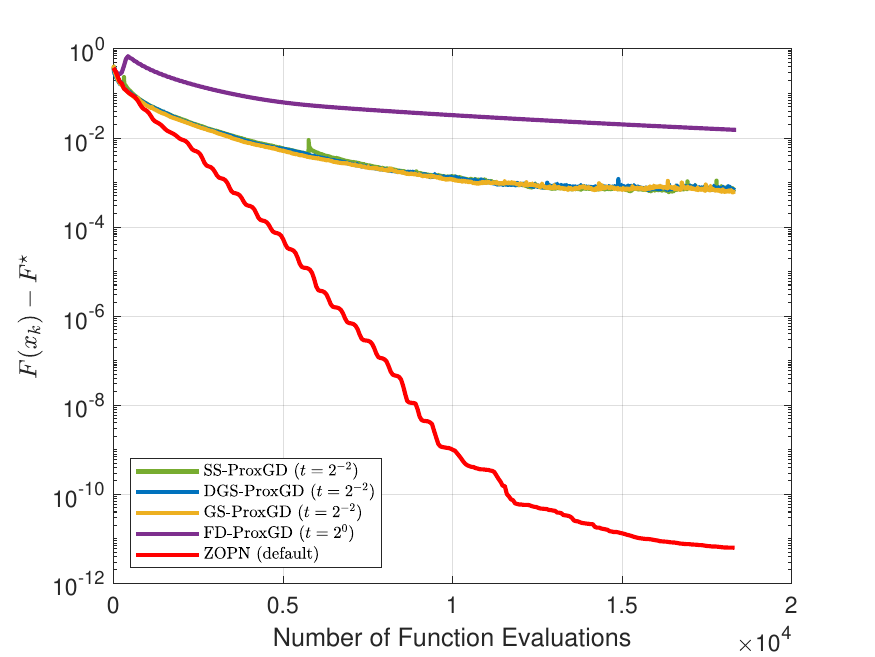}
		\caption{\texttt{sonar}}
		\label{L2logsonar}
	\end{subfigure}
	
	\centering
	\begin{subfigure}{0.24\linewidth}
		\centering
		\includegraphics[width=1\linewidth]{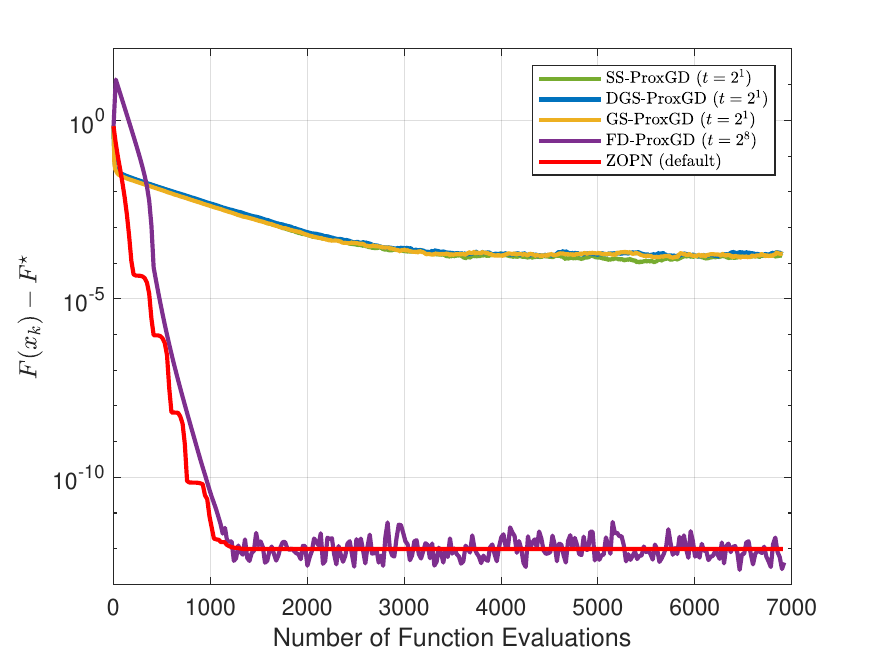}
		\caption{\texttt{svmguide3}}
		\label{L2logsvmguide}
	\end{subfigure}
	\centering
	\begin{subfigure}{0.24\linewidth}
		\centering
		\includegraphics[width=1\linewidth]{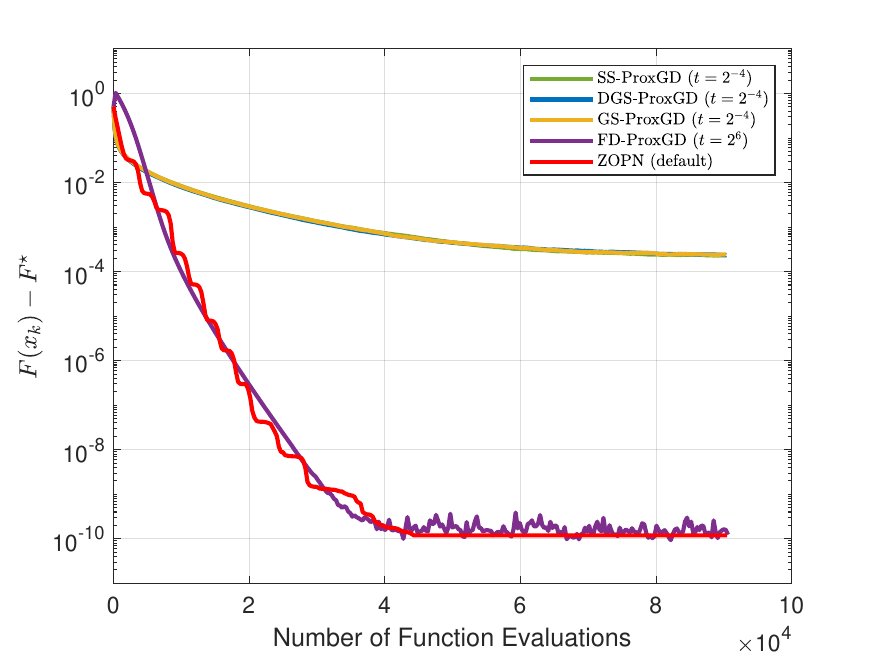}
		\caption{\texttt{w1a}}
		\label{L2logw1a}
	\end{subfigure}
	\centering
	\begin{subfigure}{0.24\linewidth}
		\centering
		\includegraphics[width=1\linewidth]{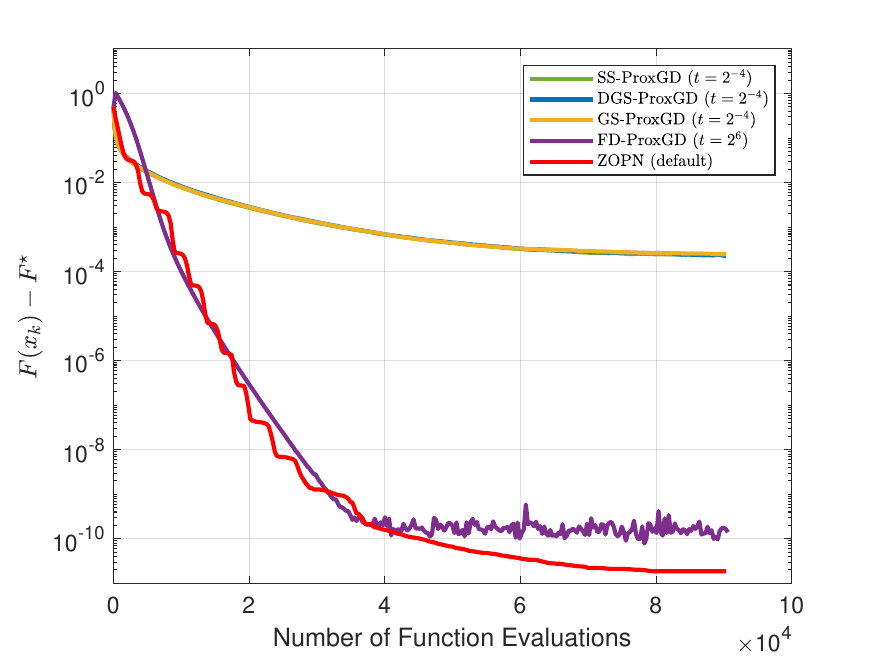}
		\caption{\texttt{w4a}}
		\label{L2logw4a}
	\end{subfigure}
	\centering
	\begin{subfigure}{0.24\linewidth}
		\centering
		\includegraphics[width=1\linewidth]{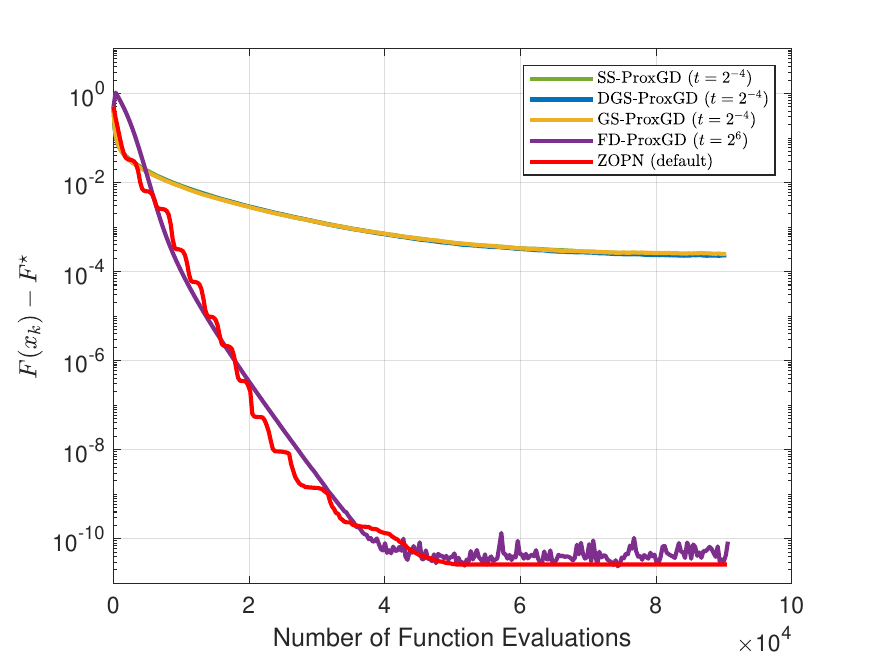}
		\caption{\texttt{w8a}}
		\label{L2logw8a}
	\end{subfigure}
	\caption{Objective value versus NF on the  $\ell_2$-regularized logistic regression problem.}
	\label{fig:L2logNF}
\end{figure}

The results are shown in Figure \ref{fig:L2logNF}.
ZOPN exhibits a significant advantage on  \texttt{a1a}, \texttt{a4a}, \texttt{a9a}, \texttt{heart}, \texttt{ijcnn1},
\texttt{mushrooms} and \texttt{sonar}, while it performs similarly to FD-ProxGD on the rest of datasets.
In general, 
ZOPN can get a high precision solution with less function queries compared with other algorithms.

\subsection{Elastic net-regularized binary classification}\label{subsec:Ebinary}
Consider the elastic net-regularized binary classification problem \cite{IPZOPM2024,ZOPSVRG2019,SZOPM2024}
\begin{equation*}
	\min_{x\in \mathbb{R}^n}\frac1p \sum_{i=1}^p \frac{1}{1+\exp (b_i a_i^{\top}x  )} + \zeta_1  \| x \|_1 + \frac{\zeta_2}{2}  \| x \|^2,
\end{equation*}
where we set $\zeta_1=10^{-3}$ and $\zeta_2=2\times 10^{-3}$.

The results are presented in Figure \ref{fig:BinaryNF}.
ZOPN outperforms FD-ProxGD on \texttt{a1a}, \texttt{a4a}, \texttt{a9a} and \texttt{sonar}, and both of them surpass the performance of SS-ProxGD, DGS-ProxGD and GS-ProxGD.
For other datasets, the performance of ZOPN and FD-ProxGD are similar, and both of them are superior to other algorithms.

\begin{figure}[htbp!]
	\centering
	\begin{subfigure}{0.24\linewidth}
		\centering
		\includegraphics[width=1\linewidth]{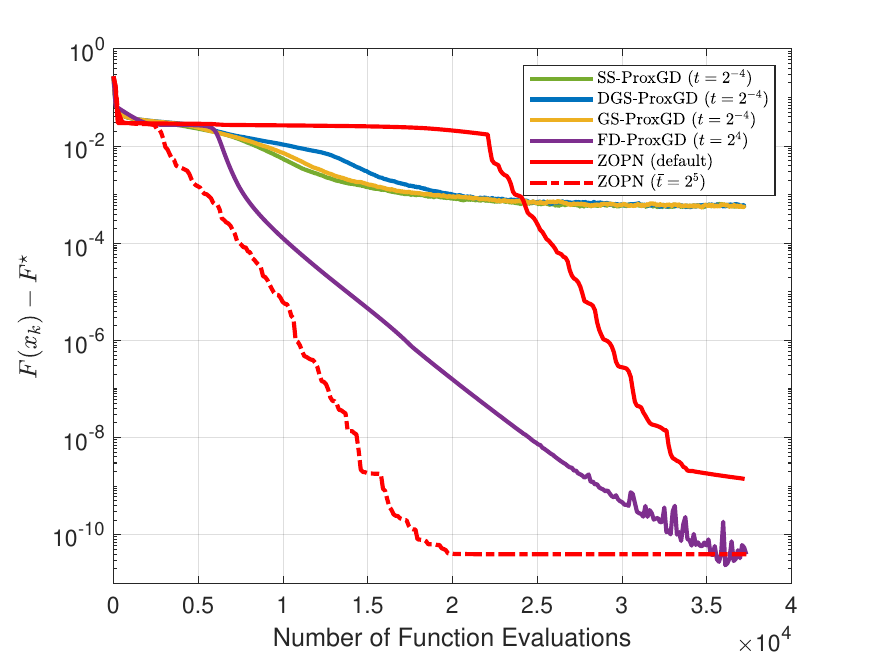}
		\caption{\texttt{a1a}}
		\label{Binarya1a}
	\end{subfigure}
	\centering
	\begin{subfigure}{0.24\linewidth}
		\centering
		\includegraphics[width=1\linewidth]{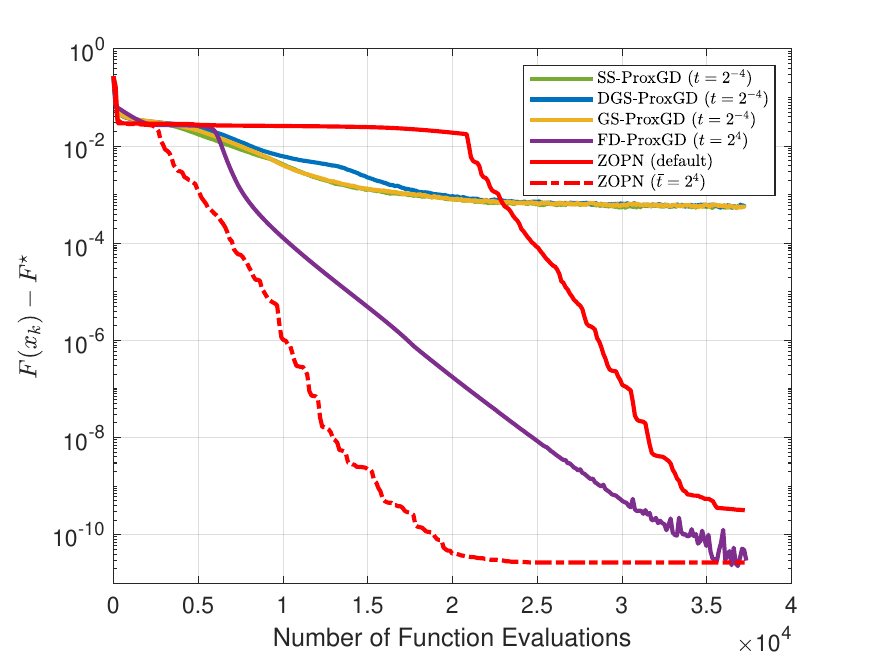}
		\caption{\texttt{a4a}}
		\label{Binarya4a}
	\end{subfigure}
	\centering
	\begin{subfigure}{0.24\linewidth}
		\centering
		\includegraphics[width=1\linewidth]{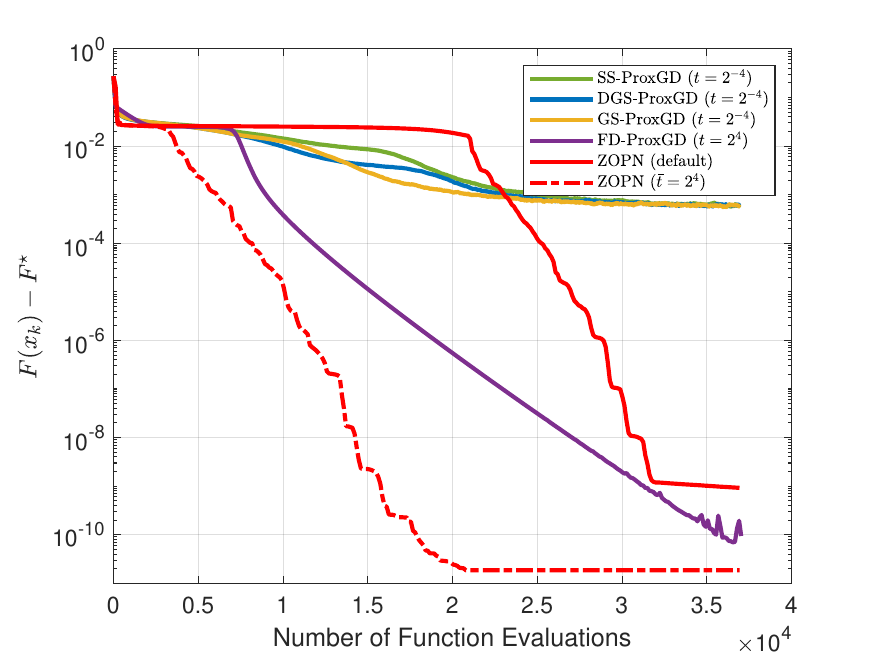}
		\caption{\texttt{a9a}}
		\label{Binarya9a}
	\end{subfigure}
	\centering
	\begin{subfigure}{0.24\linewidth}
		\centering
		\includegraphics[width=1\linewidth]{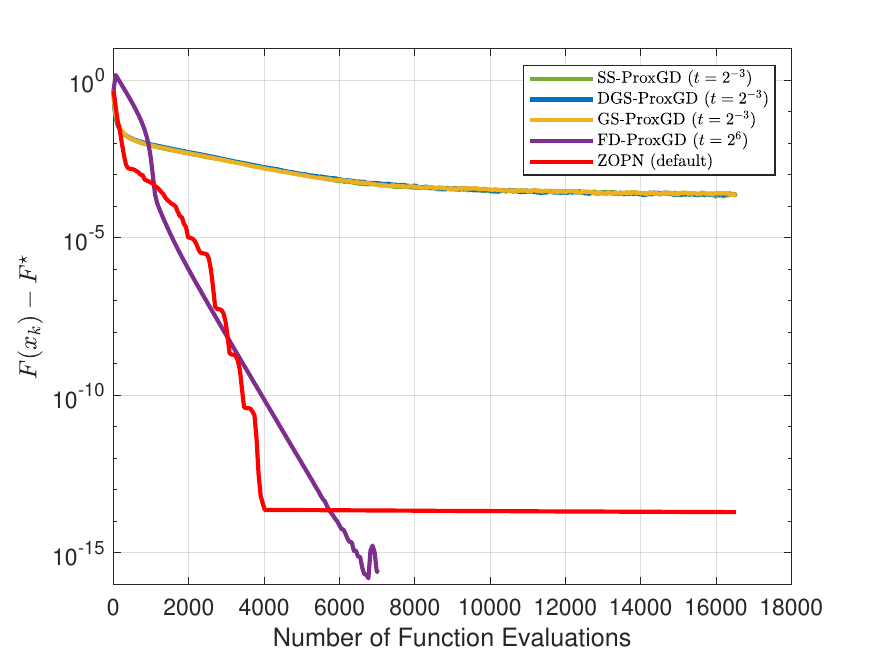}
		\caption{\texttt{covtype}}
		\label{Binarycovtype}
	\end{subfigure}
	
	\centering
	\begin{subfigure}{0.24\linewidth}
		\centering
		\includegraphics[width=1\linewidth]{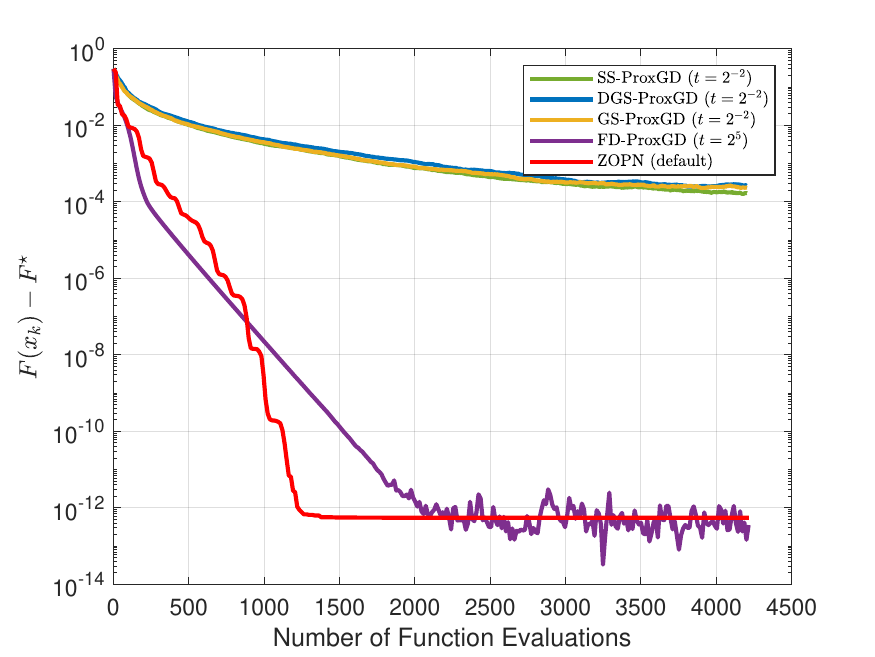}
		\caption{\texttt{heart}}
		\label{Binaryheart}
	\end{subfigure}
	\centering
	\begin{subfigure}{0.24\linewidth}
		\centering
		\includegraphics[width=1\linewidth]{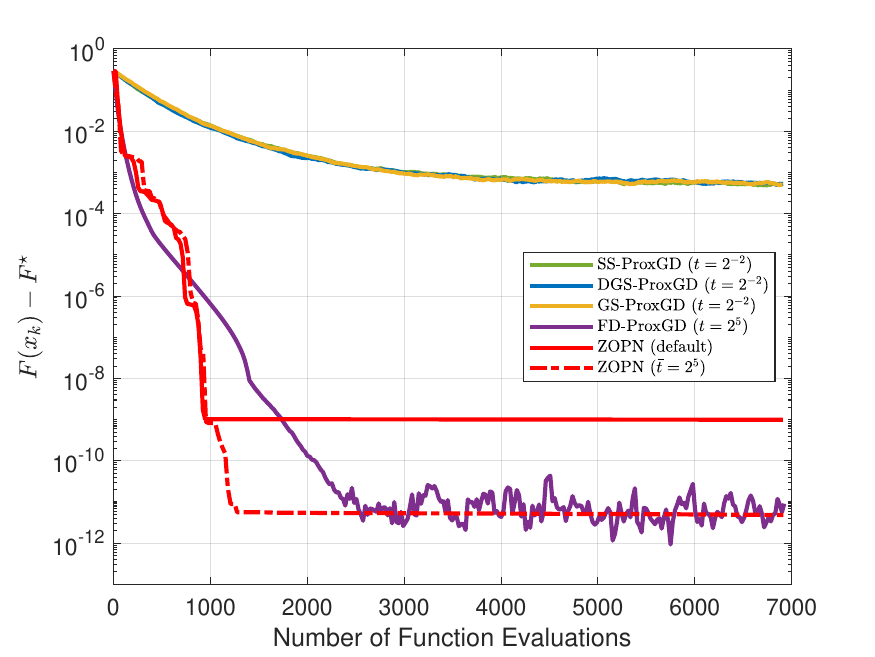}
		\caption{\texttt{ijcnn1}}
		\label{Binaryijcnn}
	\end{subfigure}
	\centering
	\begin{subfigure}{0.24\linewidth}
		\centering
		\includegraphics[width=1\linewidth]{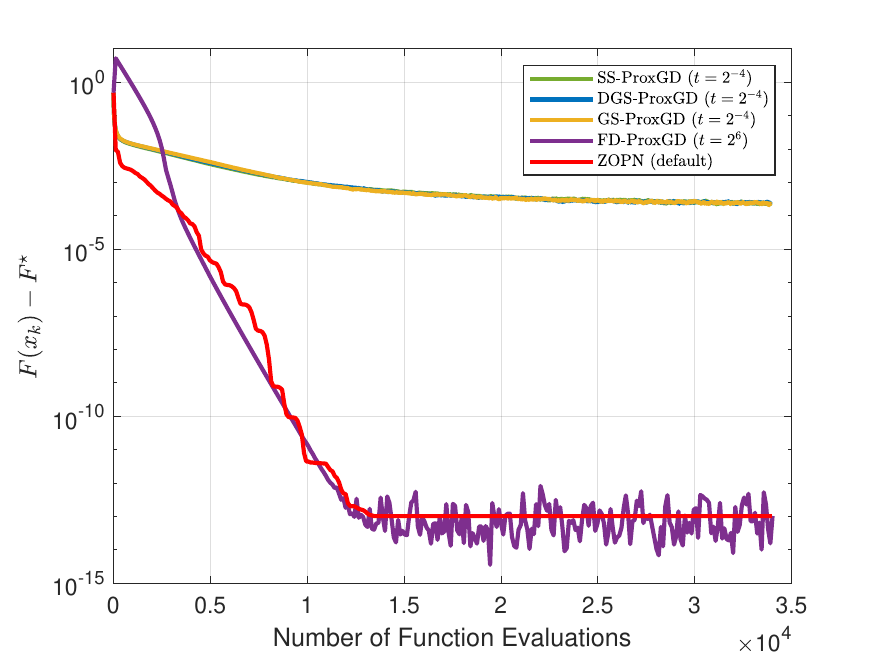}
		\caption{\texttt{mushrooms}}
		\label{Binarymushroom}
	\end{subfigure}
	\centering
	\begin{subfigure}{0.24\linewidth}
		\centering
		\includegraphics[width=1\linewidth]{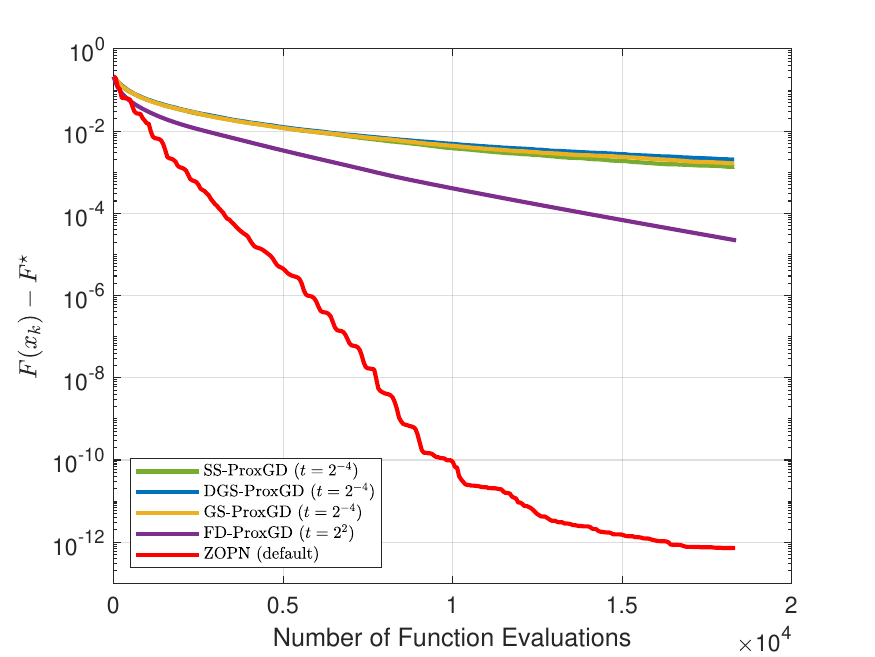}
		\caption{\texttt{sonar}}
		\label{Binarysonar}
	\end{subfigure}
	
	\centering
	\begin{subfigure}{0.24\linewidth}
		\centering
		\includegraphics[width=1\linewidth]{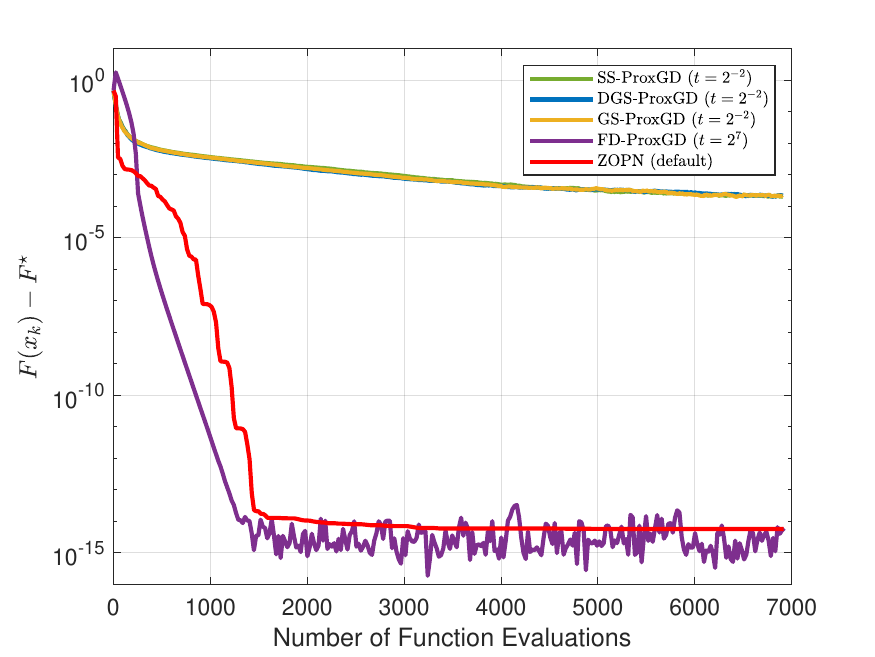}
		\caption{\texttt{svmguide3}}
		\label{Binarysvmguide}
	\end{subfigure}
	\centering
	\begin{subfigure}{0.24\linewidth}
		\centering
		\includegraphics[width=1\linewidth]{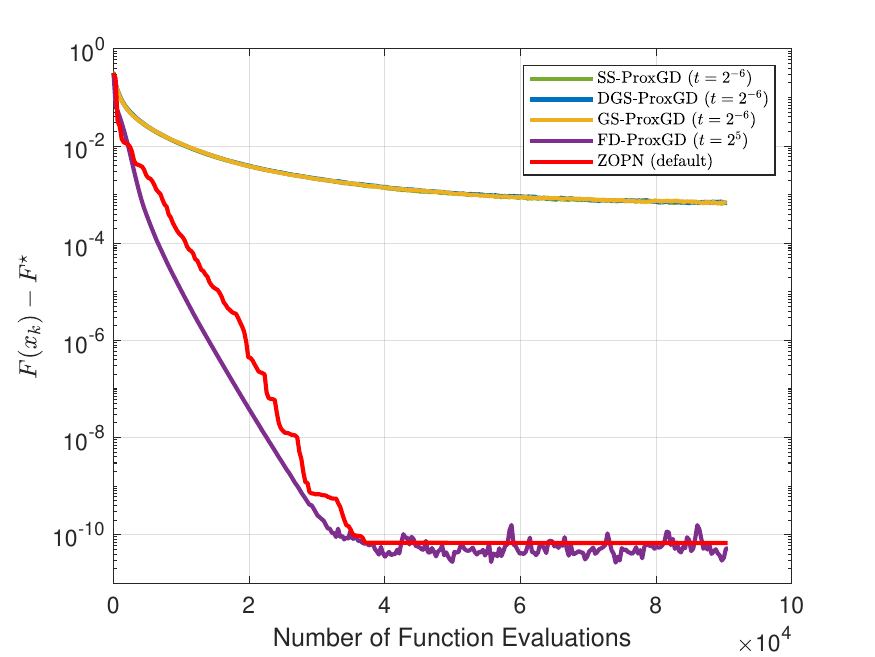}
		\caption{\texttt{w1a}}
		\label{Binaryw1a}
	\end{subfigure}
	\centering
	\begin{subfigure}{0.24\linewidth}
		\centering
		\includegraphics[width=1\linewidth]{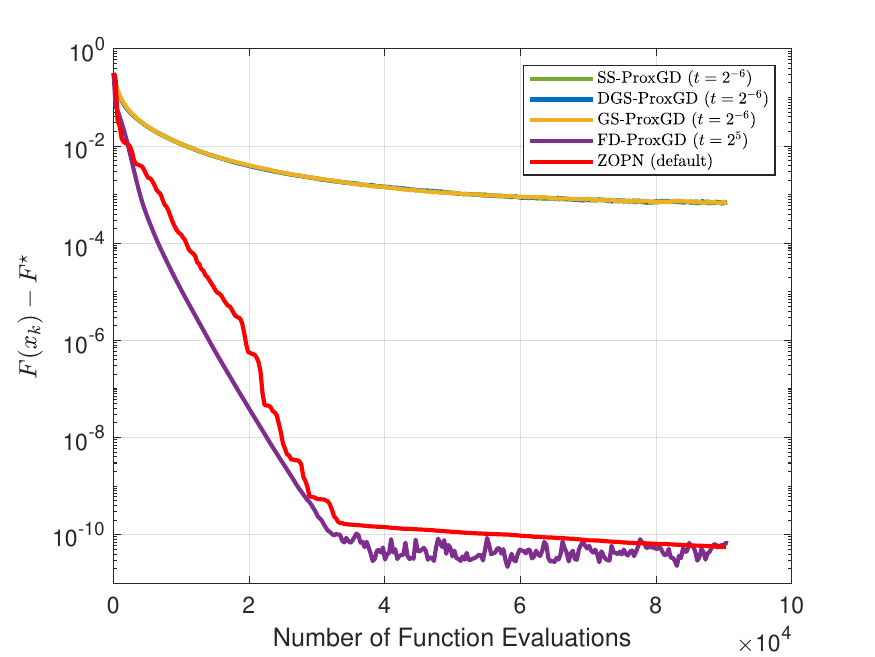}
		\caption{\texttt{w4a}}
		\label{Binaryw4a}
	\end{subfigure}
	\centering
	\begin{subfigure}{0.24\linewidth}
		\centering
		\includegraphics[width=1\linewidth]{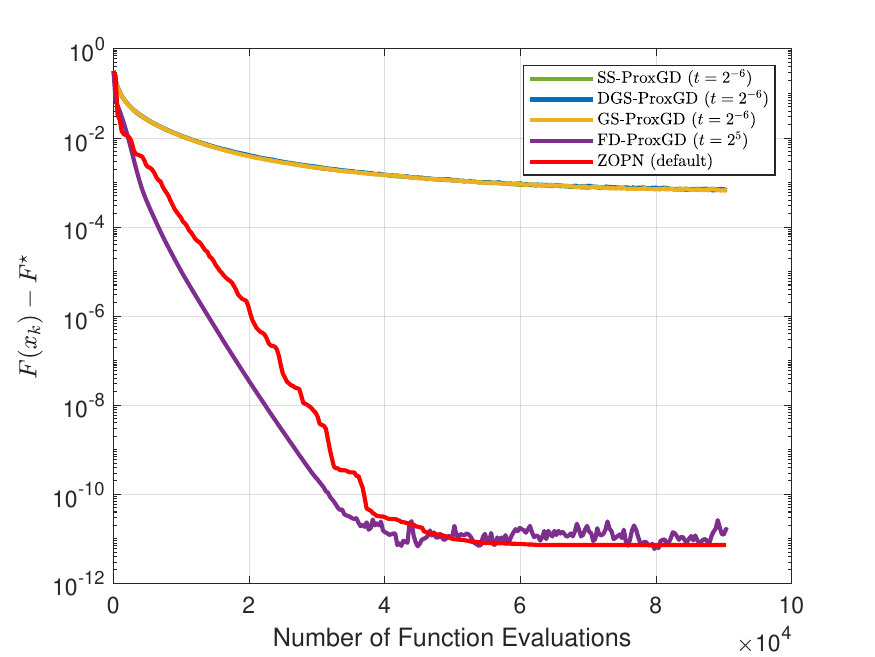}
		\caption{\texttt{w8a}}
		\label{Binaryw8a}
	\end{subfigure}
	\caption{Objective value versus NF on the  elastic net-regularized binary classification problem.}
	\label{fig:BinaryNF}
\end{figure}

\subsection{Nonconvex support vector machine problem}\label{subsec:SVM}
Consider the nonconvex support vector machine (SVM) problem \cite{SPQN2019}
\begin{equation}\label{eq:SVM}
	\min_{x\in \mathbb{R}^n}\frac1p \sum_{i=1}^p  (1 - \tanh (b_i a_i^{\top}x  )  )  + \zeta  \| x \|_1,
\end{equation}
where we set $\zeta=10^{-5}$.

The results are presented in Figure \ref{fig:SVMNF}.
For this nonconvex problem, SS-ProxGD, DGS-ProxGD and GS-ProxGD show slightly different behaviors.
However, ZOPN still performs the best among all algorithms.

\begin{figure}[htbp!]
	\centering
	\begin{subfigure}{0.24\linewidth}
		\centering
		\includegraphics[width=1\linewidth]{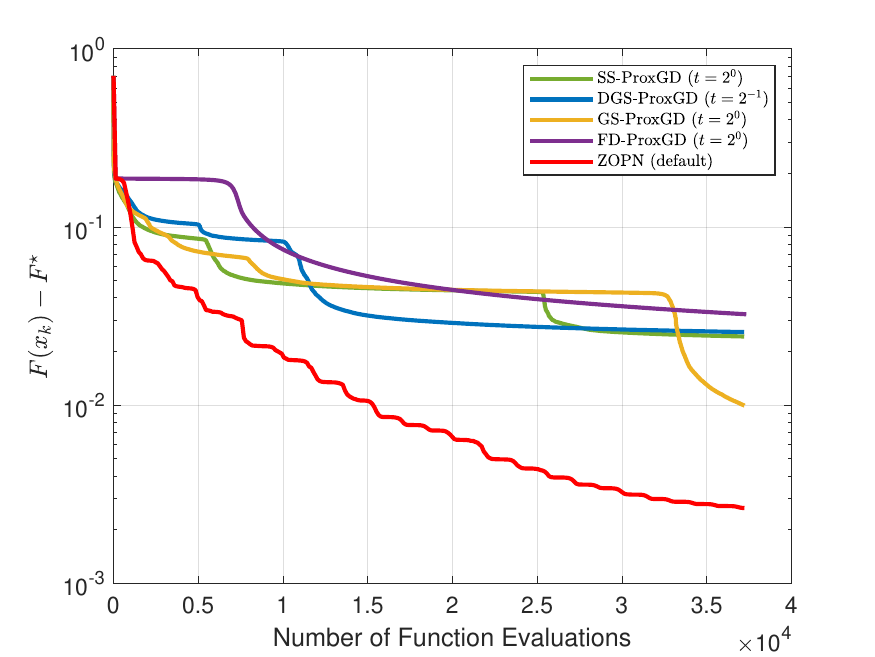}
		\caption{\texttt{a1a}}
		\label{SVMa1a}
	\end{subfigure}
	\centering
	\begin{subfigure}{0.24\linewidth}
		\centering
		\includegraphics[width=1\linewidth]{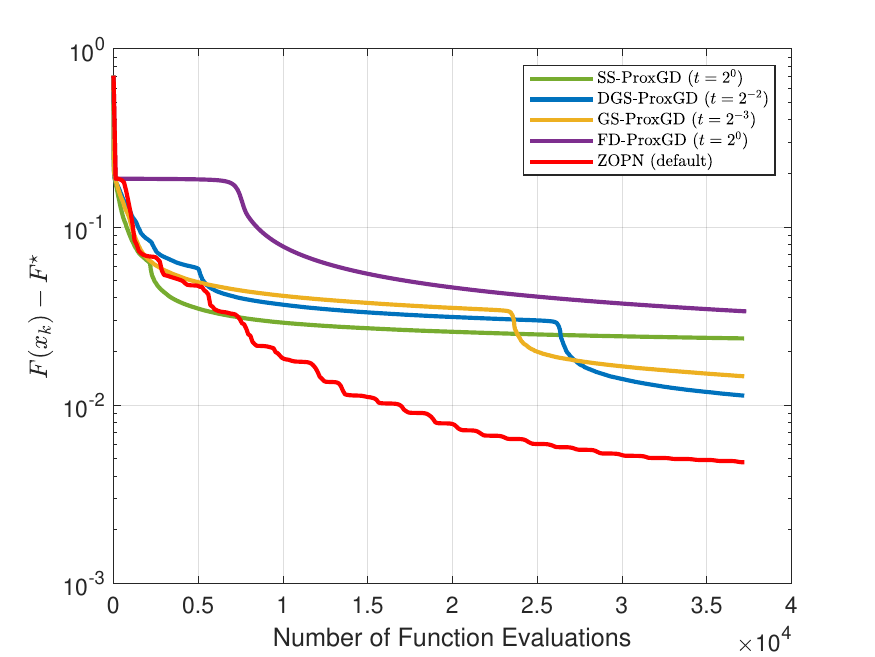}
		\caption{\texttt{a4a}}
		\label{SVMa4a}
	\end{subfigure}
	\centering
	\begin{subfigure}{0.24\linewidth}
		\centering
		\includegraphics[width=1\linewidth]{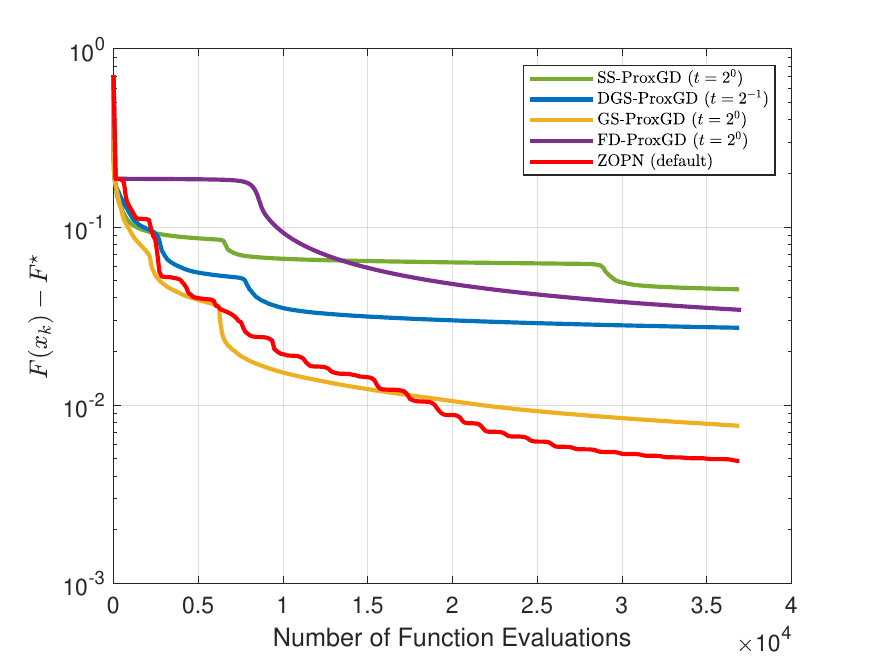}
		\caption{\texttt{a9a}}
		\label{SVMa9a}
	\end{subfigure}
	\centering
	\begin{subfigure}{0.24\linewidth}
		\centering
		\includegraphics[width=1\linewidth]{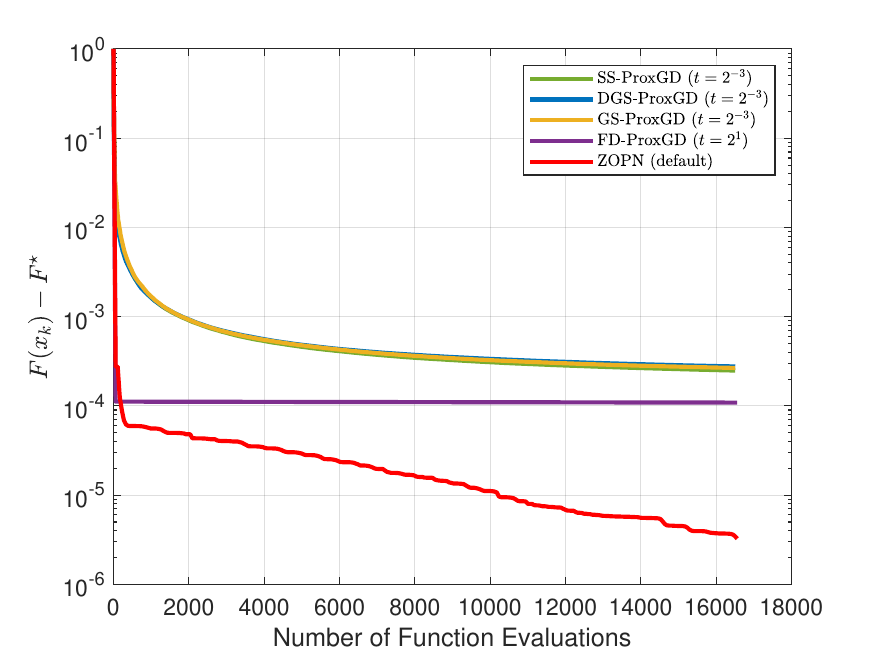}
		\caption{\texttt{covtype}}
		\label{SVMcovtype}
	\end{subfigure}
	
	\centering
	\begin{subfigure}{0.24\linewidth}
		\centering
		\includegraphics[width=1\linewidth]{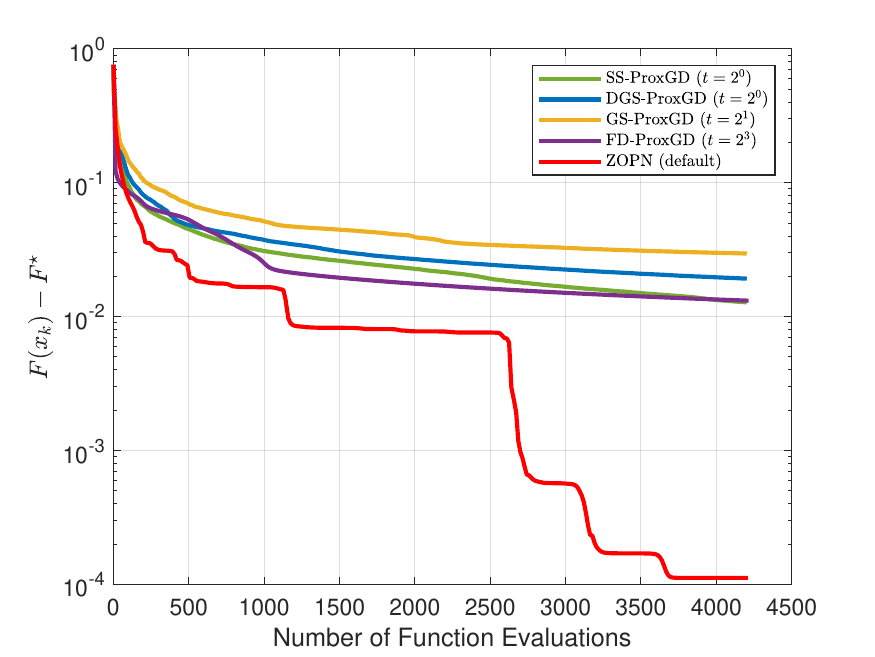}
		\caption{\texttt{heart}}
		\label{SVMheart}
	\end{subfigure}
	\centering
	\begin{subfigure}{0.24\linewidth}
		\centering
		\includegraphics[width=1\linewidth]{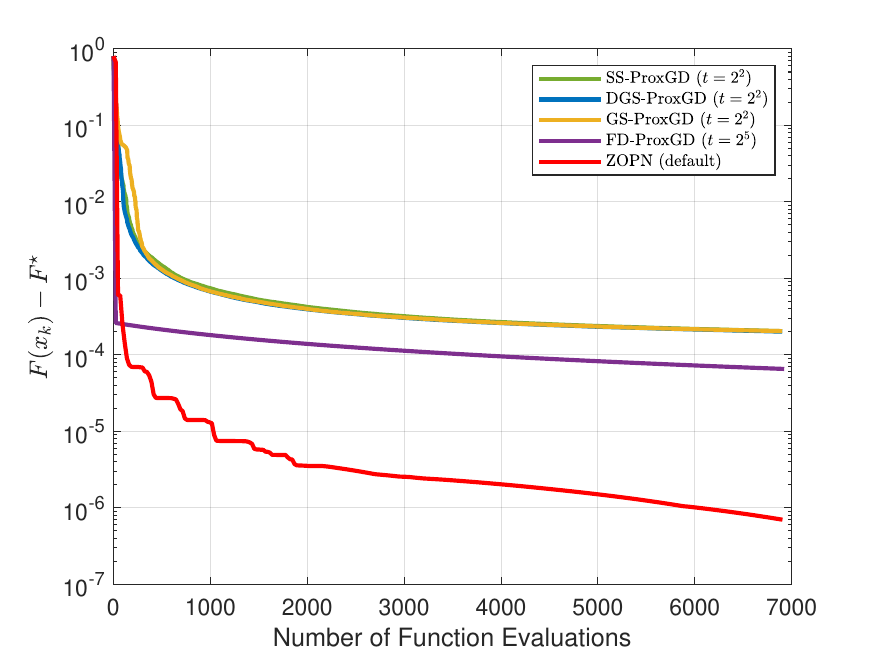}
		\caption{\texttt{ijcnn1}}
		\label{SVMijcnn}
	\end{subfigure}
	\centering
	\begin{subfigure}{0.24\linewidth}
		\centering
		\includegraphics[width=1\linewidth]{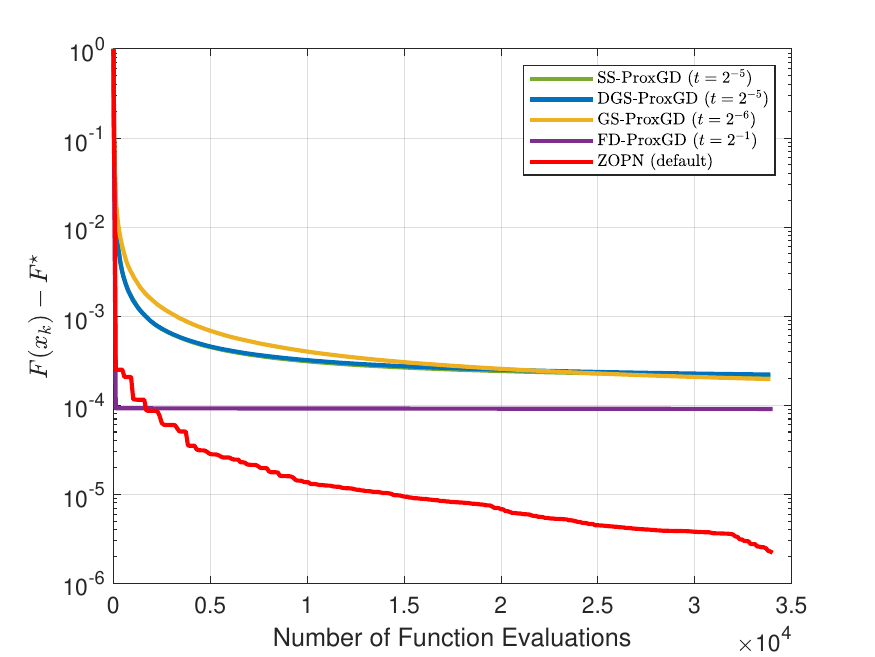}
		\caption{\texttt{mushrooms}}
		\label{SVMmushroom}
	\end{subfigure}
	\centering
	\begin{subfigure}{0.24\linewidth}
		\centering
		\includegraphics[width=1\linewidth]{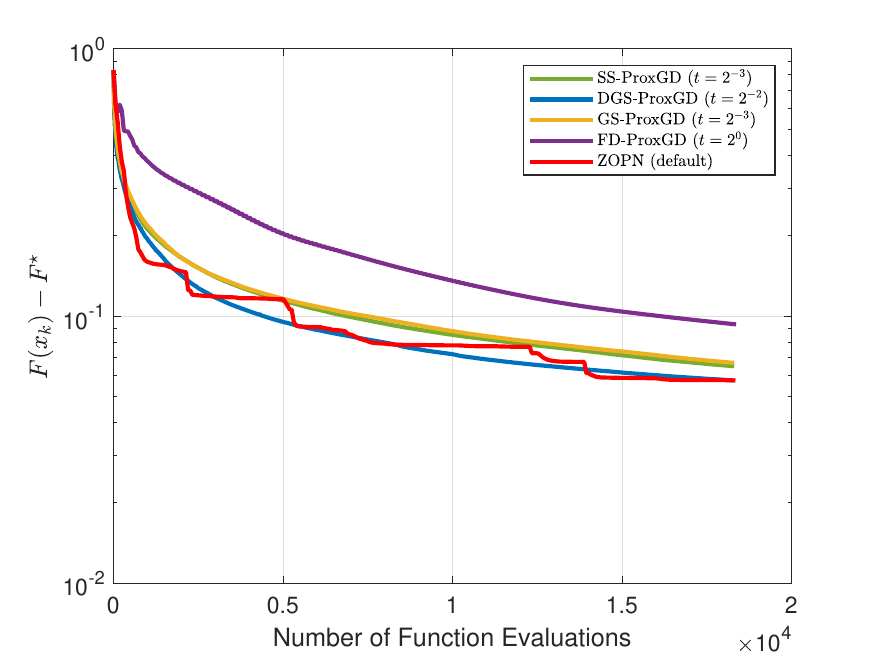}
		\caption{\texttt{sonar}}
		\label{SVMsonar}
	\end{subfigure}
	
	\centering
	\begin{subfigure}{0.24\linewidth}
		\centering
		\includegraphics[width=1\linewidth]{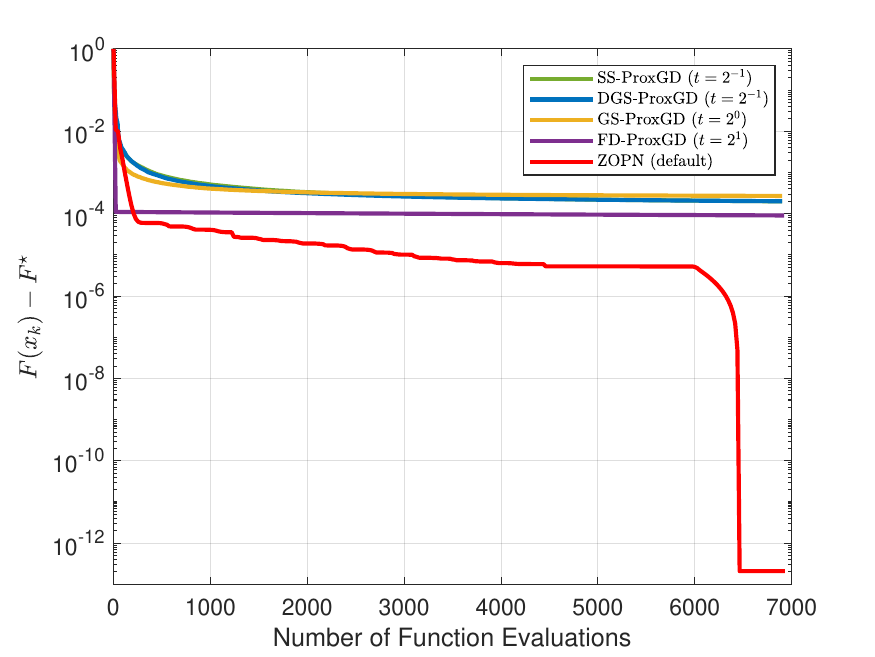}
		\caption{\texttt{svmguide3}}
		\label{SVMsvmguide}
	\end{subfigure}
	\centering
	\begin{subfigure}{0.24\linewidth}
		\centering
		\includegraphics[width=1\linewidth]{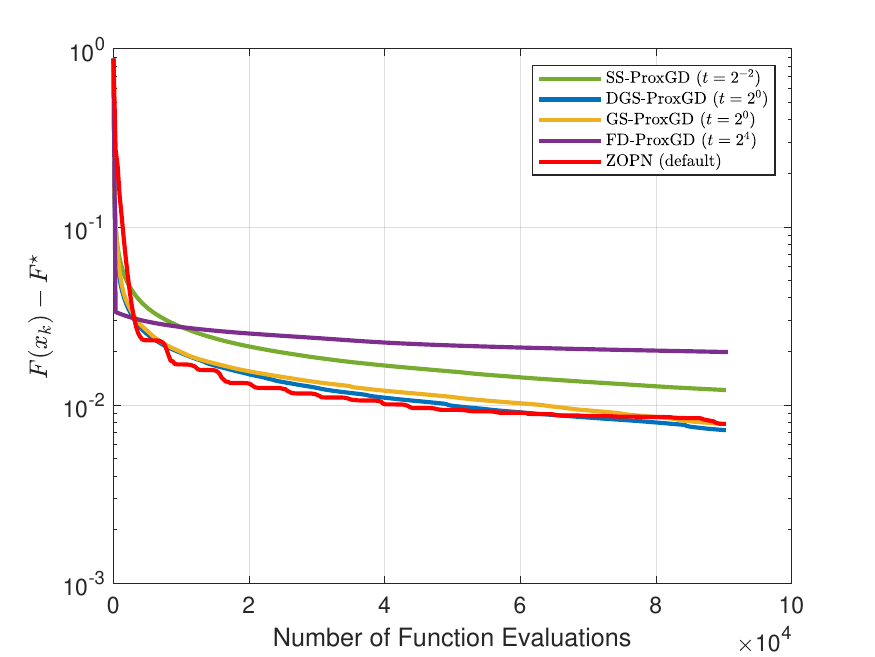}
		\caption{\texttt{w1a}}
		\label{SVMw1a}
	\end{subfigure}
	\centering
	\begin{subfigure}{0.24\linewidth}
		\centering
		\includegraphics[width=1\linewidth]{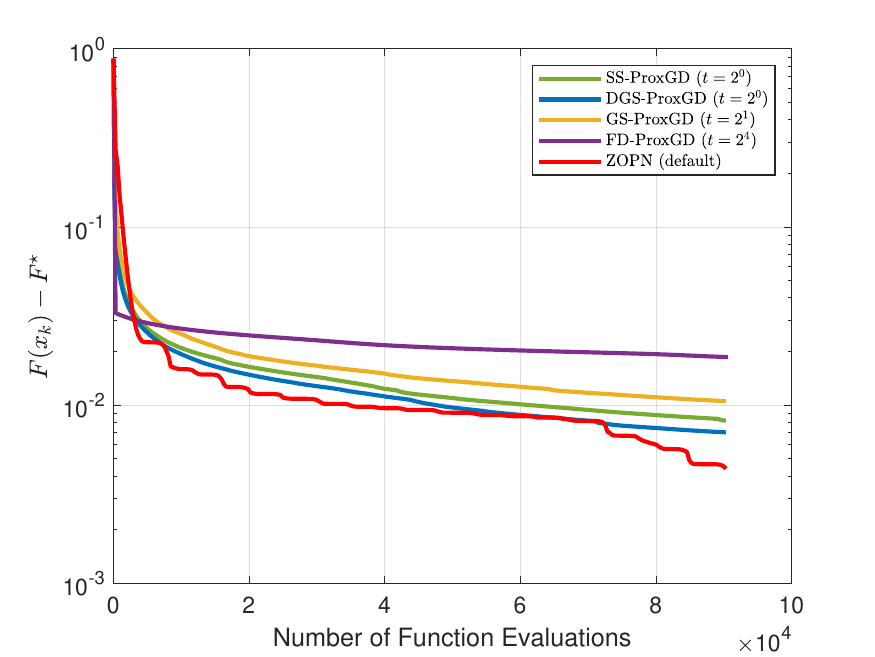}
		\caption{\texttt{w4a}}
		\label{SVMw4a}
	\end{subfigure}
	\centering
	\begin{subfigure}{0.24\linewidth}
		\centering
		\includegraphics[width=1\linewidth]{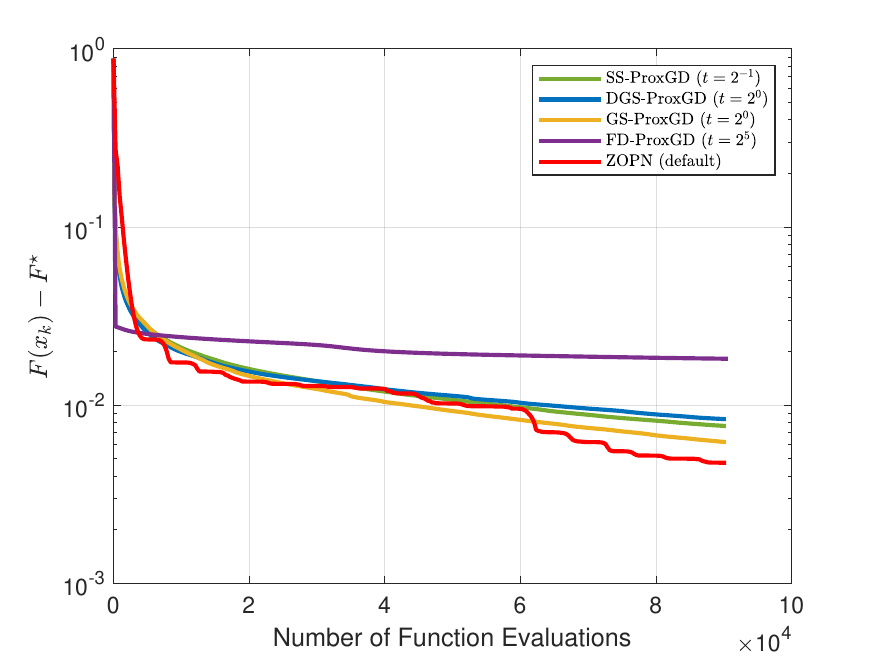}
		\caption{\texttt{w8a}}
		\label{SVMw8a}
	\end{subfigure}
	\caption{Objective value versus NF on the  nonconvex SVM problem.}
	\label{fig:SVMNF}
\end{figure}

\section{Discussions and conclusions}\label{sec:conclusion}

\subsection{Finite difference is preferable to smoothing techniques in gradient estimation for BFGS updating}\label{subsec:FDvsGS}
It is an interesting open problem to identify the suitable gradient estimations for quasi-Newton approaches as Bollapragada and Wild \cite{ZOQN2023} and Bollapragada et al. \cite{Bollapragada2024} pointed out.
In the following, we show why the finite difference seems more preferable than the smoothing techniques
in gradient estimations for BFGS method when the gradient of a function is not available.
All the results about finite difference in this section can be easily extended to any deterministic gradient estimations \eqref{eq:gradienterror}.

Note that if the exact gradient of $f$ is available and the BFGS updating $B_{k-1}$ is positive definite,
then $B_k$ is positive definite if and only if $y_{k-1}^{\top}s_{k-1}>0$,
where $y_{k-1}=\nabla f(x_k)-\nabla f(x_{k-1})$ and $s_{k-1}=x_k-x_{k-1}$.
From this perspective, we show that the forward difference-based curvature condition
can be satisfied if the smoothing parameter is chosen appropriately.

\begin{theorem}\label{thm:FDcurv}
	Suppose that Assumption~\ref{as:problem1} holds and $f$ is $\mu$-strongly convex.
	If $\Delta_k < \frac{\mu}{\sqrt{n}L_f}\min\{\|s_{k-1}\|, \|s_k \|  \}$,
	then we have $( y^{\mathrm{FD}}_{k-1}) ^{\top}s_{k-1}>0$.
\end{theorem}

\begin{proof}
	By Assumption~\ref{as:ZOPN-g} with $\kappa_{\mathrm{eg}}=\sqrt{n}L_f/2$ and the $\mu$-strong convexity of $f$,
	\begin{equation*}
		\begin{aligned}
			( y^{\mathrm{FD}}_{k-1}) ^{\top}s_{k-1} &=  \left( g^{\mathrm{FD}}_k  -g^{\mathrm{FD}}_{k-1}  \right) ^{\top}s_{k-1} \\
			&=  ( \nabla f(x_k)-\nabla f(x_{k-1}) ) ^{\top}s_{k-1}  + \left(g^{\mathrm{FD}}_k  -\nabla f(x_k)+\nabla f(x_{k-1})- g^{\mathrm{FD}}_{k-1}  \right) ^{\top}s_{k-1} \\
			&\ge \mu \|s_{k-1} \|^2   - \Big( \left\|g^{\mathrm{FD}}_k  -\nabla f(x_k) \right\|  + \left\|\nabla f(x_{k-1})- g^{\mathrm{FD}}_{k-1}   \right\| \Big)  \|s_{k-1} \| \\
			&\ge  \|s_{k-1} \|  \left(\mu \|s_{k-1} \| - \frac{\sqrt{n}L_f}2 (\Delta_k + \Delta_{k-1}) \right ) \\
			&> \|s_{k-1}\|\left( \mu\|s_{k-1}\| - \frac{\sqrt{n}L_f}2 \frac{2\mu}{\sqrt{n}L_f}\|s_{k-1}\| \right)=0.
		\end{aligned}
	\end{equation*}
	This completes the proof.
\end{proof}

Theorem \ref{thm:FDcurv} is similar to the result given in \cite[Section 2.4]{ZOQN2023}.
It gives an upper bound on the sampling radius such that the forward difference-based curvature condition can be satisfied.

Next we discuss the curvature conditions based on Gaussian smoothing and spherical smoothing.
We first give two auxiliary results derived from \cite[Theorems 2.6 and 2.11]{GE2022}.

\begin{lemma}\label{lem:GSerror}
	Suppose that Assumption \ref{as:problem1} holds. Consider the Gaussian smoothing
	\begin{equation}\label{equ:GS}
		g^{\mathrm{GS}} := g^{\mathrm{GS}}_{\Delta}(x) =  \frac{1}{N}\sum_{i=1}^{N}\frac{f(x+\Delta u_i)-f(x)}{\Delta}u_i,
	\end{equation}
	where $u_i\sim \mathcal{N}(0,I)$. For any $x\in \mathbb{R}^n$ and $w>0$, if
	\begin{equation}\label{equ:NFGS}
		N \ge \frac{3n}{\nu w^2}\left[ 3\|\nabla f(x)\|^2+\frac{L_f^2\Delta^2}{4}(n+2)(n+4)\right],
	\end{equation}
	then
	\begin{equation*}
		\left\| g^{\mathrm{GS}} - \nabla f(x) \right\| \le \sqrt{n}L_f \Delta + w
	\end{equation*}
	holds with probability at least $1-\nu$.
\end{lemma}

Lemma~\ref{lem:GSerror} indicates that the difference between the exact gradient and the approximate gradient via  Gaussian smoothing is bounded by a linear combination of the sampling radius and the error due to the sample average approximation with high probability if the number of samples satisfies~\eqref{equ:NFGS}.
A similar result can also be obtained when the gradient is approximated via the spherical smoothing.

\begin{lemma}\label{lem:SSerror}
	Suppose that Assumption~\ref{as:problem1} holds. Consider the spherical smoothing
	\begin{equation}\label{equ:SS}
		g^{\mathrm{SS}}:= g^{\mathrm{SS}}_{\Delta}(x) =
		\frac{n}{N}\sum_{i=1}^{N}\frac{f(x+\Delta u_i)-f(x)}{\Delta}u_i,
	\end{equation}
	where $u_i\sim \mathcal{U}(\mathcal{S}(0,1))$.  For any $x\in \mathbb{R}^n$ and $w>0$, if
	\begin{equation}\label{equ:NFSS}
		N \ge \left[\frac{6n^2}{w^2}\left(\frac{\|\nabla f(x)\|^2}n+\frac{L_f^2\Delta^2}4\right)+\frac{2n}{3w} (2\|\nabla f(x)\|+L_f\Delta )\right]\log\frac{n+1}\nu,
	\end{equation}
	then
	\begin{equation*}
		\left\| g^{\mathrm{SS}}  - \nabla f(x) \right\| \le L_f \Delta + w
	\end{equation*}
	holds with probability at least $1-\nu$.
\end{lemma}

Based on Lemmas~\ref{lem:GSerror} and \ref{lem:SSerror},
we show that the curvature conditions based on Gaussian smoothing and spherical smoothing can be satisfied
with high probability under certain conditions.

Let $\theta \in  ( 0,1/2 )$ and $\lambda \in (0,1)$ close to $0$  be constants.
Denote
\begin{equation*}
	g^{\mathrm{GS}}_k:=\frac{1}{N_k}\sum_{i=1}^{N_k}\frac{f(x_k+\Delta_k  u_i)-f(x_k)}{\Delta_k}u_i,
\end{equation*}
where $u_i\sim \mathcal{N}(0,I)$, and
$	y^{\mathrm{GS}}_{k-1}:=g^{\mathrm{GS}}_k  - g^{\mathrm{GS}}_{k-1} $.

\begin{theorem}\label{thm:GScurv}
	Suppose that Assumption~\ref{as:problem1} holds and $f$ is $\mu$-strongly convex.
	If we take
	\begin{equation*}
		\begin{aligned}
			\Delta_k &= \frac{\theta \mu }{\sqrt{n}L_f} \min\{\|s_{k-1}\|, \|s_k \|  \}, \nonumber \\
			N_k & \ge \frac{9n}{\nu \mu^2(1-\lambda)^2\left(\frac{1}{2}-\theta \right)^2} \left( \frac{\|\nabla f(x_k)\|}{\min\{\|s_{k-1}\|, \|s_k \|  \}} \right)^2 + \frac{3\theta^2 (n+2)(n+4)}{4\nu (1-\lambda)^2\left(\frac12 - \theta \right)^2 } 
		\end{aligned}
	\end{equation*}
	for all $k$,
	then
	$( y^{\mathrm{GS}}_{k-1}) ^{\top}s_{k-1}>0$ holds 	with probability at least $1-\nu$.
\end{theorem}

\begin{proof}
	Take
	$
	w_k=(1-\lambda)  (1/2-\theta  )\mu \min\{\|s_{k-1}\|, \|s_k \|  \}
	$
	for all $k$.
	By Lemma \ref{lem:GSerror}, if
	\begin{equation*}
		\begin{aligned}
			N_k &\ge \frac{3n}{\nu w_k^2}\left[ 3\|\nabla f(x_k)\|^2+\frac{L_f^2\Delta_k^2}{4}(n+2)(n+4)\right]  \\
			& = \frac{9n}{\nu \mu^2(1-\lambda)^2\left(\frac{1}{2}-\theta \right)^2} \left( \frac{\|\nabla f(x_k)\|}{\min\{\|s_{k-1}\|, \|s_k \|  \}} \right)^2  + \frac{3\theta^2 (n+2)(n+4)}{4\nu (1-\lambda)^2\left(\frac12 - \theta \right)^2 },
		\end{aligned}
	\end{equation*}
	then
	\begin{equation}\label{equ:GSerror}
		\left\| g^{\mathrm{GS}}_k - \nabla f(x_k) \right\| \le \sqrt{n}L_f \Delta_k + w_k
	\end{equation}
	holds with probability at least $1-\nu$.
	
	By the $\mu$-strong convexity of $f$ and \eqref{equ:GSerror},
	\begin{equation*}
		\begin{aligned}
			( y^{\mathrm{GS}}_{k-1}) ^{\top}s_{k-1} &= \left( g^{\mathrm{GS}}_k -g^{\mathrm{GS}}_{k-1} \right) ^{\top}s_{k-1} \\
			&=  ( \nabla f(x_k)-\nabla f(x_{k-1}) ) ^{\top}s_{k-1}  + \left(g^{\mathrm{GS}}_k -\nabla f(x_k)+\nabla f(x_{k-1})- g^{\mathrm{GS}}_{k-1} \right) ^{\top}s_{k-1} \\
			&\ge \mu \|s_{k-1} \|^2  - \Big( \left\|g^{\mathrm{GS}}_k-\nabla f(x_k) \right\|  +\left\|\nabla f(x_{k-1})- g^{\mathrm{GS}}_{k-1} \right\| \Big)  \|s_{k-1}  \| \\
			&\ge  \|s_{k-1} \|  ( \mu \|s_{k-1}  \| - \sqrt{n}L_f(\Delta_k+\Delta_{k-1}) - (w_k+w_{k-1}) )  \\
			& \ge \|s_{k-1} \| \Big( \mu  \|s_{k-1} \| - \sqrt{n}L_f \frac{2\theta \mu }{\sqrt{n}L_f} \|s_{k-1} \|   - 2 (1-\lambda) (1/2-\theta  )\mu  \|s_{k-1} \|\Big) \\
			& = \lambda (1-2\theta) \mu  \|s_{k-1}  \|^2 > 0.
		\end{aligned}
	\end{equation*}
	This completes the proof.
\end{proof}

Denote
\begin{equation*}
	g^{\mathrm{SS}}_k :=\frac{n}{N_k}\sum_{i=1}^{N_k}\frac{f(x_k+\Delta_k  u_i)-f(x_k)}{\Delta_k}u_i,
\end{equation*}
where $u_i\sim \mathcal{U}(\mathcal{S}(0,I))$, and
$	y^{\mathrm{SS}}_{k-1}:= g^{\mathrm{SS}}_k - g^{\mathrm{SS}}_{ k-1} $.

\begin{theorem}\label{thm:SScurv}
	Suppose that Assumption~\ref{as:problem1} holds and $f$ is $\mu$-strongly convex. If we take
	\begin{equation*}
		\begin{aligned}
			\Delta_k &= \frac{\theta \mu}{\sqrt{n}L_f}\min\{\|s_{k-1}\|, \|s_k \|  \}, \nonumber \\
			N_k & \ge \left[\frac{6n}{\mu^2 (1-\lambda)^2\left(\frac{1}{2}-\frac{\theta}{\sqrt{n}} \right)^2} \left( \frac{\|\nabla f(x_k)\|}{\min\{\|s_{k-1}\|, \|s_k \|  \}}\right)^2 \right.\nonumber \\
			& \hspace{2em} + \frac{4n}{3\mu (1-\lambda)\left(\frac{1}{2}-\frac{\theta}{\sqrt{n}} \right)} \frac{\|\nabla f(x_k)\|}{\min\{\|s_{k-1}\|, \|s_k \|  \}}   \nonumber \\
			& \left. \hspace{2em} +
			\frac{2\theta \sqrt{n}}{3(1-\lambda)\left(\frac{1}{2}-\frac{\theta}{\sqrt{n}} \right)^2} +
			\frac{3\theta^2 n}{2(1-\lambda)^2\left(\frac{1}{2}-\frac{\theta}{\sqrt{n}} \right)^2}\right]\log\frac{n+1}\nu 
		\end{aligned}
	\end{equation*}
	for all $k$, then
	$( y^{\mathrm{SS}}_{k-1}) ^{\top}s_{k-1}>0$ holds
	with probability at least $1-\nu$.
\end{theorem}

\begin{proof}
	Take
	$ w_k= (1-\lambda)\left(\frac{1}{2}-\frac{\theta}{\sqrt{n}} \right)\mu \min\{\|s_{k-1}\|, \|s_k \|  \}$
	for all $k$. By Lemma~\ref{lem:SSerror}, if
	\begin{equation*}
		\begin{aligned}
			N_k & \ge  \left[\frac{6n^2}{w_k^2}\left(\frac{\|\nabla f(x_k)\|^2}n+\frac{L_f^2\Delta_k^2}4\right)+\frac{2n}{3w_k} (2\|\nabla f(x_k)\|+L_f\Delta_k )\right]\log\frac{n+1}\nu \\
			&= \left[\frac{6n}{\mu^2 (1-\lambda)^2\left(\frac{1}{2}-\frac{\theta}{\sqrt{n}} \right)^2} \left( \frac{\|\nabla f(x_k)\|}{\min\{\|s_{k-1}\|, \|s_k \|  \}}\right)^2 \right.\nonumber \\
			& \hspace{2em} + \frac{4n}{3\mu (1-\lambda)\left(\frac{1}{2}-\frac{\theta}{\sqrt{n}} \right)} \frac{\|\nabla f(x_k)\|}{\min\{\|s_{k-1}\|, \|s_k \|  \}}   \nonumber \\
			& \left. \hspace{2em} +
			\frac{2\theta \sqrt{n}}{3(1-\lambda)\left(\frac{1}{2}-\frac{\theta}{\sqrt{n}} \right)^2} +
			\frac{3\theta^2 n}{2(1-\lambda)^2\left(\frac{1}{2}-\frac{\theta}{\sqrt{n}} \right)^2}\right]\log\frac{n+1}\nu,
		\end{aligned}
	\end{equation*}
	then
	\begin{equation}\label{equ:SSerror}
		\left\| g^{\mathrm{SS}}_k - \nabla f(x_k) \right\| \le L_f \Delta_k + w_k
	\end{equation}
	holds with probability at least $1-\nu$.
	
	By the $\mu$-strong convexity of $f$ and \eqref{equ:SSerror},
	\begin{equation*}
		\begin{aligned}
			( y_{k-1}^{\mathrm{SS}}) ^{\top}s_{k-1}
			&\ge \mu \|s_{k-1}  \|^2  - \Big( \left\| g^{\mathrm{SS}}_k -\nabla f(x_k) \right\|  +\left\|\nabla f(x_{k-1})- g^{\mathrm{SS}}_{ k-1} \right\| \Big)  \|s_{k-1}  \| \\
			& \ge  \|s_{k-1} \| ( \mu \|s_{k-1}  \| - L_f(\Delta_k + \Delta_{k-1}) - (w_k+w_{k-1})  )  \\
			& \ge \|s_{k-1} \| \bigg( \mu \|s_{k-1}  \| - L_f \frac{2\theta \mu }{\sqrt{n}L_f} \|s_{k-1} \|  - 2(1-\lambda) \left(\frac{1}{2}-\frac{\theta}{\sqrt{n}} \right)\mu  \|s_{k-1} \|\bigg) \\
			& = \lambda \left(1-\frac{2\theta}{\sqrt{n}} \right)\mu  \|s_{k-1} \|^2 > 0.
		\end{aligned}
	\end{equation*}
	This completes the proof.
\end{proof}

It is shown in Theorems~\ref{thm:GScurv} and \ref{thm:SScurv} that the curvature conditions based on Gaussian smoothing and spherical smoothing hold with probability $1-\nu$ if the upper bounds on the sampling numbers are of order larger than $\varOmega(n^2/\nu)$ and $\varOmega(n\log(n/\nu))$, respectively,
which are bigger than the fixed $n$ samples for the forward difference-based curvature condition (cf. Theorem~\ref{thm:FDcurv}).
This suggests that the finite difference may be more suitable than smoothing techniques for the BFGS update.
However, it is still a challenging problem to provide lower bounds on the sampling numbers for the Gaussian smoothing and spherical smoothing such that the curvature conditions are satisfied.

\subsection{Conclusions}
In this paper, we develop a unified derivative-free proximal Newton-type framework for composite optimization that accommodates a broad class of gradient and Hessian estimation methods. Notably, this is the first work formalizing proximal Newton-type algorithms under the full zeroth-order setting, whereas existing studies mainly focus on first-order schemes or adopt cubic/regularized cubic Newton structures rather than the proximal Newton-type formulation.
We derive the iteration and oracle complexity bounds for the algorithm to attain an $\epsilon$-optimal solution under both nonconvex and strongly convex settings.
We also establish its local R-superlinear convergence under standard assumptions, a desirable theoretical outcome that is uncommon in derivative-free optimization. Unlike existing superlinear convergence results relying on lazy Hessian updates and full finite-difference approximation, our analysis is built upon the Dennis--Mor\'{e} condition and applies to a much wider range of Hessian estimation strategies.
Furthermore, we theoretically verify for the first time that BFGS is more compatible with deterministic finite-difference gradient estimators than with smoothing-based counterparts such as Gaussian and spherical smoothing schemes, providing a theoretical justification for the widely observed empirical choice in existing numerical implementations.
Numerical experiments are also presented to demonstrate the efficiency of the proposed algorithm.

	\bibliographystyle{amsplain}

\end{document}